\documentclass[a4paper]{article}

\usepackage[T1]{fontenc}
\usepackage[utf8]{inputenc}
\usepackage{xcolor,amsmath,amssymb,amsthm,mathtools,hyperref,subcaption}
\hypersetup{colorlinks=true,allcolors=blue}
\usepackage[capitalize,nameinlink,noabbrev]{cleveref}
\usepackage{lmodern}
\usepackage[english]{babel}
\usepackage[all,cmtip]{xy}
\usepackage{tkz-fct}
\usepackage{tikz}
\usepackage{bm}
\usepackage{comment,graphicx,soul}
\usepackage{amsfonts}
\usepackage{multicol}

\theoremstyle{definition} 

 \newtheorem{definition}{Definition}[section]
 \newtheorem{remark}[definition]{Remark}
 \newtheorem{example}[definition]{Example}

\theoremstyle{plain}      

 \newtheorem{proposition}[definition]{Proposition}
 \newtheorem{theorem}[definition]{Theorem}
 \newtheorem{corollary}[definition]{Corollary}
 \newtheorem{lemma}[definition]{Lemma}
 
 \newtheorem*{theorem*}{Theorem}
  \newtheorem*{proposition*}{Proposition}

\newcommand{\Z}{\mathbb Z}

\newcommand{\R}{\mathbb R}
\newcommand{\C}{\mathbb C}
\newcommand{\CP}{\mathbb {CP}}
\newcommand{\CPstar}{\mathbb {CP}^{2*}}

\newcommand{\RP}{\mathbb {RP}}

\newcommand{\A}{{\mathbb A}}
\newcommand{\la}{\langle}
\newcommand{\ra}{\rangle}
\renewcommand{\P}{\mathbb {P}}

\renewcommand{\S}{\mathbb S}

\newcommand{\F}{\mathcal F}

\newcommand{\bA}{\bm A}
\newcommand{\ba}{\bm a}
\newcommand{\bb}{\bm b}
\newcommand{\bc}{\bm c}
\newcommand{\bd}{\bm d}
\newcommand{\bp}{\bm p}
\newcommand{\bo}{\bm o}
\newcommand{\bq}{\bm q}
\newcommand{\br}{\bm r}

\newcommand{\bmm}{\bm m}

\renewcommand{\phi}{\varphi}

\newcommand{\HdC}{{\mathbf H}_\C^2}
\newcommand{\HuuC}{{\mathbf H}_\C^{1,1}}
\newcommand{\HdR}{{\mathbf H}_\R^2}
\newcommand{\HuuR}{{\mathbf H}_\R^{1,1}}

\newcommand{\bHdC}{\partial{\mathbf H}_\C^2}
\newcommand{\bHdR}{\partial{\mathbf H}_\R^2}

\newcommand{\unit}[1]{\mathrm{UT}#1}

\newcommand{\arc}[2]{{#1}\curvearrowright{#2}}

\renewcommand{\Re}{{\rm Re}}
\renewcommand{\Im}{{\rm Im}}

\newcommand{\PGL}{\mathrm{PGL}}

\newcommand{\PU}{\mathrm{PU}}
\newcommand{\PO}{\mathrm{PO}}

\newcommand{\SU}{\mathrm{SU}}

\author{E. Falbel, A. Guilloux, P. Will \footnote{P.Will is partially supported by the French National Research Agency in the framework of the « Investissements d'avenir » program « ANR-15-IDEX-02» and the LabEx PERSYVAL « ANR-11-LABX-0025-01»}.}

\title{Slim curves, limit sets and spherical CR uniformisations}

\begin{document}

\maketitle

\begin{abstract}
We consider here the $3$-sphere $\S^3$ seen as the boundary at infinity of the complex hyperbolic plane $\HdC$. It comes equipped with a contact structure and two classes of special curves. First $\R$-circles are boundaries at infinity of  totally real totally geodesic subspaces and are tangent to the contact distribution. Second, $\C$-circles, which  are boundaries of complex totally geodesic subspaces and are transverse to  the contact distribution. 

We define a quantitative notion, called slimness, that measures to what extent a continuous path in the sphere $\S^3$ is near to be an $\R$-circle.  We analyze the classical foliation  of the complement of an $\R$-circle by arcs of $\C$-circles. Next, we consider deformations of this situation where the $\R$-circle becomes a slim curve.  We apply these concepts to the particular case where the slim curve is the limit set of a quasi-Fuchsian subgroup of $\PU(2,1)$. As a consequence, we describe a class of spherical CR uniformizations of certain cusped $3$-manifolds.
\end{abstract}


\section{Introduction}

The frame of this work is the study of quasi-Fuchsian deformations in complex hyperbolic space $\HdC$, which can be thought of as the unit ball in $\C^2$. Using a projective model, the isometry group of $\HdC$ can be identified with ${\rm PU(2,1)}$, the subgroup of PGL(3,$\C$) corresponding to those transformations preserving a Hermitian form of signature $(2,1)$.  

Complex hyperbolic space is a rank one Hermitian symmetric space and as such, it is a K\"ahler manifold with negative $\frac14$-pinched curvature. Totally geodesic real planes and complex lines realize the extremal values of the sectional curvature (namely, $-1$ for complex lines and $-\frac14$ for real planes). The boundary at infinity of $\HdC$ can be seen as the $3$-sphere $\S^3$. Complex lines and totally geodesic real planes give rise to two distinguished classes of curves in $\S^3$ : $\C$-circles and $\R$-circles respectively (see \cite{Goldman}). The sphere $\S^3$ inherits  a CR structure from the complex hyperbolic space. This CR structure defines a contact structure for which $\C$-circles are everywhere transverse (they are the chains of the CR structure) and $\R$-circles are Legendrian. We review these structures in \cref{section:PU21-action}.

 We first consider ${\rm PO(2,1)}$ seen as the stabilizer of a  totally real totally geodesic subspace of $\HdC$. These subspaces are often called real planes for short, and the typical example is 
$\HdR\subset \HdC$, which, in coordinates, is the set of real points of the complex unit ball. This embedding ${\rm PO(2,1)}\subset {\rm PU(2,1)}$ gives isometric actions of Fuchsian subgroups of PO(2,1) preserving $\HdR$. Such subgroups are called $\R$-Fuchsian. The main theme we address here is to study deformations of $\R$-Fuchsian subgroups of ${\rm PU(2,1)}$. 

The complex hyperbolic plane  has another type of totally geodesic subspaces : complex lines, which give rise to the notion of $\C$-Fuchsian subgroups of PU(2,1). But, contrary to the $\R$-Fuchsian case,  any deformation of a cocompact $\C$-Fuchsian subgroup of ${\rm PU(2,1)}$ is still $\C$-Fuchsian (see \cite{T} for this rigidity result and \cite{Koz-Mau} for a review and further generalizations).

For a discrete subgroup of $\PU(2,1)$, a most natural object to consider is its limit set in $\S^3$ , which is a topological circle in the quasi-Fuchsian case. 
We aim at understanding the relative position of the limit set of a  quasi-Fuchsian group and $\C$-circles in $\S^3$.

\subsection{Horizontality, hyperconvexity and slimness in the sphere}

We consider three related notions for subsets in $\S^3$. For the definition of these notions, we use the Cartan invariant $\A$ of  triples of points in $\S^3$. It is a numerical invariant that classifies oriented triples up to the action of $\PU(2,1)$. 
For now, let us only mention that the Cartan invariant takes all values in $[-\frac\pi2,\frac\pi2]$, and that a triple (of pairwise distinct points) $(p_1,p_2,p_3)$ is contained in an $\R$-circle (resp. a $\C$-circle) if and only if $\A(p_1,p_2,p_3) = 0$ (resp. $\A(p_1,p_2,p_3) = \pm\frac\pi2$). In particular, we note that if 
$|\A(p_1,p_2,p_3)|<\frac\pi2$ and the three points are distinct then the triangle ($p_1,p_2,p_3)$ does not belong to any $\C$-circle. A more detailed presentation is given in \cref{section:subspaces}.

Though our initial interest was for limit sets, we will first drop the invariance assumption. In \cref{sec:horizontality-slimness}, we work with arbitrary compact subsets $E$ of $\S^3$ and consider the following three properties:

\begin{itemize}
 \item {\it Horizontality.} This is an extension for arbitrary compact subsets of $\S^3$ of the concept of Legendrian submanifolds and is a local property. It is defined in \cref{def:horizontality}. It amounts to ask that convergences $p_n\to p$ in $E$ only happen tangentially to the contact structure, see \cref{lem:local}. We describe in \cref{subsec:horizontality-orbits-1} some horizontal orbits of one-parameter subgroups.
 \item {\it Hyperconvexity.} A subset of $\S^3$ is called hyperconvex if  its intersection with any $\C$-circle contains at most two points. This notion is a version of a central notion in the theory of Anosov representations, stemming from \cite{Labourie}
and, in a context similar to this paper, in \cite{PozzettiSambarinoWienhard}.
 \item {\it Slimness.} This is a quantitative notion that implies hyperconvexity and horizontality. For a closed subset $E$ of 
 $\S^3$, we define
 $$\A(E) = \sup\lbrace |\A(p,q,r)|, p,q,r \in E\rbrace.$$
We say that $E$ is \emph{$\alpha$-slim} whenever $\A(E)\leqslant \alpha <\pi/2$, see \cref{def:slimness}. It directly implies hyperconvexity, from the above mentionned properties of $\A$. But it also implies horizontality, as proven in \cref{prop:alphaslim_horizontal}. In case $E$ is the limit set of a representation of a surface group, the quantity $\A(E)$ can be interpreted using bounded cohomology as a Gromov norm of a cohomology class in the case of limit sets, see Point 5 in \cref{rem:definition}.
\end{itemize}
We give geometric interpretations of slimness in \cref{sec:projections}.
The simplest examples to study these three properties and their consequences are $\R$-circles (boundaries of real planes). They are Legendrian, hyperconvex and $0$-slim since any ideal triangle in an $\R$-circle has vanishing Cartan invariant. 
We will describe other families of examples and non-examples in \cref{section:examples}. In particular, we show that  slim deformations of $\R$-circles do exist. We define  \emph{bent $\R$-circles}, see \cref{section:bent}. In Heisenberg coordinates, for each $0< \theta< \pi$, the set
\[E_\theta = \{[r,0], r\in \R_+\} \cup \{[re^{i\theta},0], r\in \R_+\}\cup \{\infty\}\]
is slim, see \cref{prop:bent-Rcircles-slim}. Note that $E_\pi$ is in fact an $\R$-circle.

Moreover, 
as explained in \ref{sec:deformation-R-fuchsian}, if $\Gamma \subset \PO(2,1)$ is a cocompact $\R$-Fuchsian group, then it can be deformed in $\PU(2,1)$ and the limit sets will be slim along this deformation, at least locally. This remark is essentially borrowed from \cite{PozzettiSambarinoWienhard}.

\subsection{A foliation on the complement of an $\R$-circle}\label{intro:R-circles}

We relate the three properties above and a known identification between the complement of $\R$-circles and the unit tangent bundle $\unit{\HdR}$. Assume $\Lambda_0$ is the $\R$-circle $\partial_\infty\HdR$ and denote by $\Omega_0$ its complement in $\S^3$. Then for any pair of distinct points $p\neq q$ in $\Lambda_0$, denote by $\mathcal L(p,q)\subset \HdC$ the unique complex line containing  $p$ and $q$. 
The  $\C$-circle $\partial_\infty \mathcal L(p,q) \subset \S^3$ is naturally oriented by the complex structure of $\mathcal L(p,q)$. Moreover, it intersects $\Lambda_0$ only at $p$ and $q$ since $\Lambda_0$ is hyperconvex. It is therefore divided into two connected components, which are oriented intervals. We will denote these intervals by $\arc{p}{q}$ and $\arc{q}{p}$.  The starting point of our work is the following classical proposition:

\begin{proposition*}
 The open set $\Omega_0$ is homeomorphic to the unit tangent bundle of $\HdR$. In this homeomorphism, the arcs 
 $\arc{p}{q}$ correspond to the orbits of the geodesic flow on $ \unit{\HdR}$.
\end{proposition*}

We refer to \cref{subs:foliation} for more details. This proposition also tells us that $\Omega_0$ is foliated by the arcs $\arc pq$. All along this paper, we reinterpret it in various ways, see \cref{coro:foliation}, \cref{R-Fuchsian-UTB}, \cref{coro:foliation-2}, \cref{proposition:homeoUT}. 

 If $\HdR$ is acted on by an $\R$-Fuschsian subgroup of PO(2,1)$\subset$PU(2,1) then so is $\Omega_0$ and the above homeomorphism descends to a homeomorphism between $\Gamma\backslash\Omega_0$ and the unit tangent bundle $\unit{(\Gamma\backslash\HdR)}$ of $\Gamma\backslash\HdR$ where orbits of the geodesic flow correspond to projection of arcs.

\subsection{Deforming the foliation}

Describing deformations of this foliation when deforming 
$\Lambda_0$ is one of the main points of this article (\cref{section:deform-foliation}).
As explained before, there exist deformations $(\Lambda_t)$ of $\Lambda_0$ such that all $\Lambda_t$ are slim. Denote by $\Omega_t$ the complement of $\Lambda_t$.
 First, we prove that arcs of $\C$-circles sweep out $\Omega_t$:
\begin{theorem*}[First point of \cref{thm:foliation-deformation}]
Let $\Lambda_t$ be a continuous family of slim circles, with $\Lambda_0$ an $\R$-circle. Then, for all $t$, the arcs $\arc pq$, for $p\neq q \in \Lambda_t$, sweep out $\Omega_t$.
\end{theorem*}
The strategy to prove this theorem is interesting per se. We first prove in \cref{sec:extension} that a horizontal and hyperconvex circle $\Lambda$ can be continuously extended \emph{outside} the complex hyperbolic space: there is an explicit continuous embedding of the Möbius strip in $\CP^2\setminus\HdC$ whose intersection with $\partial_\infty \HdC$ is exactly $\Lambda$. Our construction is flexible enough to prove that, under deformations of $\Lambda$, the Möbius strips deform by homotopy, see \cref{sec:surjectivity}. One can then apply an argument of intersection in homology to prove the theorem.

Thanks to this theorem, we can exhibit an actual deformation $\Lambda_t$ such that arcs of $\C$-circles define a foliation of $\Omega_t$:
\begin{theorem*}[\cref{prop:deform-standard-foliation}]
For any $\theta\in [\pi/2,3\pi/2]$, the set of arcs of $\C$-circles with endpoints in $E_\theta$ defines a foliation of $\S^3\setminus E_{\theta}$.
\end{theorem*}
A caveat is necessary here: not all bent $\R$-circles give rise to a foliation. Indeed, if the bending is too strong ($|\pi -\theta|>\frac{\pi}{2}$), then some arcs do intersect.

It is hard to deform an $\R$-circle into a slim circle invariant under a group and such that the arcs between couples of its points define a foliation. Indeed, the invariance by a single non-real loxodromic element implies that some arcs intersect. Recall that a loxodromic element of $\PU(2,1)$ is non real if the trace of its cube - which is well defined - is not real.
\begin{theorem*}[Second point of \cref{thm:foliation-deformation}]
Let $\Lambda$ be a slim circle, that is invariant by a non-real loxodromic transformation. Then there are arcs $\arc pq$, with $p\neq q \in \Lambda$, that intersect in the complement $\Omega$ of $\Lambda$.
\end{theorem*}

We get as a corollary that no non$-\R$-fuchsian deformation of a lattice in $\PO(2,1)$ determines a foliation of the complement of its limit set by arcs of $\C$-circles.  This can also be interpreted as  the following rigidy theorem:

\begin{theorem*}[See \cref{thm:crown-uniformisations}]
Let $\Gamma$ be a cocompact lattice in $\PO(2,1)$ and $\rho:\Gamma\to \PU(2,1)$ be a small deformation of the inclusion. Let $\Lambda$ be its limit set and $\Omega$ its complement.

If $\Omega$ is foliated by arcs $\arc pq$, for $p\neq q \in \Lambda$, then $\rho$ is $\R$-fuchsian.
\end{theorem*}

\subsection{Drilling and crown-type uniformisations}

We will call here CR-spherical uniformization of a manifold $M$ a homeomorphism $M \simeq \Omega/\rho(\pi_1(M))$, where $\Omega$ is an open subset of the sphere on which $\rho(\pi_1(M))$ acts properly discontinuously (see \cite{Fanny-ICM}). One should be careful with this definition as, sometimes,  uniformization refers to the case  where $\Omega$ is assumed to be the domain of discontinuity of $\rho(\pi_1(M))$. This is for instance the definition taken by Deraux in \cite{deraux-experimental} (see Definition 1.3 there). In particular when $\Omega$ is the domain of discontinuity of $\rho(\pi_1(M))$, then the 3-manifold that is uniformized appears as the boundary at infinity of a quotient of the complex hyperbolic  plane. 
This happens for most of the examples of uniformizations of hyperbolic 3-manifolds that have been constructed (see for instance \cite{Schwartz-rhochi,DerauxFalbel,ParkerWill}), but we will consider here examples where it is not the case. Note also that $\Omega$ needs not be simply connected - and is not in our examples. As a consequence, $\rho$ is  not  injective in general.

Going back to deformation of $\R$-fuchsian surface groups, general arguments about geometric structures, namely Ehresmann-Thurston principle and work by Guichard-Wienhard \cite{GuichardWienhard}, imply that, when deforming $\Gamma$ by a deformation $\rho$ close enough to the inclusion, the complement $\Omega$ of the limit set $\Lambda$ of $\rho(\Gamma)$ still uniformizes $\unit{\Sigma}$. We recall these arguments in \cref{proposition:deformation-quotient}.

We can drill along closed orbits of the geodesic flow in $\unit{\Sigma}$. For an oriented closed geodesic $\lambda$, denote by $\unit{\Sigma}(\lambda)$ the unit tangent bundle drilled out along the natural lift of $\lambda$. The uniformizations of $\unit{\Sigma}$ described above naturally give uniformizations of $\unit{\Sigma}(\lambda)$. The manifolds constructed in this way cover in particular a number of hyperbolic cusped manifolds. We say that $\lambda$ is filling if its complement in $\Sigma$ is a union of discs. Then, by \cite{FoulonHasselblatt}, as soon as $\lambda$ is \emph{filling}, the drilled out unit tangent bundle is hyperbolic. We sum up this discussion in the proposition:
\begin{proposition}[\cref{coro:drilled}]
Every manifold obtained by drilling a closed orbit of the geodesic flow in the unit tangent bundle of a hyperbolic surface admits a family of CR-spherical uniformizations.

An infinite number of cusped hyperbolic $3$-manifolds can be obtained this way.
\end{proposition}

We use the previous work to describe explicitly these uniformizations: having fixed a small deformation $\rho$, we want to describe an open subset whose quotient by $\rho(\Gamma)$ is homeomorphic to $\unit{\Sigma}(\lambda)$. To achieve that goal, we consider an element $\gamma$  in $\Gamma$ whose oriented axis lifts $\lambda$. A small deformation $\rho$ verifies that $\rho(\gamma)$ is still a loxodromic transformation: it has a repelling and an attractive fixed points, denoted by $\rho(\gamma)_-$ and $\rho(\gamma)_+$, both belonging to the limit set $\Lambda$ of $\Delta:=\rho(\Gamma)$. We call the axis at infinity of $\delta:=\rho(\gamma)$ the arc $\alpha(\delta) = \arc{\rho(\gamma)_-}{\rho(\gamma)_+}$. Then, we define in \cref{sec:Crowns} the \emph{crown}:
\[{\rm Crown}_{\Delta,\delta} = \Lambda \cup\Bigl(\bigcup_{g\in\Gamma} \rho(g)\cdot\alpha_\delta\Bigr).\]
The crown is a closed set containing the limit set and is invariant under the action of $\Delta = \rho(\Gamma)$. We denote by $\Omega_{\Delta,\delta}$ its complement. We describe the following explicit family of uniformizations of $\unit{\Sigma}(\lambda)$:
\begin{theorem*}[See \cref{pro:deformed_crowns}]
For a small enough deformation $\rho$, the quotient $\Delta\backslash\Omega_{\Delta,\delta}$ is homeomorphic to $\unit{\Sigma}(\lambda)$.
\end{theorem*}
The proof works by deformation: if $\rho$ is $\R$-fuchsian, this proposition is only a rephrasing of the foliation property. By small deformations, everything varies continuously and the family of axis $\rho(g)\cdot \alpha(\delta)$ do not intersect.

\subsection{Further questions and open problems}

As mentioned above, many of the previously known examples 
of spherical CR uniformizations of hyperbolic 3-manifolds have been constructed as quotients of the 
whole discontinuity region of a discrete subgroup of 
PU(2,1). Many of them also share a another common 
feature: the holonomy groups of the structure appear as 
degenerations of quasi-Fuchsian deformations of 
discrete subgroups of PO(2,1), typically $(p,q,r)$-triangle groups. Note that other uniformizations have 
been obtained by applying Dehn-filling techniques to 
uniformizations obtained from these degenerations 
\cite{Schwartz-book,Acosta-surgeries-1,Acosta-surgeries-2}. The typical situation observed is 
the following.

Let $\Gamma$ be a Fuchsian group, and $\rho_0 : \Gamma
 \longrightarrow \PO(2,1)\subset \PU(2,1)$ be an 
 $\R$-Fuchsian representation. For a variety of 
 examples of 1-parameter families of deformations 
 $\rho_t$ of $\rho_0$, there exists a word $w$ in 
 $\Gamma$ which becomes parabolic for a critical value 
 $t_w$ (for any $t<t_w$, all words are mapped to 
 loxodromic transformations). The representation 
 $\rho_t$ is discrete and faithful on the interval $[0,t_w]$ and is either non-discrete 
 or non-faithful for $t>t_w$. It is in particular the 
 conjectured situation when $\Gamma$ is a triangle group. 
 Indeed, the Schwartz conjectures 
 \cite{Schwartz-triangle} predict precisely which word 
 $w$ should become parabolic. In all cases where a detailed 
 study of the long-time deformations of a triangle 
 group have been achieved, the manifold at infinity for the critical value $t=t_w$ 
 is a hyperbolic knot or link complement \cite{Schwartz-rhochi,DerauxFalbel,ParkerWill,Jiang-Wang-XIe,Ma-Xie}. Note that in the 
 case of triangle groups the character variety has 
 dimension 1 thus the situation is relatively simple 
 algebraically. However, even in this simpler case, 
 doing a complete analysis is a difficult and very 
 technical task based on the construction of fundamental domains. 
 Also, it is not completely clear to this day if one can predict what 
 $3$-manifold is likely to appear as degeneration of a 
 given triangle group deformation (see for instance the ubiquity phenomenon described by Deraux in Theorem 1.5 of \cite{deraux-experimental} and extended recently by Alexandre in \cite{Raphael-Alexandre}). The Schwartz conjectures have been generalized to some extent for quasi-Fuchsian deformations of surface groups by Parker and Platis (see Problem 6.2 in \cite{Parker-Platis}). We hope that this work could be a step toward a better understanding of these long time deformations.

Let us describe the situation of the $(3,3,4)$-triangle group, generated by three reflections $\iota_1$, $\iota_2$, $\iota_3$, see \Cref{ex:triangle-groups} for precise notations. It is a known fact that  the degeneration of the $(3,3,4)$-triangle group  corresponds to the word $w=\iota_3\iota_2\iota_1\iota_2$ becoming parabolic and yields a uniformization of the figure eight knot complement by the even subgroup of the triangle group (see \cite{DerauxFalbel,Parker-Wang-Xie}). The trace of the image of $w$, denoted by $\tau$, can be used (up to a 2-fold covering) as a coordinate for the deformation space.  We thus have a 1-parameter family of representations $\rho_\tau$ of the $(3,3,4)$-triangle group in PU(2,1).

The $\R$-fuchsian representation corresponds to the value $\tau = 2+2\sqrt{2}$, whereas the
degeneration corresponds to $\tau = 3$ (in that case $w$ is mapped to a unipotent parabolic). Note that it can happen that $\tau$ becomes smaller than $3$, in which case $\rho_\tau(w)$ is elliptic, and the representation is either non discrete or non-faithful in that case. One can estimate the supremum of Cartan invariants $\A(\Lambda_\rho)$ for the limit sets of these representations. Numerical experimentations indicate that the supremum is strictly increasing from $0$ to $\pi/2$ as $\tau$ decreases from $2+2\sqrt{2}\sim 4.828$ to $3$, with $\pi/2$ being attained for the degeneration (see \cref{fig:CartanEstimation}, where the horizontal coordinate is $\tau$).  In other words, the limit sets of the representations $\rho_\tau$ seem to remain slim until the degeneration.

Applying techniques using fundamental domains (see \cite{DerauxFalbel,Parker-Wang-Xie}), it is possible to prove that the manifold uniformized by the action of the even subgroup of the $(3,3,4)$-triangle group on its discontinuity region is as follows.

\begin{itemize}
 \item For $\tau_0 = 2+2\sqrt{2}$ the group is $\R$-Fuchsian, and the 3-manifold uniformised by the action of the even subgroup of the $(3,3,4)$-triangle group on its discontinuity region is the unit tangent bundle of the $(3,3,4)$-orbisurface.
 \item For $\tau \in ]3,2+2\sqrt{2}[$, the image of the group group remains discrete and isomorphic to the $(3,3,4)$-triangle group. The manifold at infinity remains the same.
  \item For $\tau =3$, the word $w$ becomes unipotent parabolic. This implies a pinching of the limit set (the attractive and repulsive fixed points of $\rho_\tau(w)$ and of its conjugates coalesce), and the manifold at infinity changes : it is the figure eight knot complement.
\end{itemize}

However, at the initial value $\tau_0 = 2+2\sqrt{2}$, the action of the group on the complement of the {\it crown} associated to $\rho_{\tau_0}(w)$ already uniformizes the figure eight knot complement (this follows from \cite{Dehornoy}). Here the open subset giving the uniformisation is smaller than the discontinuity region. So we conjecture that all along this deformation, the crowns remain embedded and we have
a family of uniformizations of the figure eight knot
complement, with the last one being by the actual
domain of discontinuity of the represented group.

Our results show that in general the topological type of the 
uniformized 3-manifold remains constant close to the 
$\R$-Fuchsian crown-type uniformization, without considering explicit fundamental domains.

\begin{figure}[ht]
\begin{center}
\includegraphics[width=.6\textwidth]{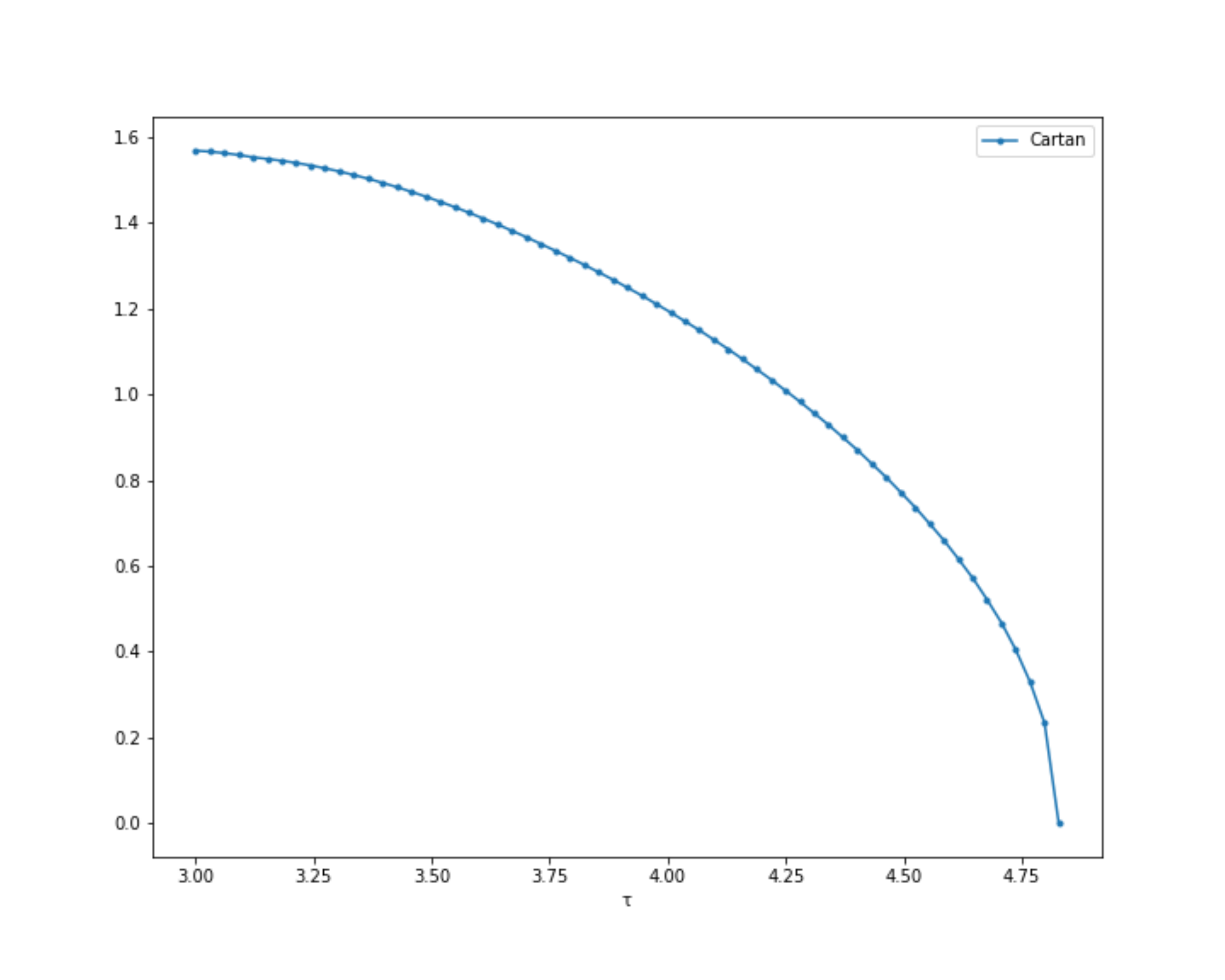}
\caption{Estimation of the supremum of Cartan invariant for the $(3,3,4)$-triangle groups.}\label{fig:CartanEstimation}
\end{center}
\end{figure}

One can also consider larger deformation spaces. The website \cite{Landscape} presents experimentations about the even subgroup of the $(3,3,4)$-triangle group, which has a $2$-parameter family of deformations around the $\R$-fuchsian one, together with an estimation of the supremum of the Cartan invariant. The following questions seem very natural, for any $\R$-Fuchsian group:
\begin{itemize}
	\item Does the whole connected component of convex-cocompact (or, equivalently, Anosov) deformation of the $\R$-fuchsian representations consists of slim ones (or, equivalently, hyperconvex Anosov)?
	\item Which words in the group can become parabolic at the boundary of slim convex-cocompact  representations?
	\item In the case of a representation at the boundary of slim convex-cocompact deformations with a finite set of classes of words having become parabolic, is the topology of the uniformized manifold related to the topology of a crown for the $\R$-fuchsian representation?
\end{itemize}

\medskip

\textbf{Acknowledgements. } We thank Danny Calegari, Pierre Dehornoy, Patrick Foulon, Julien March\'e and Andrés Sambarino for enlightening exchanges.


\section{$\PU(2,1)$-geometry of $\CP^2$\label{section:PU21-action}}

One of the main thrusts behind this paper is that the geometry of some convex-compact representations of surface groups in $\PU(2,1)$ are best understood considering not only the natural action on the complex hyperbolic space $\HdC$ and its $3$-sphere at infinity $\partial \HdC$ but also on its complement $\HuuC$ in $\CP^2$. We hope to illustrate how the whole $\PU(2,1)$-geometry of $\CP^2$ helps understanding these representations. In this section, we review necessary material about this geometry.

We will constantly use points in the projective space $\CP^2$ 
and lifts to $\C^3$. In this situation, we will denote the 
point and its lift by the same letter, but bolded for the 
lift. As example, if $p$ is a point (resp. $A$ is a projective 
transformation), $\bp$ is a lift of $p$ (resp. $\bA$ is a 
matrix lift of $A$).  We denote by $\PU(2,1)$ the projective 
unitary group associated to a Hermitian form $\la\cdot,
\cdot\ra$ of signature $(2,1)$ on $\C^3$. At this stage, we do 
not specify this form.

\subsection{Action of $\PU(2,1)$ on $\CP^2$.\label{section:orbits}}

The action of $\PU(2,1)$ on $\CP^2$ has three orbits which are the projections to $\CP^2$ of the three cones 
in $\C^3$ defined by 
\begin{eqnarray} V^- &=& \{ Z\in\C^3,\la Z,Z\ra<0\},\nonumber\\ 
                 V^+ &=& \{ Z\in\C^3,\la Z,Z\ra>0\},\\ 
                 V^0 &=& \{ Z\in\C^3,\la Z,Z\ra=0\}\nonumber.
\end{eqnarray}
Clearly the two orbits $\P(V^\pm)$ are open, and $\P(V_0)$ is closed. We will say that a point $p\in\CP^2$ has \emph{negative, null or positive type} when it belongs respectively to $\P(V^-)$, $\P(V^0)$ or $\P(V^+)$.
As sets, the two open orbits identify respectively to the homogeneous spaces
\begin{eqnarray}\label{eq:homog-spaces}
\P(V^-) & \sim & \PU(2,1)/\mathrm{P}(\mathrm{U}(2)\times \mathrm{U}(1)) =\HdC\nonumber\\ 
\P(V^+) & \sim & \PU(2,1)/\mathrm{P}(\mathrm{U}(1)\times \mathrm{U}(1,1)) = \HuuC 
\end{eqnarray}
We thus view these two homogeneous spaces as subsets of $\CP^2$, where each of them appears as the complement of the closure of the other. These two spaces  can be equipped with metrics : a Hermitian one for $\HdC$ and pseudo-Hermitian one for $\HuuC$. Let us describe these metrics (see also \cite{Weber}). First, whenever $p\in\CP^2$ does not have null type, we use the identification of the tangent space at $p$ given by

\begin{equation}\label{eq:tangent-space}T_p\CP^2 = {\rm Hom}\left( \C\bp,\bp^\perp\right).\end{equation}

Now, if $\alpha,\beta$ are two linear maps $\C\bp\longrightarrow \bp^\perp$, the metric is given by
\begin{equation}\label{eq:metric}
 h_p(\alpha,\beta)=-4\dfrac{\la\alpha(\bp),\beta(\bp)\ra}{\la\bp,\bp\ra}
\end{equation}
Choosing a lift $\bp$ of $p$ so that $\la\bp,\bp\ra=-1$ if $p\in\HdC$ and $\la\bp,\bp\ra=+1$ if $p\in\HuuC$, and 
identifying $\alpha$ and $\beta$ with the images of $\bp$ denoted by $\alpha(\bp)=u$, $\beta(\bp)=v$, we obtain
\begin{eqnarray}\label{eq:infinitesimal-metric}
 h_p(u,v) & = & 4\la u,v\ra\mbox{ if }p\in\HdC \nonumber\\
 h_p(u,v) & = & -4\la u,v\ra \mbox{ if } p\in\HuuC 
\end{eqnarray}
If $p\in\HdC$, the direction $\C \bp$ has negative type, and the restriction of $\la\cdot,\cdot\ra$ to $(\C \bp)^\perp$ has signature $(+,+)$. Thus in this case, $h$ is a Hermitian metric on $T\HdC$, whose real part is Riemannian. This is the complex hyperbolic metric. The factor $4$ in \eqref{eq:metric} corresponds to normalising the sectional curvature of $\HdC$ as being pinched between $-1$ and $-\frac14$. If $p\in\HuuC$, the direction $\C \bp$ has positive type, and the restriction of $\la\cdot,\cdot\ra$ to $(\C p)^\perp$ has signature $(+,-)$. Therefore, in this case, $h$ is a pseudo-Hermitian metric on $T \HuuC$ with (Hermitian) signature $(1,1)$, whose real part is pseudo-Riemannian with signature $(2,2)$.

 The complex hyperbolic distance on $\HdC$ can be expressed in Hermtian terms by
\begin{equation}\label{eq:hyp-distance}
\cosh^{2}\left(\dfrac{d(p,q)}{2}\right) = \dfrac{\la \bp,\bq\ra\la\bq,\bp\ra}{\la \bp,\bp\ra\la\bq,\bq\ra}. 
\end{equation}

The third orbit $\P(V^0)$ of the $\PU(2,1)$-action on $\CP^2$, the closed one, is the projection to $\CP^2$ of the quadric 
$\{Z\in\C^3,\la Z,Z\ra = 0\}$. This orbit can be thought of as the boundary at infinity of $\HdC$, and we will denote it as $\partial\HdC$ (it is of course also the boundary of $\HuuC$ as well). It is a $3$-sphere and we will also often denote it simply by $\S^3$. Once a lift $\bp$ of $p$ is chosen, the tangent space $T_p\partial\HdC$ can be identified with the $3$-dimensional real vector subspace of $\C^3$ defined by $ \{Z\in\C^3,\, \Re\left(\la Z,\bp\ra\right)=0\}$. This tangent space contains the complex $1$-dimensional subspace $\ker (\la\cdot,\bp\ra)$. This defines a $CR$-structure on $\partial\HdC$, which is the homogeneous CR structure given by the field of tangent complex lines $\bigl(\ker \la\cdot,\bp\ra\bigr)_{p\in\partial\HdC}$, see  \cite{BurnsShnider}. 
The contact structure defined by this field of planes allows one to define horizontal submanifolds

\begin{definition}\label{def:CR-horizontal}
A smooth submanifold of $\bHdC$ is \emph{horizontal}  if at each point its tangent space is included in the contact plane.
\end{definition}
Such a manifold, if connected, can only be a point or a Legendrian curve. One of the main point of Section \ref{sec:horizontality-slimness} will be to extend this notion to non-smooth locally closed sets.

\subsection{Coordinate systems}\label{section:coordinates}

Let us describe the objects considered in the previous section with the following two special choices of Hermitian forms.

\begin{equation}\label{eq:herm-models}
H_B = \begin{bmatrix} 1 & & \\ & 1 & \\ & & -1 \end{bmatrix}
\mbox{ and }
H_S = \begin{bmatrix}  & & 1\\ & 2 & \\1 & &  \end{bmatrix}.
\end{equation}
Using the Hermitian form given by $H_B$ leads to the so-called {\it ball model} of $\HdC$. With this choice of coordinates, $\HdC$ can be seen as the unit ball of $\C^2$, where $\C^2$ itself is seen as the affine chart $Z_3=1$ of $\CP^2$. Any point in $\HdC$ can be lifted to $\C^3$ in a unique way as a vector $[z_1,z_2,1]^T$, where $z_i\in\C$ and $|z_1|^2 + |z_2|^2<1$. In this model, the boundary $\partial\HdC$ is just the 3-sphere $\S^3$ defined by $|z_1|^2 + |z_2|^2=1$. 
 In turn,  $\HuuC$ identifies with the complement in $\CP^2$ of the closed ball 
$\HdC\cup\S^3$.

On the other hand, if one uses the form $H_S$, then the projection of $V^-\cup V_0$ to $\CP^2$ is contained in the affine chart $\{Z_3=1\}$, except for the projection of $[1,0,0]^T$, which is at infinity. Thus any point in the closure of $\HdC$ admits a unique lift to $\C^3$ which is given by

\begin{equation}\label{eq:lift-Siegel} v_{(z,t,u)} =  \begin{bmatrix} -|z|^2-u+it\\z\\1 \end{bmatrix}
\mbox{ and } {\bm \infty} = \begin{bmatrix} 1\\0\\0\end{bmatrix},\end{equation}
where $z\in\C$, $t\in\R$ and $u\geqslant 0$. These coordinates are often called {\it horospherical coordinates} since the level sets of $u>0$ are the horospheres centered at $\infty$. When necessary, we will call the vector  given in \eqref{eq:lift-Siegel} the {\it standard lift} of a point in $\HdC$. We will denote by $[z,t]$ the point in $\partial\HdC$ which is the projection of $v_{z,t,0}$. Note that
$$ \la v_{(z,t,u)},v_{(z,t,u)} \ra=-2u,$$
so that the vectors $v_{(z,t,u)}$ for which $u<0$ are lifts of those points of $\HuuC$ that belong to the affine chart $\{Z_3=1\}$. The line at infinity is the projection to $\CP^2$ of $\ker(\la \cdot, {\bm \infty}\ra)$. It can be identified with the tangent complex line at $\infty$. Similarly, the tangent complex line $\ker(\la,\bp\ra)$ at points $p=[x+iy,t]\in\partial\HdC$ is easily seen to be the  kernel of the $1$-form 
\begin{equation}\label{eq:contact-form}
 \alpha = dt -2xdy+2ydx.
\end{equation}

The $1$-form $\alpha$ is the contact form of the Heisenberg group. 
A $C^1$ curve $\gamma$ in $\partial\HdC$ is horizontal, or Legendrian, if and only if its velocity belongs to the contact plane. This condition can be written with lifts in a simple way:
 $\gamma$ is horizontal if and only if it satisfies
\begin{equation}\label{eq:CR-horizontal}
\forall s\in\R,\, \la \overset{{\bm \cdot}}{{\bm \gamma}}(s),{\bm \gamma}(s)\ra = 0,
\end{equation}
where ${\bm \gamma}(s)$ is the standard lift of $\gamma(s)$.

\subsection{Totally geodesic subspaces and the Cartan invariant\label{section:subspaces}}

The maximal totally geodesic spaces of $\HdC$ come in the following two types.
\begin{enumerate}
 \item The complex lines  of $\HdC$ are the non-empty intersections with $\HdC$ of projective lines in $\CP^2$. Note that a projective line intersects $\HdC$ iff it is the projectivisation of a hyperbolic $2$-plane of $\C^{2,1}$ (that is, those where the restriction of $\la\cdot ,\cdot \ra$ has signature $(+,-)$). The sectional curvature along a complex line is constant and equal to $-1$. Typical examples are the complex axes of coordinates in the ball model of $\HdC$.
\item The real planes of $\HdC$ are the non-empty intersections with $\HdC$ of real projective planes. Real projective planes intersecting $\HdC$ can be described as projectivisations of totally real subspaces of $\C^{2,1}$, that is $3$ dimensional real subspaces of $\C^{2,1}$ for which the restriction of $\la\cdot ,\cdot\ra$ is real. These real planes realize the other bound $-1/4$ of the sectional curvature.
 \end{enumerate}
 
  As just said, complex lines and real planes of $\HdC$ are the intersections with $\HdC$ of projective complex lines or projective real planes of $\CP^2$. When clear from the context, we will often use the words complex line or real plane for both the complex hyperbolic or projective objects. When necessary, we will precise complex hyperbolic lines or real hyperbolic planes, as opposed to complex projective lines and real projective planes.
  
  We will make a constant use of the curves defined in $\partial\HdC$ by intersecting complex lines and real planes with $\HdC$.
\begin{definition}
A {\it $\C$-circle} in $\partial\HdC$ is the intersection of a complex line of $\HdC$ with $\partial\HdC$. Similarly, an {\it $\R$-circle} in $\partial\HdC$ is the intersection of a real plane of $\HdC$ with $\partial\HdC$.
\end{definition}

\begin{example}\label{example:RandC-circles}
Examples of $\R$- and $\C$-circles in the Heisenberg space are depicted in Figures \ref{fig:Rcircles} and \ref{fig:Ccircles}. Their description is as follows:
\begin{enumerate} 
 \item  In Heisenberg coordinates, the two axes of coordinates in the plane $\C\times\{0\}$ are examples of $\R$-circles, and more generally, so is any line through the origin in that plane.  The axis $\{[0,t],t\in\R\}$ is a $\C$-circle. More generally, the $\R$-circles that contain the point $\infty$ are  the lines through a point $p$ that are contained in the contact plane at $p$. The $\C$-circles through $\infty$ are the vertical lines. 
 \item The $\R$-circles that do not contain the point $\infty$ are (compact) circles whose projections onto $\C$ is 
 a square lemniscate (the tangents at the double point of the projection are orthogonal). The $\C$-circles not containing $\infty$ are ellipses contained in contact planes, that are centered at the contact point. 
\end{enumerate}
 \end{example}
Note in particular that $\R$-circles are horizontal, whereas $\C$-circles are everywhere transverse to the contact distribution. The latter facts are clear in the situation where the considered $\R$ or $\C$-circle contains $\infty$, and follow from the transitivity of the action of $\PU(2,1)$ on the two families of complex hyperbolic lines and real hyperbolic planes.

Another notable difference between $\C$-circles and $\R$-circles is that $\C$-circles have a natural orientation which is induced by the complex structure of the complex line they bound, whereas $\R$-circles do not have a natural $\PU(2,1)$-invariant orientation.

\begin{figure}[ht]
\begin{center}
 \begin{tabular}{ccc}
\scalebox{0.5}{\includegraphics{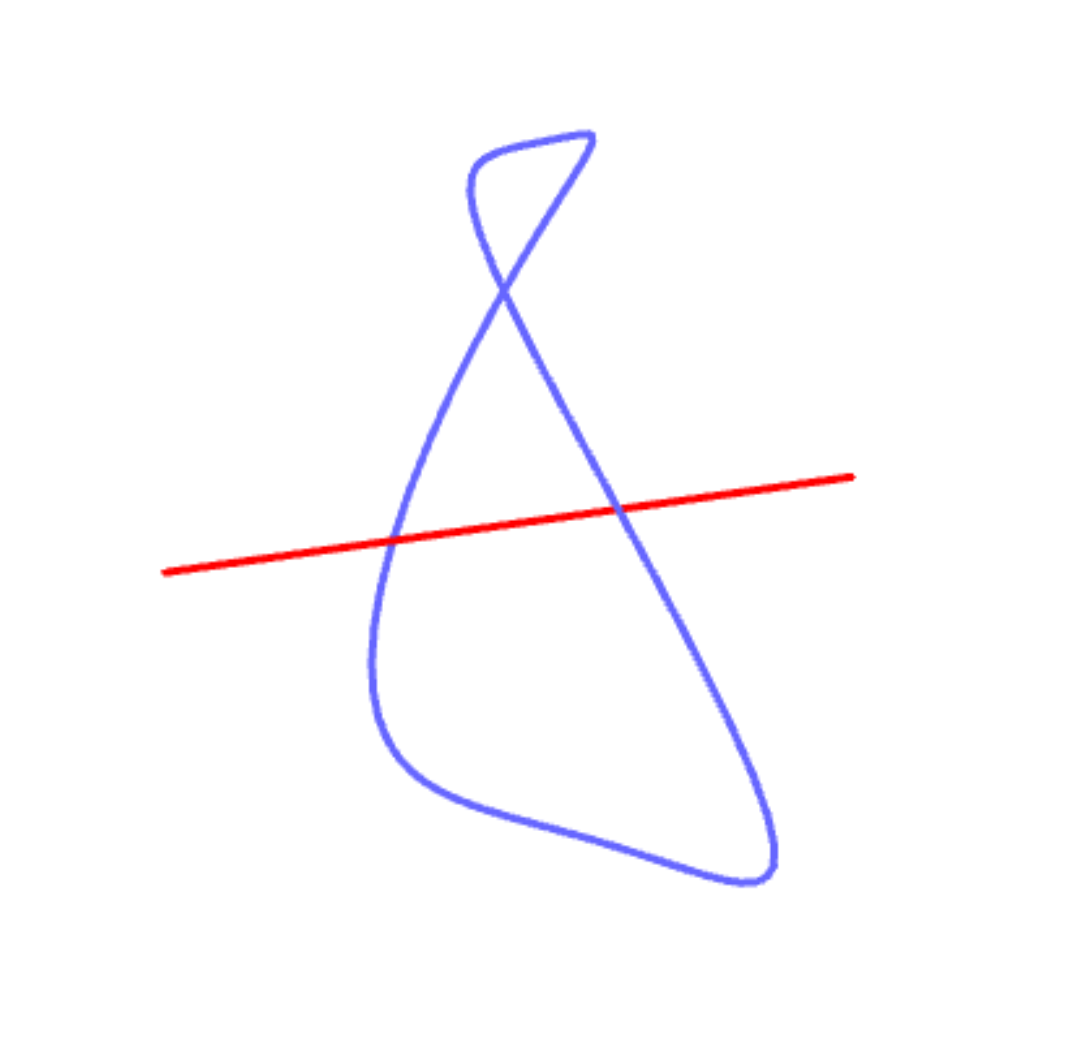}} 
   & &
\scalebox{0.5}{\includegraphics{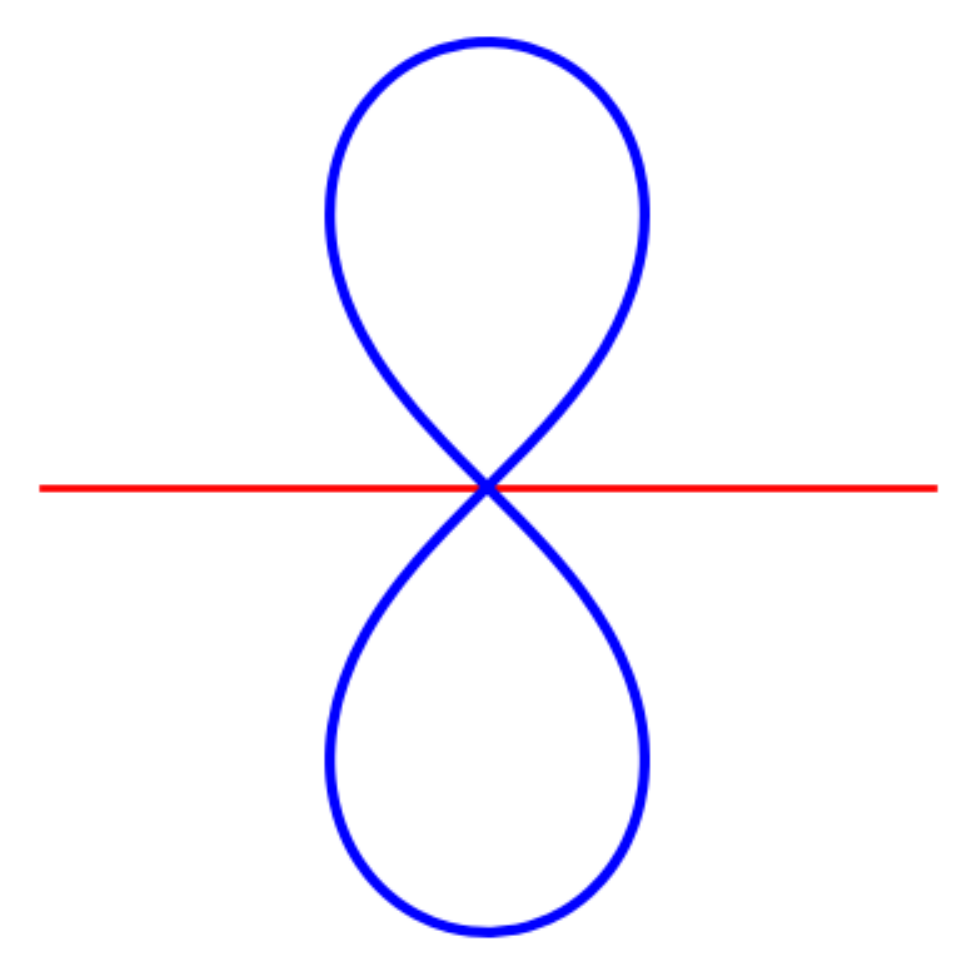}}\\
   \end{tabular}
\caption{Two $\R$-circles: the red line is the $x$ axis of the Heisenberg coordinates. It is the boundary of $\HdR=\R^2\cap\HdC$. The blue curve is the boundary of a real plane orthogonal to $\HdR$. The left picture is a view in perspective in Heisenberg space, and the right picture is the vertical projection of the two $\R$-circles on $\C$. \label{fig:Rcircles} }
\end{center}   
\end{figure}

\begin{figure}[ht]
\begin{center}
 \begin{tabular}{cc}
 \scalebox{0.5}{\includegraphics{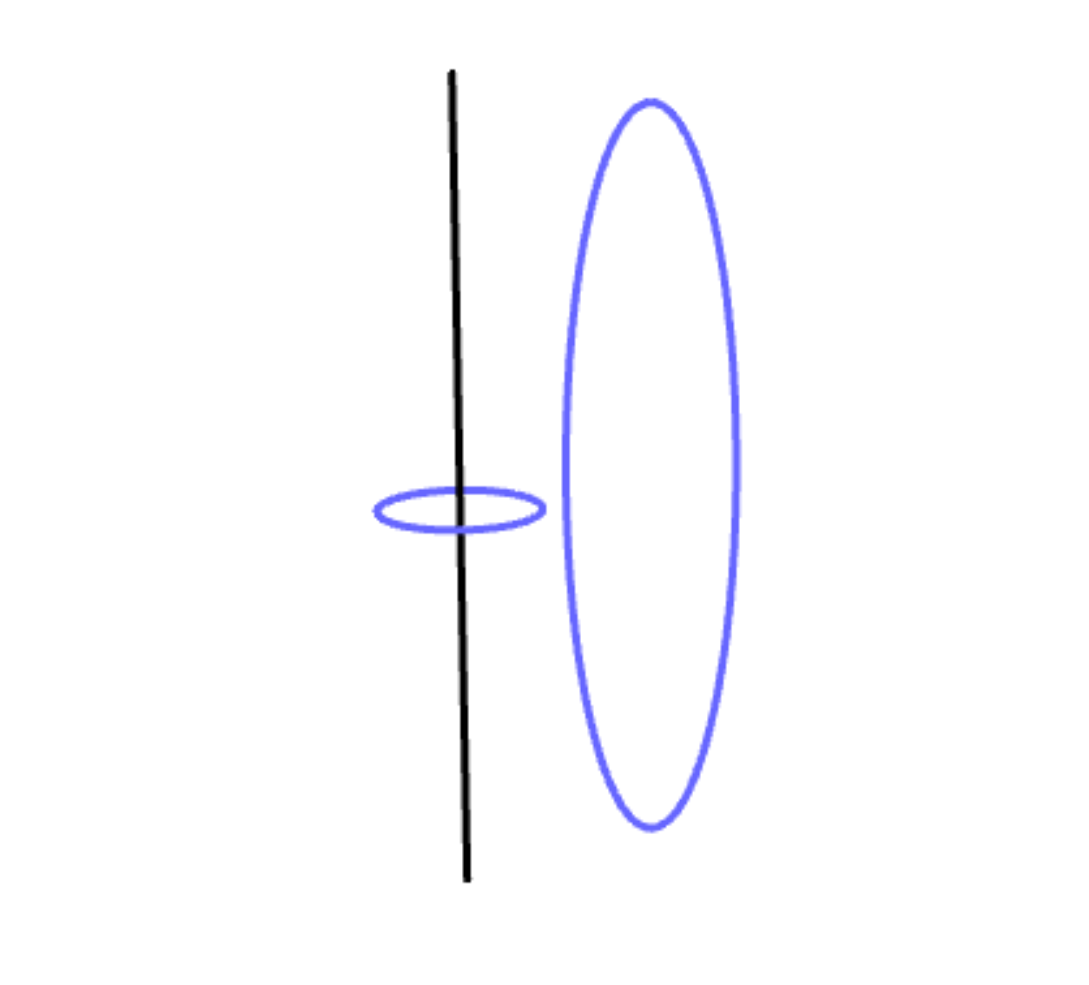}}
   & 
\scalebox{0.5}{\includegraphics{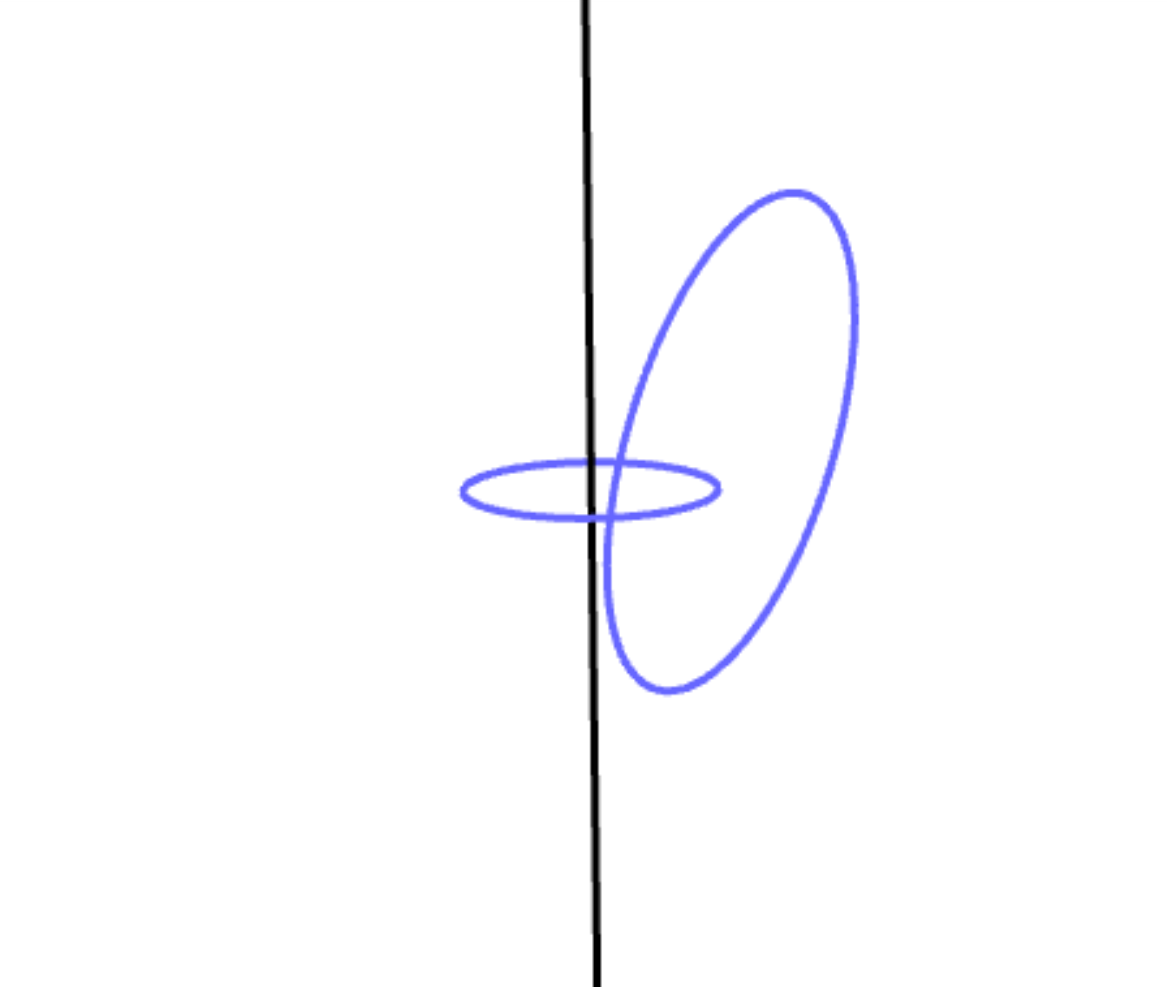}}\\
   \end{tabular}
\caption{Examples of $\C$-circles in Heisenberg space. On both pictures, the black line is a $\C$-circle passing through ${\bm \infty}$. Note that the pair of blue $\C$-circles on the left is unlinked, whereas the one the right picture they is linked.\label{fig:Ccircles}}
\end{center}   
\end{figure}

The Cartan invariant will play an important role in our work, from \cref{sec:slimness} on. It gives an easy characterization of triples of points that lie in a $\C$-circle or in an $\R$-circle.
\begin{definition}\label{defi:Cartan}
 Let $(p,q,r)$ be a triple of  points in $\partial\HdC$. If the points are pairwise distinct  we define the {\it Cartan invariant} of the triple $(p,q,r)$ to be
 \begin{equation}\label{eq:Cartan} 
  \A(p,q,r)=\arg\Bigl(-\la \bp,\bq\ra\la\bq,\br\ra\la\br,\bp\ra\Bigr).
 \end{equation}
 If at least two of the points coincide we define it to be  $\A(p,q,r)=0$.
\end{definition}
The quantity \eqref{eq:Cartan} does not depend on the choices made for lifts, and is $\PU(2,1)$-invariant.
The following statement sums up the main features of this invariant (see \cite[Chapter 7]{Goldman} for proofs).
\begin{proposition}\label{prop:Cartan-properties}
 The Cartan invariant enjoys the following properties.
 \begin{enumerate}
  \item For any triple $(p,q,r)$, $\A(p,q,r)\in[-\pi/2,\pi/2]$.
  \item Two triples of pairwise distinct points  $(p_1,p_2,p_3)$ and $(q_1,q_2,q_3)$ have the same Cartan invariant if and only if there exists a map $g\in \PU(2,1)$ such that $g(p_i)=q_i$ for $i=1,2,3$.
  \item For a triple of distinct points, $|\A(p,q,r)|=\pi/2$ if and only if the triple $(p,q,r)$ lies on the boundary of a complex line.
  \item $\A(p,q,r)=0$ if and only if the triple $(p,q,r)$ lies on the boundary of a real plane.
  \item $\A$ is a $3$-cocycle. In particular, if $p,q, r,s$ are four points, we have: 
  \begin{equation}\label{eq:Cartan-cocycle}
  \A(p,q,r) - \A(p,q,s) + \A(p,r,s) - \A(q,r,s)=0.\end{equation}
 \end{enumerate}
\end{proposition}

\subsection{The line map and the duality between $\HdC$ and $\HuuC$}\label{section:complex-lines}

We will often work with projective lines in $\CP^2$. The set of lines in $\CP^2$ can be described as the dual projective space, denoted by $\CPstar$. The Hermitian form gives a natural identification between $\CP^2$ and $\CPstar$, which in turn gives a polarity between points and lines of $\CP^2$. We review here some basic properties of this notion of polarity.

\begin{definition}\label{def:polar}
 Let $p$ be a point in $\CP^2$ and $L$ a projective complex line. We say that $p$ is {\it polar} to $L$ if $L=\P(\bp^\perp)$.
\end{definition}
The restriction of the Hermitian form on planes can be of signature $(+,+)$, $(-,+)$ or degenerate. Using polarity we can describe the situation as follows:
\begin{itemize}
 \item Positive type directions are orthogonal to $2$-planes with signature $(+,-)$. This means that points in $\HuuC$ are polar to complex lines that intersect $\HdC$.
 \item Negative type directions are orthogonal to $2$-planes with signature $(+,+)$. Thus points of $\HdC$ are polar to complex lines contained in $\HuuC$.
 \item Null type directions are orthogonal to $2$-planes with signature $(0,+)$. In fact, a point $p$ in $\partial\HdC$  is polar to its orthogonal complex line $p^\perp$ tangent to $\partial\HdC$ at $p$. It is the only case where $p$ belongs to its polar line.
\end{itemize}
In particular, we observe that $\HuuC$ is in bijection with the Grassmanian of complex lines of $\HdC$. This is indeed  another usual definition of $\HuuC$. 

Let us denote by $\Delta$ the diagonal of $\CP^2\times\CP^2$, and by $\Delta_{\S^3}$ the portion of $\Delta$ given by $\S^3$. For any pair $a,b$ of distinct points in $\CP^2$, we call $\mathcal{L}(a,b)$ the (unique) complex line containing $a$ and $b$. We note that this defines a $\PGL(3,\C)$-equivariant map on $(\CP^2\times\CP^2)\setminus \Delta$, where $\PGL(3,\C)$ acts diagonally on $\CP^2\times \CP^2$. We can extend this map to $\Delta_{\S^3}$ in a $\PU(2,1)$-equivariant way by defining $\mathcal{L}(a,a)$ to be the complex line tangent to $\S^3$ at $a$, that is $\mathcal{L}(a,a)=\P(\ker(\la ,\ba\ra))$. This is the largest $\PU(2,1)$-equivariant extension of $\mathcal{L}$ to a  subset of $\CP^2\times \CP^2$ larger than $(\CP^2\times\CP^2)\setminus \Delta$.

\begin{definition}\label{def:line-map}
The map 
\begin{equation}\label{eq:defi-linemap}\mathcal{L} : \Bigl(\left(\CP^2\times \CP^2\right)\setminus \Delta\Bigr)\cup\Delta_{\S^3} \longrightarrow \CPstar\end{equation} 
defined above is called the {\it line map}.
\end{definition}

The following proposition will play an important role in our work.

\begin{proposition}\label{prop:L-not-continuous}
The line map is continuous on $\CP^2\times \CP^2\setminus \Delta$, but it is not continous at any point of $\Delta_{\S^3}$.
\end{proposition}

\begin{proof}
Observe that for any neighborhood $U$ of a point $p$ in $\Bigl(\left(\CP^2\times \CP^2\right)\setminus \Delta\Bigr)\cup\Delta_{\S^3}$,
$\mathcal{L}(p,U)=\CP^1$ identified to the set of all lines passing through $p$.  This shows that the line map is not continuous at diagonal
points.
\end{proof}

As explained before, we can identify $\CPstar$ to $\CP^2$ using the Hermitian form. Using this polarity, we define a variant of the line map:
\begin{definition}
For any pair $(a,b)$ of distinct points in $\CP^2$ the point $a\boxtimes b$ is the projection of the unique line orthogonal to both $a$ and $b$.
\end{definition}
As a direct consequence of the above discussion and definitions, we have
\begin{lemma}\label{lem:box-polar}
\begin{enumerate}
\item For any pair $(a,b)$ of distinct points of $\CP^2$, the line $\mathcal{L}(a,b)$ is polar to $a\boxtimes b$.
\item For any point $a$ in $\S^3$, the line $\mathcal{L}(a,a)$ is polar to $a$.
\end{enumerate}
\end{lemma}

By twisting the usual exterior product, one obtains a useful way of computing $a\boxtimes b$ (see also Section 2.2.7 of \cite{Goldman}). This will allow us, later on, to explicitly compute when working with the line map.
\begin{definition}\label{def:boxprod}
 Let $\ba$ and $\bb$ be two vectors in $\C^3$, and let $J$ be the matrix of the Hermitian form in the canonical basis. 
 We denote by $\ba \boxtimes \bb$ (the box product) the vector $\overline{J^{-1} \ba\wedge \bb}$
\end{definition}
Remark that the vector $\ba \boxtimes \bb$ is orthogonal to $\ba$ and $\bb$ : this follows directly from 
\begin{equation}\label{eq:boxtimes}
 \la X , \ba\boxtimes \bb\ra  =  X ^T J\cdot J^{-1} \ba\wedge \bb,
\end{equation}
which is clearly vanishing if $X=\ba$ or $X=\bb$. The vector $\ba\boxtimes\bb$ vanishes if and only if $\ba$ and $\bb$ are proportional. If $a\neq b$,  the point $\P(\ba\boxtimes\bb)$ is $a\boxtimes b$.

Computing with the box product is made easier by the following relations, that all come from standard identities for the usual exterior product. For any vectors $\ba,\bb,\bc,\bd \in \C^3$, we have (see also section 2.2.7 of \cite{Goldman})
\begin{eqnarray}
  \la \ba\boxtimes \bb,\bc\boxtimes \bd\ra & = & \la \bd,\ba\ra\la\bc,\bb\ra-\la\bc,\ba\ra\la\bd,\bb\ra\label{eq:box-expansion}\\
  \la \ba,\bb\boxtimes \bc\ra & = & \det(\ba,\bb,\bc)\label{eq:box-det}\\
  (\ba \boxtimes \bb)\boxtimes( \ba \boxtimes \bc) & = &\det(\ba,\bb,\bc)\cdot \ba\label{eq:box-det-vect}
\end{eqnarray}

\subsection{Geometry of $\HuuC$\label{section:Huuc}}

We will be most interested in this paper by $\C$-circles, which are intersections of a projective line meeting $\HdC$ with its sphere at infinity. As stated before, these lines are polar to points in $\HuuC$. We thus describe here some needed properties of the geometry of $\HuuC$, in particular with respect to $\C$-circles and polarity.

First, we can understand when two $\C$-circles meet, using polarity:
\begin{lemma}\label{lem:Ccircles-meet}
Let $x \neq y$ be two points of $\HuuC$. Then the $\C$-circles polar to $x$ and $y$ meet if and only if $x\boxtimes y \in \S^3$. Their intersection is then the point $x\boxtimes y$.

In particular, if $a\neq b$ and $c\neq d$ are four points in $\S^3$, not belonging to the same complex line, the $\C$-circle through $a,b$ and the one through $c,d$ meet if and only if $(a\boxtimes b)\boxtimes (c\boxtimes d)\in\S^3$.
\end{lemma}
\begin{proof}
The line polar to $x$ and $y$ meet in exactly the point $x\boxtimes y$. So the $\C$-circles meet if and only if this point belong to the sphere. The second part follows readily. 
\end{proof}

We denote by $\RP^2\subset \CP^2$ the projection to $\CP^2$ of $\R^3\subset \C^3$. As we have seen above, the intersection of $\RP^2$ with $\HdC$ is $\HdR$, and its intersection with $\S^3$ is the $\R$-circle $\partial\HdR$. No that $\RP^2$ is exactly the set of points in $\CP^2$ fixed by complex conjugation.

Restricted to $\R^3$, the Hermitian form gives a scalar product. This comes with a notion of polarity. This two notion are of course coherent:
\begin{lemma}\label{lem:RP2-Ccircle}
 The following are equivalent.
\begin{enumerate}
\item A point $m\in\HuuC$ belongs to $\RP^2$
\item The complex line $L_m$ polar to $m$ intersects $\HdR$ along a geodesic
\item The $\C$-circle $\partial L_m$ intersects $\partial \HdR$ in exactly two points.
\end{enumerate}
\end{lemma}
\begin{proof}
The last two items are equivalent since $\HdR$ is totally geodesic.

Assume $L_m$ intersects $\HdR$ along a geodesic $\gamma$, and pick two points $p,q\in\gamma$. Then $m=p\boxtimes q$ by definition. As $p$ and $q$ belong to $\RP^2$, they are fixed by complex conjugation and so does $m$. So $m$ belongs to $\RP^2$. 

Conversely, assume that $m\in\RP^2$, and pick a lift $\bmm$ with real coefficients. Then the orthogonal $\bmm^\perp$ intersects $\R^3$ along a 2 dimensional (real) vector subspace $V$. The projection of this subspace to $\CP^2$ intersects $\HdC$ along a geodesic which is contained in $\HdR$. 
\end{proof}

We can look at intersections of tangent lines to the sphere with $\RP^2$. Elementary projective geometry gives:
\begin{proposition}\label{prop:tangent-line-RP2}
Let $p\in \S^3\setminus\partial\HdR$. Then we have:
\begin{enumerate}
\item The complex line $\mathcal{L}(p,p)$ tangent to $\S^3$ at $p$ intersects $\RP^2$ in exactly one point. 
\item The point $m = \mathcal{L}(p,p)\cap\RP^2$ is polar to a complex line whose $\C$-circle intersects $\partial\HdR$ twice and contains $p$.
\end{enumerate}
\end{proposition}

\begin{proof}
Consider $\bar{p}$  the complex conjugate of $p$.  Two distinct complex lines intersect at only one point say $m=   \mathcal{L}(p,{p})\cap  \mathcal{L}(\bar{p},\bar{p})$.  This point is then fixed by complex conjugation and therefore $m\in \RP^2$. It is the only intersection point of $\mathcal L(p,p)$ with $\RP^2$. 

For the second part, observe that the line  $\mathcal{L}(p,\bar{p})$ is polar to $m$ as $m$ is in the tangent lines at $p$ and $\bar{p}$ simultaneously. It is invariant by complex conjugation, so belong to $\RP^2$.
One can then use the previous proposition to assert that its associated $\C$-circle intersects twice $\partial\HdR$.
\end{proof}

Given $a\neq b\in C$ and $C$ the $\C$-circle through $a,b$, we call {\it arc of $\C$-circle} a connected component of $C\setminus\{a,b\}$. The points $a$ and $b$ are called the endpoints of the arc. Two distinct points of $C$ define two arcs of $C$-circle supported by $C$. Note that these two arcs are naturally oriented by the orientation of $C$. This leads to the following notation:
\begin{definition}\label{def:arc}
Let $a\neq b$ be two distinct points of $\S^3$. We call \emph{arc from $a$ to $b$}, denoted by $\arc ab$, the portion of the $\C$-circle through $a$ and $b$ oriented from $a$ to $b$.
\end{definition}
These arcs are chains in the CR-geometry. They are the main character of our central object of study in this paper, that we describe below.

\subsection{A foliation by arc of $\C$-circles and unit tangent bundles\label{subs:foliation}}

The geometry explained above may be used to describe a well-known foliation of $\S^3\setminus \partial \HdR$ by arcs of $\C$-circles -- see \cref{fig:Ccircles-foliation}:

\begin{corollary}\label{coro:foliation}
Let $R$ be an $\R$-circle in $\S^3$. The set of arcs $\arc ab$ of $\C$-circles whose endpoints $a, b$ belong to $R$ defines a foliation of $\S^3\setminus R$.
\end{corollary}
\begin{proof}
Since $\PU(2,1)$ acts transitively on the set of real planes of $\HdC$, we may assume that $R=\partial\HdR$.  Now, let $p\in \S^3\setminus \partial\HdR$. A complex line contains $p$ if and only if it is polar to a point $n$ in the tangent complex line $\mathcal{L}(p,p)$, and intersects $\partial\HdR$ in two points if and only if $n\in\RP^2$. The result follows thus directly from Proposition \ref{prop:tangent-line-RP2}.

Moreover, if to such arcs $\arc ab$ and $\arc cd$ meet at $p\in \S^3\setminus R$, then the $\C$-circles supporting these arcs also meet at $\bar{p}$. Since $p\neq \bar{p}$, they have two common points, so the $\C$-circles are indeed the same one. This $\C$-circle has only two intersection point with $R$, by \cref{prop:tangent-line-RP2}. So $\{a,b\}=\{c,d\}$. As two opposed arcs of a single $\C$-circle are disjoint, we have $(a,b)=(c,d)$.
\end{proof}

The foliation given by \cref{coro:foliation} is well-known and  has the following nice geometric interpretation. It follows from the fact that $\S^3\setminus\partial\HdR$ is naturally homeomorphic to the unit tangent bundle of $\HdR$. This homeomorphism can be described as follows. We denote by $J$ the complex structure on $\HdC$. Viewing $\HdR$ as a subspace of $\HdC$, a point in $\unit{\HdR}$ is a pair $(p,\vec{u})$, where $p\in\HdR$ and $\vec{u}$ is a unit vector at $p$, tangent to $\HdR$. Then the homeomorphism is given by
\begin{eqnarray}
\phi & : & \unit{\HdR} \longrightarrow \S^3\setminus\HdR \nonumber\\
     &  & (p,\vec{u}) \longmapsto \gamma(p,-J_p\vec{u},+\infty),\nonumber
\end{eqnarray}
where $\gamma(p,J_p\vec{u},+\infty)$ is the point at infinity of the geodesic from $p$ in the direction $\vec{u}$ (see \cref{figure:unit-tangent-bundle}). Note that $\phi$ could be in fact defined on the whole $\unit{\HdC}$. If $(p,\vec{u})\in \unit{\HdR}$, the associated orbit of the geodesic flow is obtained by applying $J$ to all unit tangent vectors along the geodesic in $\HdR$ spanned and oriented by $(p,\vec{u})$. In particular, it is contained in the boundary of the complex line of $\HdC$ spanned by $p$ and $\vec{u}$. In fact, it is exactly the arc of $\C$-circle connecting 
$\gamma(p,J_p\vec{u},-\infty)$ and $\gamma(p,J_p\vec{u},+\infty)$ on which the natural orientation (given by the complex structure) coincide with the one given by $\vec{u}$. On \cref{figure:unit-tangent-bundle}, this arc is the "lower" one.

This foliation and its link to the unit tangent bundle $\unit{\HdR}$ may in turn be used to understand CR-spherical structures on unit tangent bundles of hyperbolic surfaces.
Let us consider an $\R$-Fuchsian group $\Gamma\subset \PO(2,1)\subset \PU(2,1)$ acting on $\HdR$. We note that the limit set of such a group is $\Lambda_\Gamma=\partial\HdR$, and its disconinuity region is $\Omega_\Gamma = \S^3\setminus \partial\HdR$. The homeomorphism $\phi$ constructed above is clearly $\PO(2,1)$-equivariant, thus  $\phi$  descends to the quotient, and one easily obtain the following classical result -- see for instance  \cite[Proposition 2.7]{Goldman-Kapovich-Leeb}, where the same result is obtained by considering Euler numbers of circle bundles.

\begin{proposition}\label{R-Fuchsian-UTB}
Let $\Gamma\subset \PO(2,1)\subset \PU(2,1)$ be a Fuchsian group, preserving $\HdR$. 
Then $\phi$ realizes an homeomorphism between the unit tangent bundle $\unit{(\Gamma\backslash\HdR)}$ and the quotient of $\Omega_\Gamma$ by the action of $\Gamma$.

Moreover, in the cover, the map $\phi$ sends orbits of the geodesic flows in $\unit{\HdR}$ to arcs $\arc ab$ where $a\neq b\in \partial \HdR$. 
\end{proposition}

\begin{figure}

\begin{tabular}{lc}

 \begin{minipage}[l]{.46\linewidth}
\includegraphics[width=\textwidth]{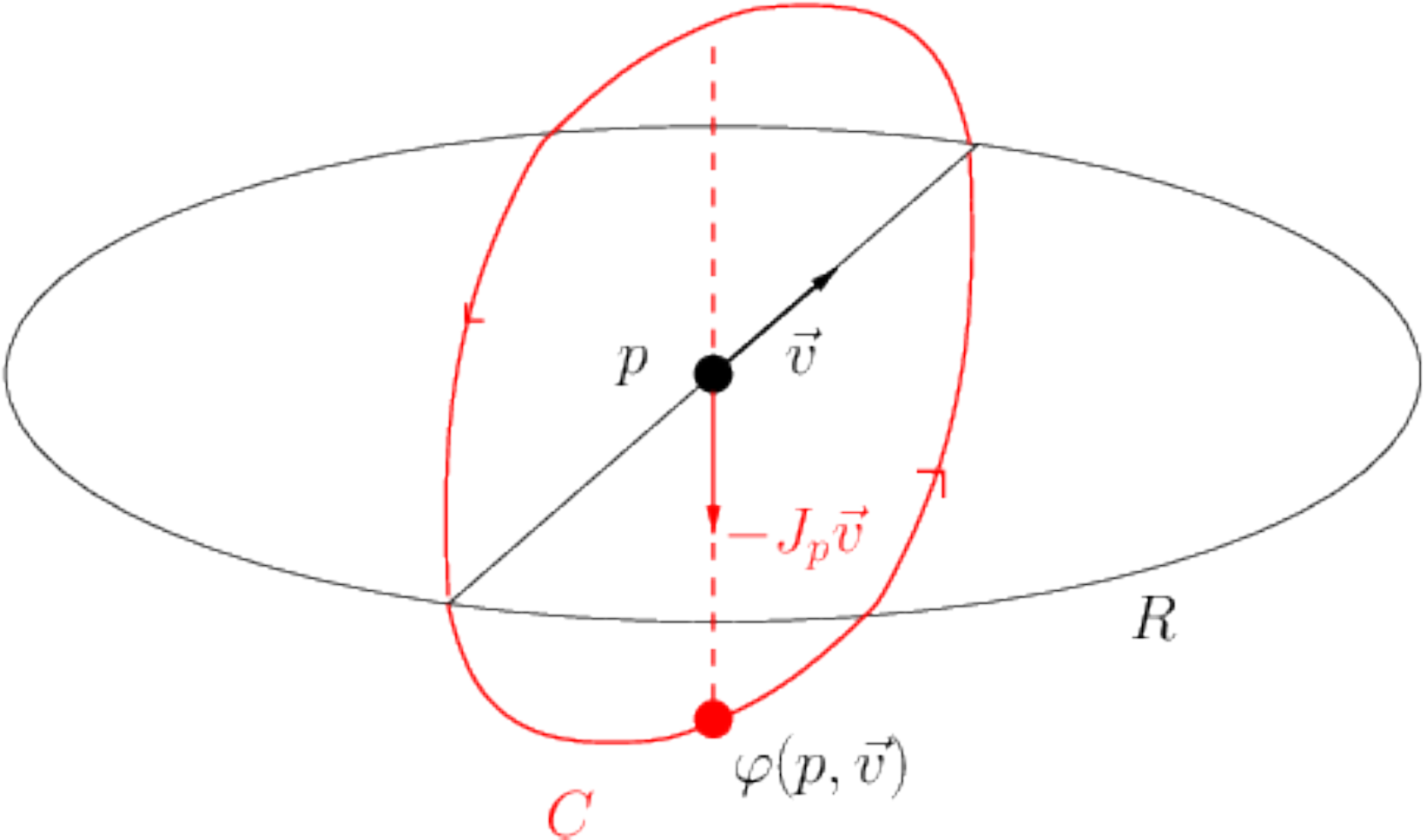}
   \end{minipage} \hfill
   &
   \begin{minipage}[c]{.46\linewidth}
\includegraphics[width=\textwidth]{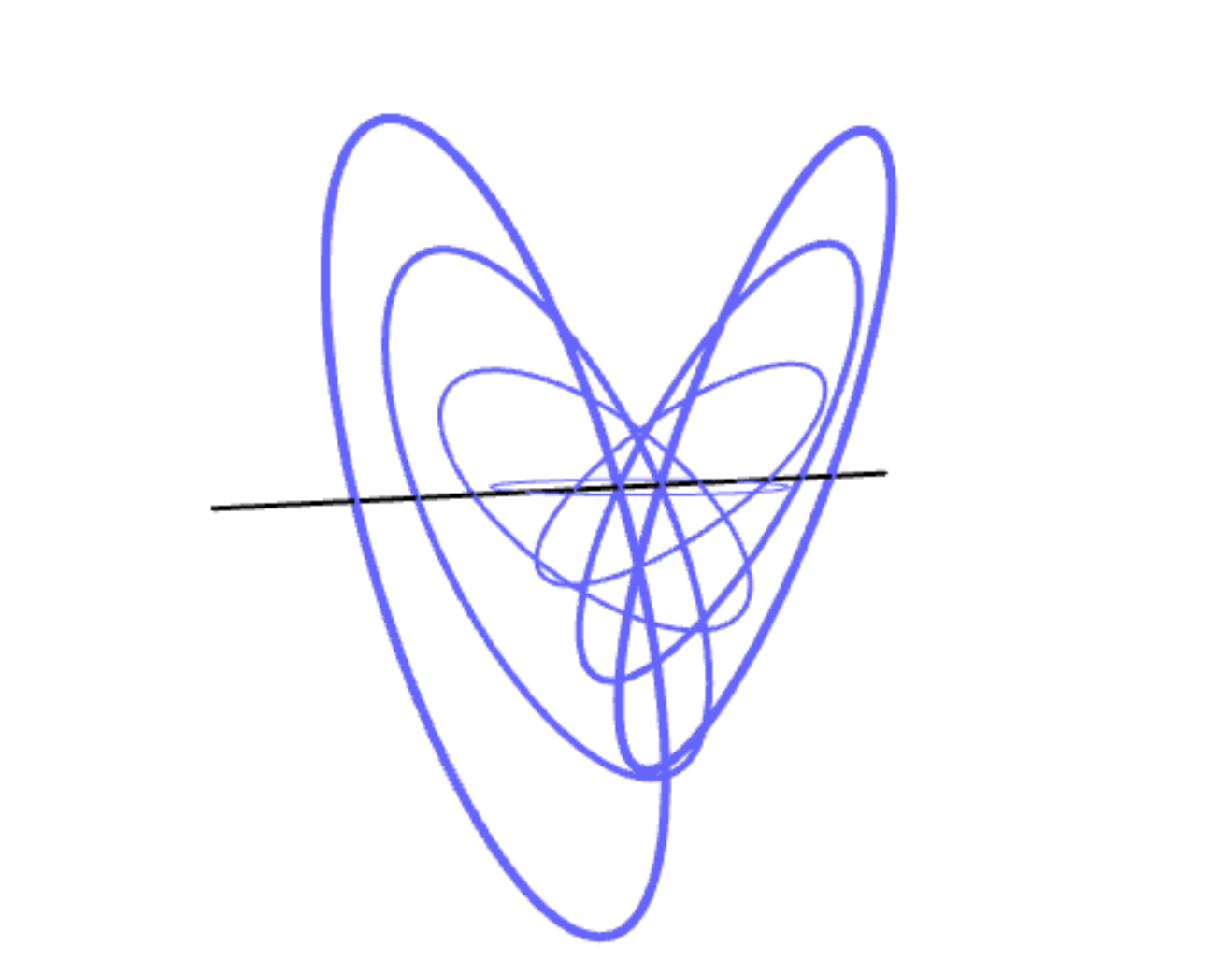}
   \end{minipage}\\

\begin{minipage}[l]{.46\linewidth}   
\caption{Identification between the complement in $\S^3$ of an $\R$-circle $R$ and the unit tangent bundle of $\HdR$ \label{figure:unit-tangent-bundle}} 
\end{minipage}

&

\begin{minipage}[l]{.46\linewidth}
\caption{A few leaves of the foliation of $\S^3\setminus\partial\HdR$ by arcs of $\C$-circles (see Corollary \ref{coro:foliation}) \label{fig:Ccircles-foliation}}
\end{minipage}   

\end{tabular}

\end{figure}


The main goal of this paper is first to understand what happens to this foliation
 when deforming the $\R$-circle and second, in presence of a group acting, to 
understand how the map $\phi$ deforms. But, if we consider any deformation, no 
meaningful description can be given. So, we first define in 
\cref{sec:horizontality-slimness} a notion of horizontality for non-smooth curves
 and a quantitative version called \emph{slimness}. Under these conditions we will
 be able to understand how the foliation deforms in \cref{section:deform-foliation}
 and understand better the equivariant case in \cref{sec:Crowns}.

\section{Horizontality and Slimness in the CR-sphere}\label{sec:horizontality-slimness}

From now on, we focus especially on the sphere at infinity $\S^3 = \bHdC$. So we will rather work with $\R$- and $\C$-circles than complex hyperbolic lines and real hyperbolic planes. All properties stated with $\R$ and $\C$-circles can be equivalently stated with their supporting complex lines and real planes.

Recall from \cref{section:orbits} that the sphere  comes with its contact structure, and the notion of \textsl{horizontality} for smooth submanifolds, see \cref{def:CR-horizontal}.
Horizontality will be one of the main concern of this paper. However, the typical sets we want to describe are limit sets of discrete subgroups 
of $\PU(2,1)$. Those subsets are not usually smooth. In fact, they are smooth 
iff it is the whole sphere or the group is not Zariski-dense.
So we have 
to devise a notion of horizontality suitable for general non-smooth subset of $\S^3$.

\subsection{Horizontality for non-smooth subsets of $\S^3 \subset \CP^2$}

One property of horizontal submanifolds can be expressed in the following way: any $\C$-circle through two close points is entirely contained in a small neighborhood of the two points. We could try and write a definition using the topology on the set of $\C$-circles. The problem is that this set is not compact, as $\C$-circles can degenerate to a single point. And we exactly want to use this degeneration: the $\C$-circle between two points in a horizontal submanifold degenerates as the two points collapse. With the previous section in mind, it appears more natural to take a step back
and work instead in $\CP^2$ and use the line map. Indeed, a $\C$-circle is defined by a unique line in $\CP^2$. Moreover,
when a family of $\C$-circles converges to a single point $p \in \S^3$, then the family of associated lines converges to the line $p^\perp$ in $\CPstar$ and their polar converge to $p$.

Recall that the restriction of the line map $\mathcal L$ to $\S^3\times \S^3$ is defined in \cref{def:line-map} by: for $e\neq f \in\S^3$, $\mathcal L(e,f)$ is the line through $e$ and $f$ and $\mathcal L(e,e)$ is the line $e^\perp$. Using polarity, we will write equivalently $\mathcal L(e,f)=e\boxtimes f$ and $\mathcal L(e,e)=e$.
Note in particular that this map is not continuous at the diagonal, see \cref{prop:L-not-continuous}.
We thus propose the following definition:
\begin{definition}\label{def:horizontality}
Let $E\subset \S^3$ be a closed subset. We say that $E$ is \emph{CR-horizontal} if the restriction $\mathcal L_E$ of
the line map $\mathcal{L}$ to $E\times E$ is continuous.
\end{definition}
Away from the diagonal, the map $\mathcal L_E$ is always continuous. So the previous definition is in fact a local property. One could restate the continuity hypothesis by asking that for any $e\in E$, for any sequences $(e_n\neq f_n)$ converging to $e$, the sequence of lines $(e_n f_n)$ converges to $e^\perp$.

We recover in the smooth case the usual definition:
\begin{proposition}
A submanifold $E$ is CR-horizontal iff it is horizontal.
\end{proposition}

\begin{proof}
Consider the tangent space to the submanifold $E$ at some point $e \in E$. This tangent space contains a vector $v$  iff there are sequences $(e_n), (f_n)$ of points in $E$, that converges to $e$ and such that the real line $(e_nf_n)$ converges to the real line $l$ through $e$ containing $v$. The vector $v$ belongs to the contact plane iff the line $l$ is included $p^\perp$. After tensorizing by $\C$, $v$ belongs to the contact plane iff the complex line $(e_nf_n)$ converges to $p^\perp$.
This proves the proposition.
\end{proof}

\begin{remark}
 In contact geometry and, more generally, in CR-geometry,   horizontal paths are usually defined as absolutly continuous paths with tangent vectors in a fixed distribution.   Here we don't need existence of derivatives.  On the other hand we are using the extrinsic geometry of the CR-structure of the sphere embedded in $\CP^2$.  
 An intrinsic way to define horizontality for arbitrary CR-structures is to use the special  paths called chains.   We are saying that $E$ is horizontal if for any converging sequence of points in $E$ the directions defined by chains between the limit and the points in the sequence converges to a direction in the contact plane.
 \end{remark}

\begin{remark}
Thanks to the previous definition, our notion of CR-horizontal manifold is an extension of the usual notion of horizontality to non-smooth sets. So from now on, we will drop the specification of "CR-" and just call this notion \emph{horizontality}.
\end{remark}

The following lemma translates in coordinates the local condition at a point $p$: at first order, points arrive at $p$ along the orthogonal line $\mathcal L(p,p)$ or equivalently tangentially to the contact structure. Recall from \cref{def:boxprod} that ${\bm p}\boxtimes {\bm q}$ is a specific lift of $p\boxtimes q$, even if in the following statement any lift could be used.
\begin{lemma}\label{lem:local}
Let $E$ be a horizontal subset of $\S^3$ containing a point $p$. Let $q$ be another point in $\S^3$. Fix lifts $\bp$ and $\bq$ for $p$ and $q$. Then any point $a\in E\setminus \{q\}$ admits a unique lift to $\C^3$ of the form
$$\ba = \bp + x \bp\boxtimes \bq + y\bq.$$
Moreover $y=o(x)$ in a neighbourhood of $p$
\end{lemma}

\begin{proof}
Any point in $\CP^2$ has a unique lift of this kind but for points in the line through $q$ and $p\boxtimes q$. This line is orthogonal to $q$, so is $\mathcal L(q,q)$. Its intersection with the sphere is $q$, so the first point of the lemma is proven.

Suppose now by contradiction that there is a sequence of points $a_n$ in $E$, converging to $p$, with lifts $\ba_n = \bp + x_n \bp\boxtimes \bq + y_n \bq$, such that $y_n$ is not negligible in front of $x_n$. Up to passing to a subsequence, it implies that there exist $x_\infty \in\C$, $y_\infty \in \C\setminus\{0\}$ and a sequence $t_n \to 0$ of positive real numbers such that, at the first order, $\ba_n = \bp + t_n (x_\infty \bp\boxtimes \bq + y_\infty \bq) +o(t_n)$. Then the polar $p\boxtimes a_n$ to the line $\mathcal L(p,a_n)$  can be computed with box-product:
\begin{align*}
    \bp \boxtimes \ba_n & = \bp \boxtimes (\bp + t_n (x_\infty \bp\boxtimes \bq + y_\infty \bq)) + o(t_n)\\
    & = t_n (x_\infty \bp\boxtimes(\bp\boxtimes \bq) +y_\infty \bp\boxtimes \bq)+o(t_n)
\end{align*}
This implies that $p\boxtimes a_n$ converges to the projection of $x_\infty \bp\boxtimes(\bp\boxtimes \bq) +y_\infty \bp\boxtimes \bq$.
Note that $\bp\boxtimes(\bp\boxtimes \bq)$ is a multiple of $\bp$. So $p\boxtimes a_n$ is a point in $p^\perp$, different from $p$ as $y_\infty\neq 0$. This contradicts the horizontality condition on $E$: $p\boxtimes a_n$ should converge to $p$ as $a_n$ goes to $p$.
\end{proof}

The non-continuity of $\mathcal{L}$ at the diagonal has an 
interesting consequence in terms of Cartan invariant. We 
will make a repeated use of the following lemma:

\begin{lemma}\label{lem:Cartan-coalesce}
 Let $e$ be a point in $\S^3$, and $(e_n)$, $(f_n)$ be two sequences of points both converging to $e$, and such that such that $\mathcal{L}(e_n,f_n)$ converge to some complex line $\ell\neq e^\perp$. Then, for any point $x \in \S^3$ distinct from $e$, it holds that $|\A(x,e_n,f_n)|\longrightarrow \pi/2$.
 \end{lemma}
 
\begin{proof}
 It suffices to prove the result for $e_n=e$.  Indeed, one can make $x=\infty$ and $e_n=0$ in Heisenberg coordinates
 by translating through $\PU(2,1)$. We have $f_n\to e$ and we suppose by contradiction that $|\A(x,e,f_n)|$ does not converge to $\pi/2$.  There is a subsequence of $\A(x,e,f_n)$ converging to $A\in ]-\pi/2,\pi/2[$. Using standard lifts, one write  in coordinates the corresponding subsequence of $f_n$  as 
 \[{\bm f_n} = \begin{bmatrix}|\lambda_n|^2(-1+ ia_n)\\\lambda_n\\1\end{bmatrix}\]
 where $a_n = \A(x,e,f_n) \to \tan{A}$ and $\lambda_n\to 0$ as $f_n\to e$.
 Now the line between $e=0$ and $f_n$ is generated by two points whose lifts are: 
 \[ {\bm e} = \begin{bmatrix}0\\0\\1\end{bmatrix} \quad \ \mathrm{ and }\ \ {\bm f_n}=\begin{bmatrix}|\lambda_n|^2(-1+ ia_n)\\\lambda_n\\0\end{bmatrix} =  \lambda_n \begin{bmatrix}\bar{\lambda_n}(-1+ ia_n)\\1\\0\end{bmatrix}.\]
 The complex lines determined by these vectors converge to the line $e^\perp$, which is a contradiction with the assumption.
\end{proof}

\subsection{Characteristic foliation and one-parameter subgroups}\label{subsec:horizontality-orbits-1}
A common construction of smooth horizontal curves is through the characteristic foliation
induced in an embedded surface in a contact manifold.  One defines that foliation (generically with singular points)  by the field of directions given by the intersection of the contact plane with the tangent space of the surface.

For example, if the surface is the complex plane through $0$ in the Siegel model union $\infty$, that is a sphere, the foliation has two singularities, at $0$ and $\infty$, where the contact plane is tangent to the surface. Closure of leaves of this foliation are exactly the half-lines between $0$ and $\infty$, so half of $\R$-circles. One can glue two such half lines at the two singularities, obtaining what we call \emph{bent $\R$-circles}. We will come back on this example in \cref{section:bent}.

Due to homogeneity, two-dimensional orbits of a group  $G\subset \PU(2,1)$ acting on  $\partial_\infty \HdC$ have a characteristic foliation with no singular points.  It is interesting to determine those orbits whose foliations contain orbits of one-parameter subgroups $H\subset G$.

There are four types of two dimensional orbits which correspond to the groups $\C^*$, $\R^2$, $\S^1\times \R$ and $\S^1\times \S^1$ and they are unique up to conjugation in $\PU(2,1)$.  The first one corresponds to the group of all loxodromic elements fixing two points in 
$\partial_\infty \HdC$.  The second one corresponds to a maximal abelian subgroup contained 
in the Heisenberg group, the third one to the group generated by vertical translations in the Hesisenberg group and rotations fixing  its vertical axis and the fourth one to a maximal torus.

Each group $G$ acts transitively on the leaves  of the characteristic foliation because it preserves the contact plane and therefore if one of the leaves is an orbit of a one parameter subgroup of $G$, the same holds for every other leaf.  Observe that $G$ will be the normalizer of each of its one-parameter subgroups.

We prove here that for the loxodromic case, all leaves are indeed orbits of a one-parameter subgroup.
\begin{proposition}\label{pro:paraboloids-1-parameter}
Let $L$ be a one-parameter loxodromic subgroup, normalized by a two-parameter subgroup $G$.

Then there are exactly two surfaces invariant under $G$ where the characteristic foliation is given by one-parameter orbits of $L$.
\end{proposition}

\begin{proof}
We work in the Siegel model. Loxodromic 1-parameter subgroups are conjugate to one given by 
\begin{equation}\label{eq:1-param-loxo}
L_\alpha : s\longmapsto \begin{bmatrix} e^{s\alpha} & & \\ & e^{s(\overline{\alpha}-\alpha)} & \\ & & e^{-s\overline{\alpha}} \end{bmatrix},\mbox{where $\alpha\in\C^*$ satisfies $|\alpha|\neq 1$.}
\end{equation}
 The 1-parameter subgroup $L_\alpha$ fixes $0$ and $\infty$. Now, if 
$p_0=[z_0,t_0]$ is a point in $\partial_\infty \HdC$, represented by the null vector $\bp = [ -|z_0|^2+it_0, z_0 ,1 ]^T$, then the orbit of $p_0$ under $L_\alpha$ is the curve $[z_0e^{s(2\bar\alpha-\alpha)},t_0e^{2\Re(\alpha)}]$, which is a spiral inscribed on the pararaboloïd with equation 
$t |z_0|^2=|z|^2t_0\cdot$. The normalizer of $L_\alpha$ consists of those elements of $\PU(2,1)$ that lift to diagonal matrices ${\rm diag}(\lambda,\bar\lambda/\lambda,1/\bar\lambda)$. It is easy to see that this normaliser acts simply transitively on any sheet of the paraboloïd.

Next, we remark that if $p$ is a point in $\bHdC$ and $s\longmapsto\gamma(s)$ is the orbit of $p$ under a 1-parameter subgroup $(G_s)_s$ of $\PU(2,1)$, then $\gamma$ is horizontal if and only if
  \begin{equation}  \label{eq:curve-horizontal}\Bigl\la\dfrac{d}{ds}\Bigl\vert_{s=0}\gamma(s),\bp \Bigr\ra=0.
  \end{equation}
  Indeed, the contact plane at $p$ is the kernel of $\la\cdot,\bp\ra$, and horizontality is preserved by the action of the 1-parameter subgroup. 
  
Consider now the point in $\bHdC$ given in Heisenberg coordinates by $p=[z,t]$. Then  using \eqref{eq:curve-horizontal}, we see that the orbit of $p$ is horizontal if and only if
   \begin{equation}\label{eq:1-param-loxo-horizontal}\dfrac{t}{|z|^2}=3\dfrac{\Im(\alpha)}{\Re(\alpha)}.\end{equation}
   This proves the existence part. The above condition on $t$ and $|z|$ defines a (two-sheeted) paraboloïd in the Heisenberg group, which is acted on uniquely transitively by the normaliser $G$ of $L$. This proves the uniqueness part.
\end{proof}

This analysis raises the natural question of which one-parameter subgroups do admit horizontal orbit. This question diverts us from our main goal, so we will answer it in another forthcoming paper.

\subsection{Slimness in the sphere}\label{sec:slimness}

The notion of horizontality captures a local property of subsets of $\S^3$. We  need a quantitative and global version of this property.  
We define the relevant notion in this section using the Cartan invariant to guarantee quantitatively that no $\C$-circle hits thrice the subset. We first single out this global property which is independent of horizontality. This is a particular case of one of the central notion of the theory of Anosov representations \cite{Labourie, PozzettiSambarinoWienhard}.

\begin{definition}\label{definition:hyperconvex}
We say a subset $E\subset S^3$ is hyperconvex if no three points in $E$ are contained in the same
complex line. 
\end{definition}

\subsubsection{A quantitative version of horizontality}

If $E$ is a subset of $\S^3$, we denote by $E^{3}$ the set of triples in $E$ and by 
$E^{(3)}$ the its subset of pairwise distinct triples.
We define now the central notion of this paper:
\begin{definition}\label{def:slimness}
Let $0\leqslant \alpha < \dfrac{\pi}{2}$. We call {\it $\alpha$-slim} any subset $E$ of $\S^3$ such that the absolute value of the Cartan invariant of any triple of  points is bounded above by $\alpha$ : 
\begin{equation}
\sup\left\{|\A(a,b,c)|,\, (a,b,c) \in E^{3} \right\} \leqslant \alpha
\end{equation}

As a shortcut, we say that a subset $E$ is \emph{slim} if it is $\alpha$-slim for some $\alpha < \frac{\pi}{2}$. Moreover, we denote by $\A(E)$ the supremum \( \A(E) = \sup\left(|\A(a,b,c)|, (a,b,c) \in E^{3} \right).\)
\end{definition}

This condition will prove to be a strong constraint
on subsets of $\S^3$. 
A first point to be proven is that this assumption indeed implies horizontality - see \cref{prop:alphaslim_horizontal}. Some preliminary remarks, though, follow from the definition.

\begin{remark}\label{rem:definition}
\begin{enumerate}

\item If a subset $E$  is $0$-slim, then it is included in an $\R$-circle.

\item If a subset $E$  is slim, then its intersection with any $\C$-circle has cardinality at most $2$: if $E$ had $3$ points on a $\C$-circle, then this triple of points would have Cartan invariant $\pi/2$ and thus we would have $\A(E)=\pi/2$.

\item One could likewise define the notion of $\alpha$-thickness by asking that every Cartan invariant have absolute value at least $\alpha$. A $\frac{\pi}{2}$-thick set is then included inside a $\C$-circle.

\item If $t\to E_t$ is a Hausdorff-continuous family of closed subsets in $\S^3$
then $t\to \A(E_t)$ is upper semi-continuous. Is is not continuous in general, cf \cref{subsec:Farey}.

\item
The Cartan (measurable bounded) cocycle  $\A$  determines a bounded cohomology class on 
$\PU(2,1)$ which coincides with the continuous bounded K\"ahler class $\kappa$. Let $\rho : \Gamma \to \PU(2,1)$ be a representation and $\rho^*(\kappa)\in H^2(\Gamma,\R)$ the corresponding bounded cohomology class of $\Gamma$.   One should observe that if $E=\Lambda_\rho$ is the limit set  of $\rho(\Gamma)$, we obtain that the Gromov norm of this class coincides with  the supremum of the Cartan cocycle restricted to the limit set (see Proposition 3.1 in \cite{BI}, see also \cite{BI2012}).  That is 
$$\vert\vert \rho^*(\kappa)  \vert\vert =  \A(E).$$

\end{enumerate}
\end{remark}

We now use \cref{lem:Cartan-coalesce} to prove that $\alpha$-slimness implies horizontality. Though in a different setting, the following proposition is almost the same as \cite[Theorem B]{PozzettiSambarinoWienhard}:

\begin{proposition}\label{prop:alphaslim_horizontal}
If $E$ is a slim subset of $\S^3$, then it is horizontal.
\end{proposition}

\begin{proof}
Let $E$ be a slim set. 
Suppose that it is not horizontal. This means there exist a point $e$ and two sequences $e_n \neq f_n$ converging to $e$ such that the sequence of lines $(e_n f_n)$ in $\CP^2$ converges to a line $l$ with $l\neq e^\perp$. 

Note that $E\cap l$ contains at most two points, as noted in the second point of the previous \cref{rem:definition}. So we fix an arbitrary $x\in E\setminus l$. By \cref{lem:Cartan-coalesce}, we have $|\A(e_n,f_n,x)| \to \dfrac{\pi}{2}$. 
The assumption that $(e_nf_n) \to l \neq p^\perp$ implies therefore that the sup of Cartan invariant is $\dfrac{\pi}{2}$ so contradicts the slimness assumption.
\end{proof}
For a submanifold, though, slimness implies that it is a Legendrian smooth curve:
\begin{corollary}
A connected slim submanifold of $\S^3$ is a smooth Legendrian curve.
\end{corollary}
This applies as well as for absolutely continuous paths : their tangent vectors are horizontal wherever they are defined.

\subsubsection{Local picture in Heisenberg space}

A first geometric way of seeing the horizontality of a slim subset is the following. Assume $E$ is $\alpha$-slim and the points $\bo = [0,0]$ and $\infty$ belong to $E$. Then any point in $E$ satisfies $|\A(\infty,\bo,p)|\leqslant \alpha$. In Heisenberg coordinates the point $p$ is $[z,t]$. Lifting the three points of $\C^3$, we have
\begin{equation}
 \bf{\infty} = \begin{bmatrix} 1 \\ 0 \\ 0\end{bmatrix},\, \bo = \begin{bmatrix} 0 \\ 0 \\ 1\end{bmatrix}, \bp = \begin{bmatrix} -|z|^2+it \\ z \\ 1\end{bmatrix} \mbox{ where } z\in \C \mbox{ and } t\in \R.
\end{equation}
A straightforward computation leads directly to
\begin{equation}\label{eqn:paraboloids}
 |\A(\bf{\infty},\bo,\bp)| \leq \alpha \quad \Leftrightarrow \quad |t|\leqslant  \tan(\alpha) |z|^2
\end{equation}
This means that $E\setminus \{\bo,\infty\}$ is contained in the complement of the union of the two open solid paraboloids defined by $|t|>\tan(\alpha)|z|^2$. We denote by $\mathcal{P}_\alpha$ this region.

\begin{figure}
\begin{center}
\includegraphics[width=.5\textwidth]{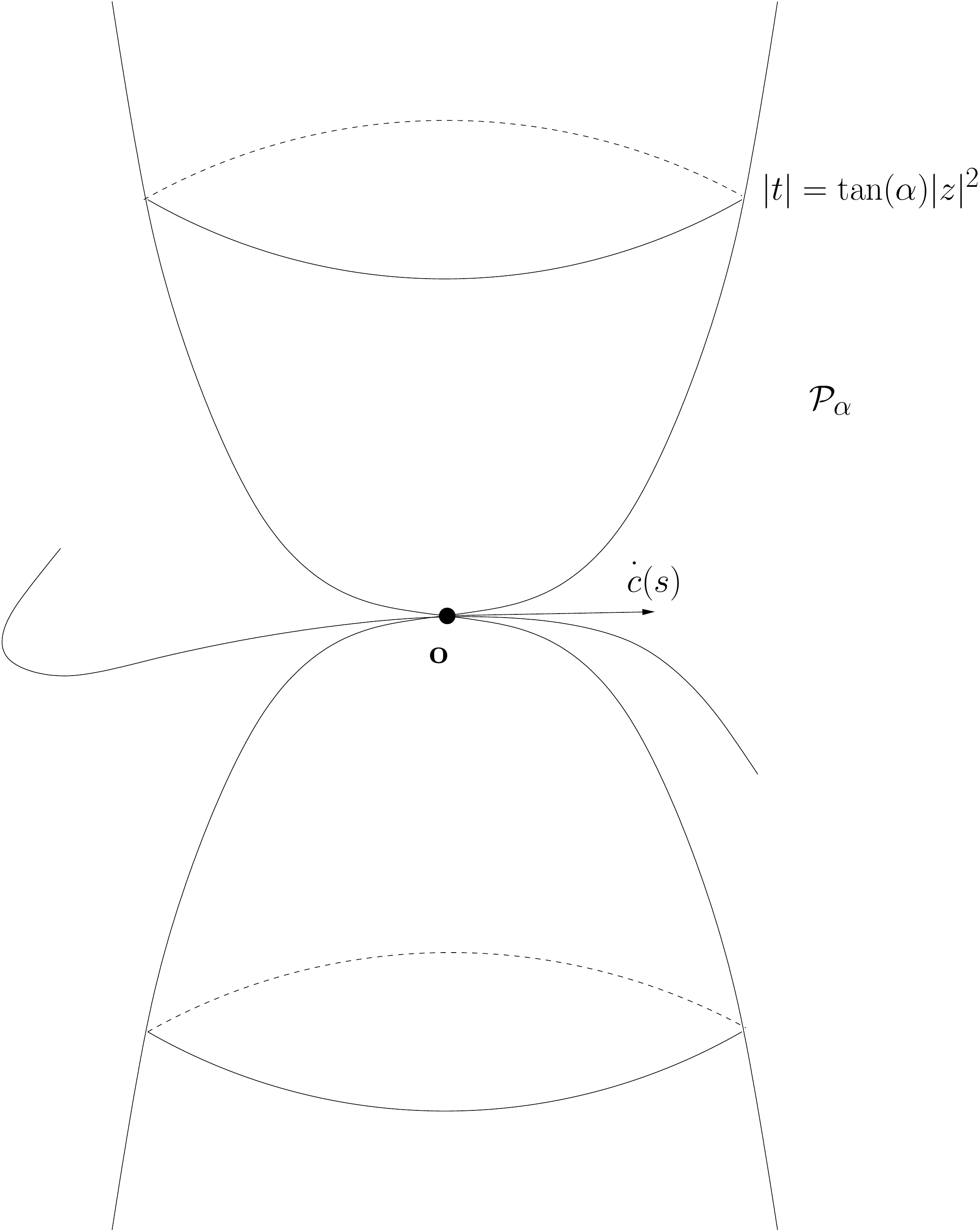}
\end{center}
\caption{Local aspect of a $\alpha$-slim curve close to the point $[0,0]$\label{alphaslim-curve}. }
\end{figure}

The paraboloïds are illustrated in \cref{alphaslim-curve}. It makes visible that, if $E$ is a slim submanifold of $\S^3$, then it is CR-horizontal (see \cref{prop:alphaslim_horizontal}). 
Moreover, Condition \eqref{eqn:paraboloids} may be used to enhance \cref{lem:local} under the hypothesis of slimness rather than horizontality: in this case and with the notation of the lemma, $y$ is easily seen to be in fact $O(|x|^2)$ instead of $o(x)$. We will not use this fact here, so we do not go into details.

\subsection{Projections of slim subsets}\label{sec:projections}

Before going on with the properties of slim subsets, one can give a few geometric interpretations of the slimness condition. 

In this section, we fix an $\alpha$-slim subset of $\S^3=\bHdC$. Recall that for $a \neq b$ in $\S^3$, the complex line through $a$ and $b$ is denoted by $\mathcal L(a,b)$, and that $\mathcal{L}(a,a)$ is the complex line tangent to $\S^3$ at the point $a$. We are now going to interpret the $\alpha$-slimness condition in terms of three different projections.


\subsubsection{Orthogonal projection on the hyperbolic disc through $a$ and $b$ in $E$.}
This interpretation explains the terminology: the maximum of the Cartan invariant is a measure of the width of the the projection of the set on hyperbolic discs defined by couple of points in the set. Let $(a,b,c)$ be a triple of pairwise distinct points in $\bHdC$. Denote by $L_{a,b}$ (resp. $\gamma_{a,b}$) the complex line (resp. the geodesic) spanned by $a$ and $b$. Let $\Pi_{a,b}$ be the orthogonal projection from $\S^3\setminus \{a,b\}$ onto $L_{a,b}$. Then the Cartan invariant of $(a,b,c)$ satisfies the following relation (see Theorem 7.1.2 in \cite{Goldman}).
\begin{equation}\label{eq:Cartan-interpretation}
\sinh^{-1}\left(\vert \tan(\A(a,b,c))\vert\right) = d(\Pi_{a,b}(c),\gamma_{a,b}).
\end{equation}

\begin{figure}
\begin{center}
\scalebox{0.5}{\includegraphics{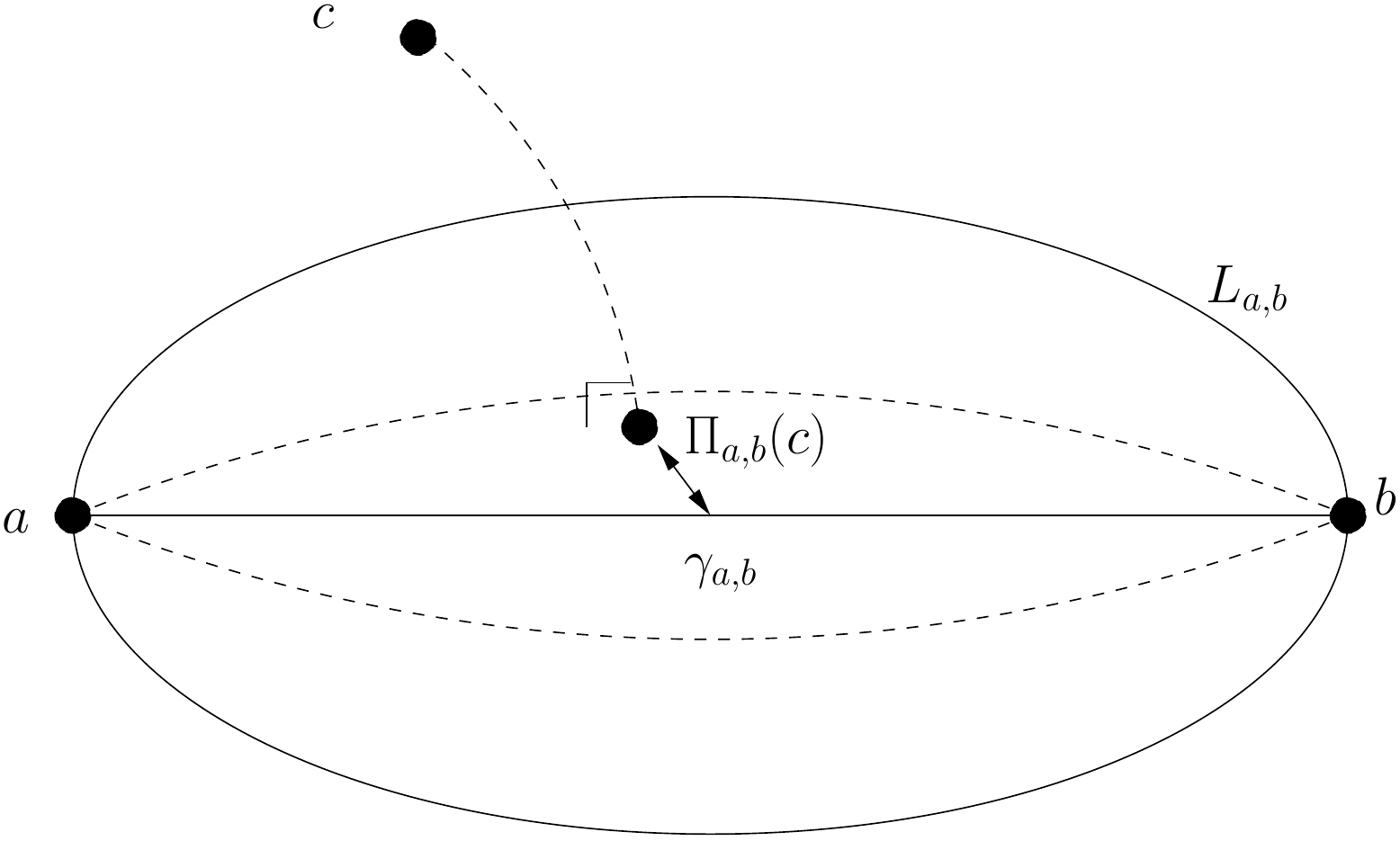}}
\end{center}
\caption{Orthogonal projection onto the complex line $L_{a,b}$. The region bounded by dashed lines in $L_{a,b}$ is a $k$-tubular neighbourhood of $\gamma_{a,b}$.}\label{fig:paraboloids}
\end{figure}

Denote by $k_{\alpha}$ the number $\sinh^{-1}\left(\tan(\alpha)\right)$. Observe that $k_\alpha=\alpha +o(\alpha^2)$. We obtain directly the following proposition, illustrated in \cref{fig:paraboloids}.  
\begin{proposition}
Let $E$ be an $\alpha$-slim subset of $\bHdC$. Then

\begin{equation}
 E \subset \bigcap_{a,b\in E,a\neq b} (\Pi_{a,b})^{-1}\left(N^{k_\alpha}(\gamma_{a,b})\right),
\end{equation}
where $N^{k_\alpha}(\gamma_{a,b})$ is the $k_\alpha$-tubular neighborhood in $L_{a,b}$ of the geodesic $\gamma_{a,b}$.
\end{proposition}

 In coordinates, we can always assume that $a=[0,0]$ and $b=\infty$ in the Heisenberg group, so that the complex line through $a$ and $b$ corresponds to those vectors $[z, 0, 1]^T$, where $\Re(z)<0$. The orthogonal projection of a point $c = [z_1 , z_2 , 1]^T$ in $\S^3$ is just $[z_1 , 0 , 1]^T$, and the Cartan invariant is $\A(a,b,c)=\arg(-z_1)$. Thus the projection of $c$ belongs to the cone $\arg(z)\leqslant \alpha$

\subsubsection{Projection on the complex line in $\CP^2$ through $a$ and $b$ in $E$}
\label{geom_interpretation_3}

We can visualize the former property in the whole complex line $\mathcal L(a,b)$ in $\CP^2$, for $a$, $b$ in $E$, and link it to another projection of $E$, defined as follows.
\begin{equation}
\begin{matrix} \Pi_{a,b}^*: & \S^3\setminus \{a,b\} &\to & \mathcal L(a,b)\\
& e & \mapsto& \mathcal{L}(e)\cap \mathcal{L}(a,b)
\end{matrix}
\end{equation}
Note that this intersection point is always outside the ball $\HdC$.
This projection is more adapted to projective geometry and inspired by considerations for \emph{generalized Hilbert distance} in \cite{FalbelGuillouxWill}. We will see shortly that our two projections $\Pi_{a,b}^\perp$ and $\Pi_{a,b}^*$ are closely related, but they serve different purposes. The former is well-adapted to arguments relating to hyperbolic geometry. 
The latter, instead, will give information on the large scale Hilbert geometry of the complement of a slim set which we will pursue in a forthcoming work.

Here again, we can chose coordinates so that $a=[0,0]$ and $b=\infty$. If 
$e=[z_1,z_2,1]^T$ is a point in $\S^3$, the intersection of $\mathcal{L}(e)$ and $\mathcal{L}(a,b)$ is the point $[ -\overline{z_1}, 0 , 1]^T$.


We obtain thus a link between $\Pi$ and  $\Pi^*$:
\begin{lemma}
We have $-\Pi_{a,b} = \overline{\Pi_{a,b}^*}$.
\end{lemma}
The previous discussion translates into the following 

\begin{figure}
\begin{center}
\scalebox{0.5}{\includegraphics{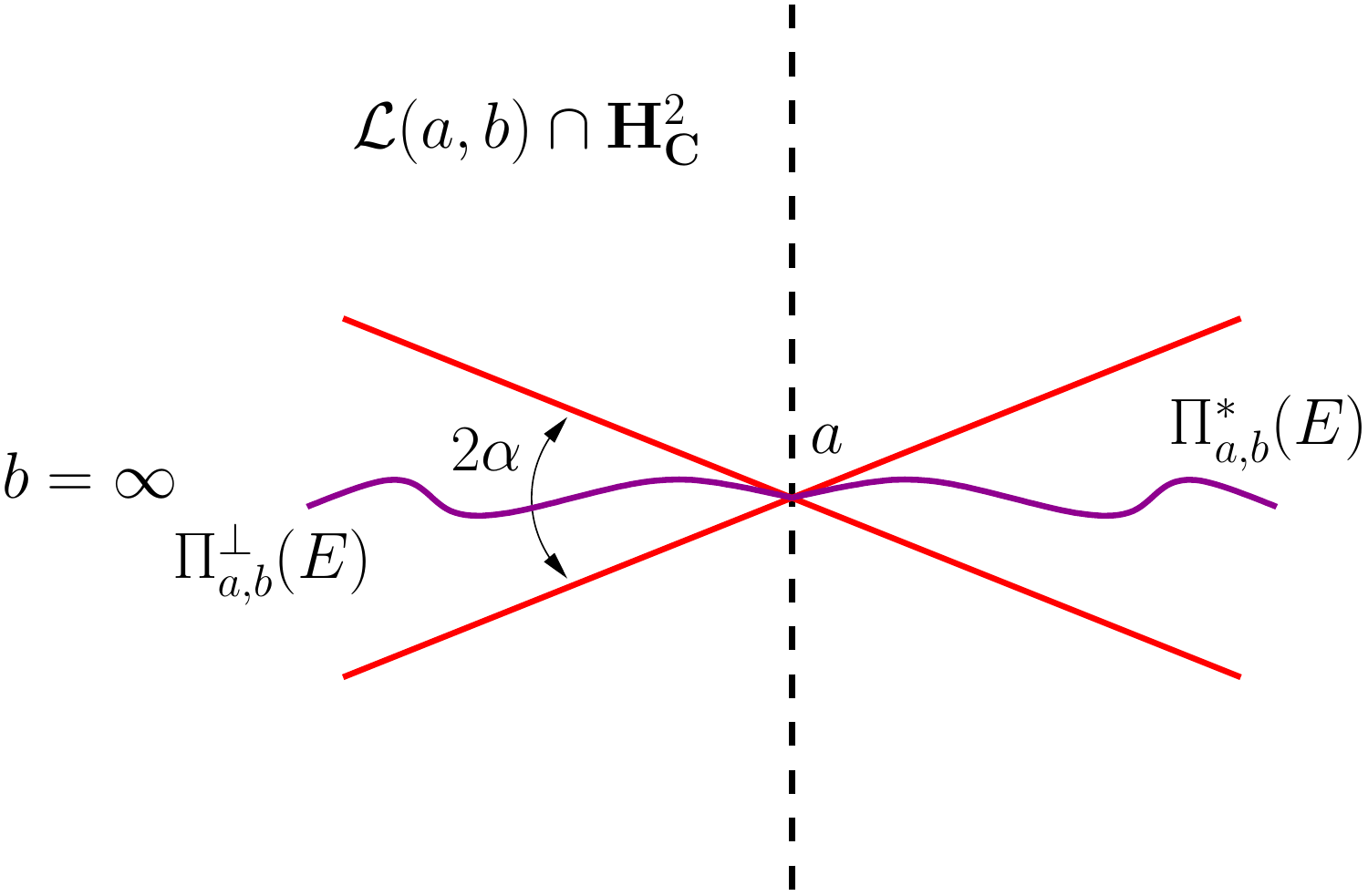}}
\end{center}
\caption{Projections of an $\alpha$-slim set $E$ on the complex line through two of its points.}
\end{figure}

\begin{proposition}\label{projection_line}
Let $E\subset \S^3$ be $\alpha$-slim for some $0 \leqslant \alpha < \dfrac{\pi}{2}$ containing at least two distinct points $a$ and $b$. Then, in the line $(ab)$ equipped by the previous coordinate, both our coordinates are included in positive cones of angle $2\alpha$:
\begin{itemize}
	\item $\Pi^*_{a,b}$ is included in $\left\{ z \in \C, \quad |\arg(z)|\leq \alpha\right\}$.
	\item $\Pi_{a,b}$ is included in $\left\{ z \in \C, \quad |\arg(-z)|\leq \alpha\right\}$.
\end{itemize}
\end{proposition}

\subsubsection{Projection on a tangent line at $e\in E$}\label{sec:Pi_e}

This last geometric interpretation uses the line map and in fact only relies on the hyperconvexity. Indeed, let $E$ be a hyperconvex subset of $\S^3$ and $e\in E$. We can project $E$ to the line $e^\perp$ via the map $\Pi_e$ defined on $\S^3$ by $\Pi_e(p) = p \boxtimes e$ and $\Pi_e(e) = e$. Geometrically, for $p\neq e$, $\Pi_e(p)$ is the intersection point of the two lines $e^\perp$ and $p^\perp$.
We have:
\begin{proposition}\label{prop:Pi_e}
The map $\Pi_e$ is surjective.

Moreover, suppose that $E$ is hyperconvex. Then the map $\Pi_e$ restricted to $E$ is injective.
\end{proposition}

\begin{proof}
The preimage of any $x\in e^\perp$ by $\Pi_e$ is the intersection between the polar line to $x$ and $\S^3$. This preimage is non empty for any point $x\in e^\perp$.

More precisely the preimage of $e\in e^\perp$ in $\S^3$ is the singleton $\{e\}$. On the other hand, the preimage of $x = \Pi_e(q) \neq e$ is exactly the $\C$-circle through $e$ and $q$. When $E$ is hyperconvex, this $\C$-circle can intersect $E$ at most twice. And we already know two intersection points: $e$ and $q$.
This proves that $\Pi_e: E \to e^\perp$ is injective.
\end{proof}

The projections $\Pi_e(E)$ play the role of space-like geodesics on $\mathcal{L}(E,E)$. In the case where $E = \partial_\infty\HdR$, they are exactly these geodesics. 

\begin{remark}\label{rmk:projection-tangent-line}
Assume that $e=\infty$. Then if a point $m\in E$ has Heisenberg coordinates $[z,t]$, it is easy to see that $\Pi_\infty(m)$ lifts to the vector $[2\bar z,1,0]^T$. In particular, the slimness of $E$ cannot be deduced from its projection on $\mathcal{L}(e,e)$ for $e\in E$, as the coordinate $t$ disappears. In the case where $e=\infty$ the map $\Pi_e$ corresponds, up to a factor $2$ and complex conjugation to the vertical projection onto 
$\C$, and the fact that it is one-to-one says that two points of $E$ cannot be vertically aligned if $E$ contains $\infty$.
 \end{remark}

\subsection{Examples and non-examples\label{section:examples}}

We describe here  examples of slim or non-slim subsets. We also introduce limit sets of surface groups, on which we will further focus in the next sections.

\subsubsection{$\R$-circles and bent $\R$-circles}\label{section:bent}

As recalled in \cref{prop:Cartan-properties}, three points are on a common $\R$-circle iff their Cartan invariant is $0$. As such, any $\R$-circle is $0$-slim. One can go a step further: those are maximal slim subsets:
\begin{proposition}\label{prop:Rcirclemaximal}
Let $E$ be a slim subset of $\S^3$ containing an $\R$-circle $R$. Then $E$ equals $R$.
\end{proposition}

\begin{proof}
  The set of arcs of $\C$-circles connecting two points of $R$ defines a foliation of the complement of $R$ in $\S^3$, as stated in \cref{coro:foliation}. 
  
  Therefore, if $E$ contains a point outside $R$, this point belongs to (exactly) one of these arcs. This gives three points in $E$ that lie on a $\C$-circle, thus have Cartan invariant equal to $\pm \pi/2$. So if $E$ strictly contains $R$, $E$ is not slim.
\end{proof}

A simple example of a slim set which is not an $\R$-circle is given by the union of two half $\R$-circles through $2$ points, see the beginning of \cref{subsec:horizontality-orbits-1}. In coordinates, we may write:
\begin{proposition}\label{prop:bent-Rcircles-slim}
For all $0<\theta < 2\pi$, the union 
\[ E_\theta = \{[x,0], r\in \R_+\} \cup \{[ye^{i\theta},0], r\in \R_+\}\]
is $\alpha$-slim for $\alpha = \left| \dfrac{\pi-\theta}{2}\right|$.
\end{proposition}

\begin{proof}
Let $R_1$ and $R_2$ be the two sets appearing in the union $E_\theta$. We want to compute the maximum of $\A(p,q,r)$ for $(p,q,r)$ in $E_\theta^{(3)}$. First, if they all belong
 to $R_1$ or all to $R_2$, as they are halves of $\R$-circles, this Cartan invariant is $0$. So we may assume, up to permutations, that $p,q$ are in $R_1$ and $r$ in $R_2$. Denote by $0$ the point $[0,0]$.
 
 Using the last point of \cref{prop:Cartan-properties} and the fact that $p,q$ and $0$ belong to $R_1$, we have the equality:
 \[
 \A(p,q,r) = \A(p,q,0) - \A(p,r,0) + \A(q,r,0) = \A(q,r,0)-\A(p,r,0)
 \]
 We first estimate $\A(q,r,0)$. Write $q=[x,0]$ and $r=[ye^{i\theta},0]$ where $x$ and $y$ are positive. We can write $x+iy = \rho e^{it}$ with $0<t<\pi/2$. Then, a direct computation in Heisenberg coordinates gives:
 \[\A(q,r,0) = \arg(-<\bq,\bm r>) = \arg(1-\sin(2t)e^{-i\theta}).\] 
 Note that $0< \sin(2t)\leq 1$. It is easily seen that $0<\A(q,r,0)\leq \frac{\pi-\theta}{2}$ with the maximum attained at $t=\pi/4$ or equivalently $x=y$.
 
 It implies that the difference $\A(p,q,r) =  \A(q,r,0)-\A(p,r,0)$ (where $p$ and $q$ are in $R_1$ and $r$ in $R_2$) is bounded between $\frac{\theta-\pi}{2}$ and $\frac{\pi-\theta}{2}$, proving the proposition.
 \end{proof}

\subsubsection{Slim circles are unknotted}

We will from now on be especially interested in slim subsets homeomorphic to circles. We give a straightforward name to these sets:
\begin{definition}\label{def:slim_circle}
A subset $E \subset \S^3$ is a slim circle if it is both homeomorphic to the circle and a slim subset of $\S^3$. 
\end{definition}

We remark here that slim circles are unknotted. This rules out a non-trivial knot in the sphere being slim.
We prove it by constructing a diagram of the knot without self-intersection.
\begin{proposition}\label{unknotted}
A slim circle is unknotted.
\end{proposition}
\begin{proof}
$E$ is a knot in $\S^3$. Let $e$ be a point in $E$, and consider the projection $\Pi_e$ defined in \cref{sec:Pi_e}. Then $\Pi_e(E)$ is a subset of the line $e^{perp}$ and as such is a diagram of the knot. By slimness and \cref{prop:Pi_e}, $\Pi_e$ is injective on $E$, so $\Pi_e(E)$ has no double points.

The image of $E$ under $\pi$ is a diagram of the knot with no double point: E is the trivial knot.
\end{proof}

The following corollary is proven in the same way.
\begin{corollary}
Suppose $E$ is a slim subset of $\S^3$. Let $F\subset E$ be a subset homeomorphic to the disjoint union of circles.
Then $F$ is an unknotted link.
\end{corollary}
For example, if $E$ is an immersion of a circle with several double points, $E$ contains disjoint circles. They cannot be knotted nor linked.

\subsubsection{Slim orbits of $1$-parameter subgroups}

\cref{subsec:horizontality-orbits-1} describes some very specific families of horizontal orbits of $1$-parameter subgroups. The previous \cref{prop:alphaslim_horizontal} implies that among all $1$-parameter orbits, those are the only one that can be slim. We will not go into the whole classification of slimness for these orbits and especially not look at all at the elliptic case. We will prove that horizontal orbits of $1$-parameter loxodromic subgroups are indeed slim. This gives concrete examples that are not $\R$-circles. On the contrary, horizontal orbits of $1$-parameter parabolic subgroups are not slim unless the group is horizontal unipotent. In this last case, the orbit is an $\R$-circle.

Let us first look at the parabolic case.
We indeed prove that invariance by a single parabolic transformation is compatible with slimness only if this transformation is horizontal unipotent.

\begin{proposition}\label{prop:invariant-parabolics}
Let $E$ be a closed slim subset of $\S^3$ wih at least two distinct points that is invariant under the action of parabolic element $u$ of $\PU(2,1)$. Then $u$ is horizontal unipotent.
\end{proposition}

\begin{proof}
We prove it by a case disjonction. Note that in any case, $E$ contains the fixed point $p$ of $u$ and another point $q\in E$.
\begin{itemize}
	\item If $u$ is vertical unipotent, then the (infinite) orbit $u^n(q)$ completely lies inside the $\C$-circle through $p$ and $q$. So $\A(p,q,u(q)) = \pm\dfrac{\pi}{2}$, which prevents the slimness of $E$.
	\item If $u$ is ellipo-parabolic, then the orbit $u^n(q)$ is contained in a cylinder foliated by $\C$-circles. Let $L$ be the compact set of lines in $\CPstar$ supported by this $\C$-circles. Note that $L$ does not contain $p^\perp$. We can extract a subsequence $q_j = u^{n_j}(q)$ such that the line $(q_jq_{j+1})$ converges to one of the line in $L$. Then, the sequence $(q_j)$ of points in $E$ converges to $p$ and the line $(q_jq_{j+1})$ does not converges to $p^\perp$. \cref{lem:Cartan-coalesce} proves that the supremum of Cartan invariants is $\dfrac{\pi}{2}$. 
\end{itemize}
The only remaining case is that $u$ is horizontal parabolic. It is of course possible, as an $\R$-circle is invariant under some horizontal parabolic elements.
\end{proof}

A direct corollary reads:
\begin{corollary}
A horizontal orbit of a $1$-parameter parabolic subgroup is slim if and only if the subgroup is horizontal unipotent. In this case, it is an $\R$-circle.
\end{corollary}

The proof of slimness in the loxodromic case, however, is more involved than for horizontality, as a marker of how much constrained is a slim set. Moreover, we are not able to estimate the parameter of slimness, leaving us with an indirect proof. We will just give a detailed sketch of the proof and spare some technicalities. Recall from \cref{eq:1-param-loxo} that the one-parameter loxodromic subgroups can be parametrized by:
\begin{equation*}
L_\alpha : s\longmapsto \begin{bmatrix} e^{s\alpha} & & \\ & e^{s(\overline{\alpha}-\alpha)} & \\ & & e^{-s\overline{\alpha}} \end{bmatrix},\mbox{where $\alpha\in\C^*$ satisfies $|\alpha|\neq 1$.}
\end{equation*}
Their horizontal orbits are described in \cref{subsec:horizontality-orbits-1}. We now prove:
\begin{proposition}
If an orbit of $L_s$ is horizontal, then it is slim.
\end{proposition}

\begin{proof}
Recall from \cref{subsec:horizontality-orbits-1} that the orbit $p_s = L_s\cdot p$ is horizontal if and only if the Heisenberg coordinates $[z,t]$ of $p$ satisfy  condition \eqref{eq:1-param-loxo-horizontal}, that is
$$t=3|z|^2\dfrac{\Im(\alpha)}{\Re(\alpha)}.$$
Denote by $\mathcal{P}_\alpha$ the paraboloid defined by the above condition. It is a simple computation in Heisenberg coordinates to verify that a $\C$-circle is either contained in $P_\alpha$ or intersects $P_\alpha$ in at most two points. Moreover, $\C$-circles contained in $P_\alpha$ are all contained in horizontal planes $t={\rm cst}$. Since the orbits of $L_s$ are never contained in such planes, this implies that the orbit we consider never intersects a $\C$-circle thrice, i.e. they are hyperconvex. Therefore no triple of points in the orbit has Cartan invariant equal to $\pi/2$. This means in particular that if the orbit were not slim, then the supremum of the Cartan invariant (which would be equal to $\pi/2$) would not be attained.

Up to a reparametrization of the $1$-parameter subgroup, and conjugating $\alpha$ if necessary, we may assume that $\alpha=1+ia$ for some $a>0$. Applying an element of the normaliser of $L_s$, we may moreover assume that the $z$-coordinate of $p$ is equal to $1$. The horizontality condition becomes 
then $p=[1,3a]$. Denote by $p_-=[0,0]= \lim_{s\to -\infty} p_s$. Taking lifts, we have

\begin{equation}
\bp_-=\begin{bmatrix}0\\0\\1\end{bmatrix}, \quad \bp = \begin{bmatrix} -1+3ia\\1\\1\end{bmatrix}\quad \mbox{and}\quad\bp_s = \begin{bmatrix} e^s(-1+3ia)\\e^{-3isa}\\e^{-s}\end{bmatrix}. 
\end{equation}

We obtain then directly

\[
\A(p_-,p,p_s) = \arg\left(e^s(1+3ia) + e^{-s}(1-3ia) -2e^{3isa}\right).
\]

Denote the latter quantity by: 
\[\mathcal P(s) =  \arg\left(e^s(1+3ia) + e^{-s}(1-3ia) -e^{-isa}\right).\]
When $s\to 0$, $\mathcal P(s)$ goes to $0$. When $s\to \pm \infty$, then $\mathcal P(s)$ goes to $\pm\arctan(3a)$. As the real part of the complex number inside the argument does not vanish, $|\mathcal P|<\dfrac{\pi}2$. So $|\mathcal P|$ admits a maximum $m_a<\dfrac{\pi}2$. Note moreover that $\mathcal P(-s) = -\mathcal P(s)$.

Recall from \eqref{eq:Cartan-cocycle} the cocycle property for the Cartan invariant, with $p,q,r,s$ four points in $\S^3$:
\[\A(p,q,r) + \A(p,q,s) + \A(p,r,s)+\A(q,r,s)=0.\]
Let $s<x<t$ be three reals. We want to estimate the Cartan invariant $\A(p_{s},p_{x},p_{t})$. By invariance of the Cartan invariant under the action of $L_{-x}$, one may assume that $x=0$, so that $p_0 = p$.
Using the cocycle equality to introduce $p_-$, we deduce that 
\[\A(p_{s},p,p_{t}) = -\mathcal P(-s)-\mathcal P(t-s)-\mathcal P(t).\]

To prove that the supremum of the Cartan invariant is strictly less than $\dfrac{\pi}{2}$, we have to control the boundary behaviour. So we are left to understand what happens when $s$ goes to $0$ or $-\infty$ and/or $t$ to $0$ or $+\infty$. In each case, one proves that at the limit the absolute value of the Cartan invariant remains bounded above by $m_a$.
\end{proof}

\subsubsection{Deformations of $\R$-Fuchsian limit sets}\label{sec:deformation-R-fuchsian}

Limit sets $\Lambda_\rho$ of convex cocompact representations $\rho$ of surface groups $\Gamma$ give rise topological circles in the sphere. A natural question is whether $\Lambda_\rho$ is slim. Clearly, if $\Lambda_\rho$ is $\alpha$-slim, with $\alpha\neq 0$, then $\rho(\Gamma)$ is discrete. A first observation is the following:

\begin{proposition}\label{pro:alphattained}
Let $\Lambda_\rho$ be the limit set of a convex cocompact representation $\rho$ of the fundamental group $\Gamma$ of a compact hyperbolic surface. Then the supremum of the Cartan invariant $\A(\Lambda_\rho)$ on the limit set $\Lambda_\rho$ is attained at a triple of distinct points.
\end{proposition}

\begin{proof}
There exists a boundary map $B_\rho: \partial_\infty\Gamma \xrightarrow{\sim} \Lambda_\rho$. 
For any $\rho$, denote by $A_\rho$ the map defined on the set of triples of distinct points $\partial_\infty \Gamma^{(3)}$
by 
\[A_\rho(p,q,r) = \A(B_\rho(p),B_\rho(q),B_\rho(r)).\]
As $B_\rho$ is a bijection between $\partial_\infty \Gamma$ and $\Lambda_\rho$, we deduce that:
\[\A(\Lambda_\rho) = \sup_{(p,q,r) \in \partial_\infty \Gamma^{(3)}} A_\rho(p,q,r).
\]
As moreover $B_\rho$ is $\Gamma$-equivariant,  we have for all $\gamma \in \Gamma$ and for all $(p,q,r)\in \partial_\infty \Gamma^{(3)}$, the equality $A_\rho(p,q,r) = A_\rho(\gamma\cdot p, \gamma\cdot q, \gamma\cdot r)$.
The action of $\Gamma$ on the set of triples of distinct points $(\partial_\infty\Gamma)^{(3)}$ is cocompact: the quotient $X$ may be identified with the unit tangent bundle to the surface \cite{Bowditch-hyperbolic}.
 
To sum up, the map $A_\rho$ descends to a continuous map defined on the compact set $\Gamma \backslash \partial_\infty \Gamma^{(3)}$ and the supremun is attained.
\end{proof}

As a corollary, if the limit set of a convex cocompact representation is hyperconvex then it is slim: the supremum is attained and cannot be $\dfrac{\pi}{2}$ by hyperconvexity.  
The following proposition gives a rich family of examples of slim circles that are not $\R$-circles.
It is proven by Pozzetti-Sambarino-Wienhard \cite[Proposition 6.2]{PozzettiSambarinoWienhard}. They work with the notion of $(1,1,2)$-hyperconvex representations. But note, by the previous proposition,  that  the limit set $\Lambda$ of a convex-cocompact representation $\rho$ of $\Gamma$ is slim if and only if the representation is $(1,1,2)$-hyperconvex Anosov: if $x,y,z$ are in the limit set, then the projective complex line generated by $x,y$ does not contain the point $z$.

\begin{proposition}\label{proposition:deformation-R-fuchsian}[Pozzetti-Sambarino-Wienhardt]
Let $\Gamma$ be the fundamental group of a compact hyperbolic surface and let $\rho_0:\Gamma\longrightarrow PO(2,1)\subset PU(2,1)$ be a representation of $\Gamma$. Then for any sufficiently small deformation $\rho$ of $\rho_0$, the limit set of $\rho(\Gamma)$ is a slim circle.
\end{proposition}

We nevertheless give a proof of this proposition as it is important for our work. In fact, we prove the following proposition, which implies the previous one:
\begin{proposition}\label{pro:limit-set-continuous}
If $\Gamma$ is the fundamental group of a compact hyperbolic surface, then the sup of the Cartan invariant $\A(\Lambda_\rho)$ on the limit set $\Lambda_\rho$ of a convex cocompact representation $\rho$ varies continuously with $\rho$ in the set of convex-cocompact representations of $\Gamma$.
\end{proposition}

\begin{proof}
It is known that the limit set $\Lambda_\rho$ of $\rho(\Gamma)$ varies continuously when $\rho$ varies inside convex-cocompact representations \cite[Lemma 5.5.4 and Remark 5.5.7]{Bourdon}. We consider again the boundary map $B_\rho: \partial_\infty\Gamma \xrightarrow{\sim} \Lambda_\rho$ which varies continuously with $\rho$. 

The map $\rho\mapsto A_\rho$ is continuous, as $\rho\to B_\rho$ is continuous.
The max on the compact set $\Gamma \backslash \partial_\infty \Gamma^{(3)}$ of the continuous function $A_\rho$ depends continuously on the function. As the function varies continuously, the dependance in $\rho$ of $\A(\Lambda_\rho) = \max_{\Gamma \backslash \partial_\infty \Gamma^{(3)}} A_\rho$ is continuous.
\end{proof} 

This gives  examples of slim circles  having low regularity. We will come back on these examples in \cref{sec:Crowns}.

\subsubsection{Deformation of the Farey triangulation and (non)-slimness}\label{subsec:Farey}

 We explore here another family $(\Lambda_\alpha)$, for $-\pi/2\leq\alpha\leq {\pi}/{2}$, of limit sets of surface groups. In this geometrically finite setting, we can see that the supremum of Cartan invariant of $\Lambda_\alpha$  is always ${\pi}/{2}$ except for $\alpha = 0$ where it vanishes. This proves that, in general, semicontinuity of $E\mapsto\A(E)$ is the best we can hope. It also makes clear that in the previous section, the cocompactness assumption was crucial.

Let us describe the explicit construction that appeared in the works on $\PU(2,1)$-representations of the modular group $\mathrm{PSL}(2,\Z)$ by Falbel-Koseleff \cite{Falbel-Koseleff}, Gusevskii-Parker \cite{Gusevskii-Parker} and Falbel-Parker \cite{FJP} (see also \cite[Section 8]{Pierre-Survey} for a survey). 

Let $\Gamma$ be the group $(\Z_2)^{*3}=\la \iota_1,\iota_2,\iota_3| \iota_k^2=1\ra$. We fix $(T_\alpha)_{\alpha\in[-\pi/2,\pi/2]}$ a continuous family of ideal triangles such that $\A(T_\alpha)=\alpha$. We denote by $p_1, p_2, p_3\in \S^3$ the ideal vertices of $T_\alpha$, and $q_k$ the projection of $p_k$ onto the geodesic 
$(p_{k+1}p_{k-1})$ (indices are taken mod. $3$). Consider the half-turns $R_k$ about $q_k$. The $R_k$'s are conjugate to the transformation given by $(z_1,z_2)\longmapsto (-z_1,-z_2)$ in ball model coordinates.  Then one defines a representation $\rho_\alpha : \Gamma\longrightarrow {\PU(2,1)}$ by setting

\[\rho_A(\iota_k)=R_k, \quad k=1,2,3.\]

This gives rise to a continuous $1$-parameter familly of representations of $\Gamma$ in $\PU(2,1)$. The groups $\rho_0(\Gamma)$ and $\rho_{\pm \pi/2}(\Gamma)$ are respectively $\R$ and $\C$-Fuchsian : they are discrete, isomorphic to $\Gamma$ and preserve totally geodesic copies of the Poincaré disc, that are real if $\alpha=0$ and complex if $\alpha=\pm\pi/2$. The orbits of the geodesics connecting the $p_k$'s generate the Farey tesselation in the $\R$ or $\C$-Fuchsian case.
The striking result of the aforementionned works is 
\begin{theorem*}[Falbel-Koseleff, Gusevskii-Parker]
 For any value of $\alpha$, the representation $\rho_\alpha$ is discrete and faithful. Morever, the type of elements remains the same all along the deformation. 
\end{theorem*}

The limit set $\Lambda_\alpha$ of $\rho_\alpha(\Gamma)$ is a $\C$-circle when $\alpha=\pm \pi/2$, an $\R$-circle if $\alpha=0$ and a circle when $0<|\alpha|<\pi/2$. It is not slim unless $\alpha=0$:
\begin{proposition}
We have $\A(\Lambda_\alpha)=\dfrac{\pi}2$ unless $\alpha=0$, in which case $\A(\Lambda_\alpha)=0$.
\end{proposition}

\begin{proof}
The group $\rho_\alpha(\Gamma)$ contains a primitive class of parabolic elements, unique up to conjugation in $\Gamma$, which is the one of $R_1R_2R_3$. It follows from the above works that this parabolic element is screw-parabolic for any value $\alpha\neq 0$, and $2$-step unipotent if $\alpha=0$. By Proposition \ref{prop:invariant-parabolics}, this proves that $\Lambda_\alpha$ is not slim unless $\alpha=0$, in which case it is an $\R$-circle, so $0$-slim.
\end{proof}

\section{Deforming the foliation by arcs of $\C$-circles}\label{section:deform-foliation}

We now come back to the foliation described in \cref{coro:foliation}. Recall that it expresses that the complement of an $\R$-circle $R$ is foliated by arcs of $\C$-circles with endpoints on $R$. We study in this section how this picture deforms when $R$ is deformed among slim curves.
One important tool to understand this is to realize any slim circle as the boundary of a Möbius band in $\HuuC$.

\subsection{The foliation and the unit tangent bundle over $\HdR$}\label{sec:foliation-unit-bundle}

Before actually deforming slim circles, we explain another way to look at the foliation described in \cref{subs:foliation}, that will be more adapted to the study of deformations. Along the way, we will come back to the natural isomorphism between the foliation and the unit tangent bundle $\unit{\HdR}$ over $\HdR$. This last point will be useful for studying limit sets of surface groups.

\subsubsection{Reinterpretation of the foliation property.}

Consider a subset $E$ in the sphere $\S^3$. We first define a notation for its complement and the subset of the sphere sweeped by arcs of $\C$-circles with endpoints in $E$:
\begin{definition}\label{def:M(E)}
For any subset $E$ in the sphere, we define the sets
\[
\Omega_E  = \S^3\setminus E\quad \textrm{ and }\quad
M_E = \{(x,y,p)\in \left(\S^3\right)^3 \textrm{ such that } x\neq y \in E,\, p\in \arc xy\}.
\]
Moreover, let $\F_E: M_E\to \S^3$ be the forgetful map $(x,y,p)\mapsto p$.
\end{definition}
When the context makes things clear, we may drop the notation of the dependence in $E$, considering the sets $\Omega$, $M$ and the map $\F$. \cref{coro:foliation} may be rephrased as the following equivalent statement:
\begin{corollary}\label{coro:foliation-2}
If $R$ is an $\R$-circle, the map $\F_R$ realizes an homeomorphism $M_R\xrightarrow{\sim} \Omega_R$.
\end{corollary}
The previous corollary splits in fact in three substatements, that we will study for slim circles:
\begin{itemize}
\item The map $\F_R$ takes values inside $\Omega_R$,
\item it is actually surjective on $\Omega_R$,
\item and it is injective.
\end{itemize}

The first point generalizes readily in the context of slim subsets:
\begin{lemma}
Let $E$ be a slim subset of $\S^3$. Then the map $\F_E: M_E\to \S^3$ takes values inside $\Omega_E$.
\end{lemma}

\begin{proof}
Let $(x,y,p)$ be an element of $M_E$. Then $p$ is a point of the $\C$-circle through $x$ and $y$, distinct from $x$ and $y$.
This $\C$-circle meets $E$ at $x$ and $y$. As $E$ is slim, it cannot meet $E$ also in $p$. So $p$ belongs to the complement $\Omega_E$ of $E$.
\end{proof}

The goal of this whole \cref{section:deform-foliation}  is to understand what happens with the last two points when $E$ is a slim circle more general that an $\R$-circle. We will prove the following theorem:
\begin{theorem}\label{thm:foliation-deformation}
Let $E$ be a slim circle and consider the map $\F_E: M_E \to \Omega_E$. We have:
\begin{itemize}
\item {\bf [Surjectivity]} If there is a continuous family of slim circles $(E_t)$, for $0\leq t\leq 1$, with $E=E_1$ and $E_0$ a $\R$-circle, then $\F_E$ is surjective.
\item {\bf [Non-Injectivity]} If $E$ is invariant by a non real loxodromic transformation, then $\F_E$ is not injective.
\end{itemize}
\end{theorem}
We will prove the two parts of this theorem independently and with arguments of very distinct flavour. The surjectivity property will be proven in \cref{sec:surjectivity}, see \cref{thm:surjectivity}. The non-injectivity statement, instead, will be proven in  \cref{sec:non-injectivity}. This theorem raises in particular the following question: is $\F_E$ always surjective onto $\Omega_E$ for any slim
circle $E$?

Before going 
further, we continue to review the case of \cref{coro:foliation} and its link with the unit tangent bundle $\unit{\HdR}$. It will in particular help understand better the set $M_E$.

\subsubsection{Back to the unit tangent bundle}

The relevance of the set $M_E$ is made clearer when we see how closely it is related to the unit tangent bundle $\unit{\HdR}$ over $\HdR$. Recall that the latter is homeomorphic to the set of triples $(x,y,z)$ of distinct points in $\bHdR$ that are cyclically positively oriented: a triple $(x,y,z)$ is associated to the unit tangent vector to the geodesic from $x$ to $y$ in $\HdR$ with base point the orthogonal projection of $z$ on this geodesic.

Let $E$ be a slim circle. Using a parametrization $\phi: \bHdR \xrightarrow{\sim} E$, we now define a natural map $\unit{\HdR} \to M_E$. A straightforward geometric lemma will prove useful:
\begin{lemma}
Let $x,y$ and $z$ be distinct points in $\S^3$. Then, there exists a unique point $p\in \arc xy $ such that the projections of $p$ and $z$ on the (real) geodesic $xy$ coincide.
\end{lemma}

This lemma is used in the following construction:
\begin{definition}\label{def:Phi_E_phi}
Let $E$ be a slim circle in $\S^3$ and $\phi : \bHdR \to E$ an homeomorphism. Then we define the map $\Phi_{E,\phi} : \unit{\HdR} \to M_E$ by sending a point $(x,y,z)$ to the point $(\phi(x),\phi(y),p)$ in $M_E$ where $p \in \arc xy$ is the unique point whose projection on the real geodesic $\phi(x),\phi(y)$ coincides with the one of $\phi(z)$.
\end{definition} 

As before, we will often denote this map simply by $\Phi$. This map is natural: if $\phi$ is equivariant for a representation $\rho$ of a group $\Gamma \subset \mathrm{Isom}(\HdR)$ to $\PU(2,1)$, then so is $\Phi$.

In the case where $R$ is an $\R$-circle, $R$ is the boundary of a real hyperbolic plane. So one can parametrize it by a map $\phi: \bHdR \to R$ which is $\mathrm{Isom}(\HdR)$-equivariant. Denote for simplicity $M = M_R$, $\Omega = \S^3\setminus R$, $\F = \F_R$ the forgetful map and $\Phi = \Phi_{R,\phi}$ the map we just defined. The following proposition then expresses the foliation described in Section \ref{subs:foliation} using this map $\Phi$.
\begin{proposition}\label{proposition:homeoUT}
Let $R$ be an $\R$-circle. Then, the map $\Phi : \unit{\HdR} \to M_R$ is a $\mathrm{Isom}(\HdR)$-equivariant homeomorphism. 

Moreover the composition $\F_R \circ \Phi$ induces a $\mathrm{Isom}(\HdR)$-equivariant homeomorphism $\unit{\HdR} \simeq \Omega_R$. This homeomorphism sends orbits of the geodesic flow to arcs $\arc ab$, $a\neq b \in R$.
\end{proposition}

Note that if now one consider an $\R$-Fuchsian representation $\rho$ of a surface group $\Gamma \subset \mathrm{Isom}(\HdR)$, then the map $\F\circ \Phi$ descends into a CR-spherical uniformization $U(\Gamma \backslash \HdR) \simeq \rho(\Gamma)\backslash \Omega$.

With the foliation fully reinterpreted, we can move on and see how each of its aspects vary when deforming the $\R$-circle into a slim circle $E$. The first tool is the construction of a surface $\RP(E)$ in $\CP^2$, homeomorphic to $\RP^2$ and that extends $E$ outside the sphere.

\subsection{Extensions of slim circles}\label{sec:extension}

In this section we will usually identify a line $\mathcal L(x,y)$ between two points in $\S^3$  to its 
polar point in $\CP^2$. Indeed, any point $p$ outside the ball in $\RP^2$ is orthogonal to a unique geodesic inside $\HdR$, whose endpoints $(x,y)$ verify $x\boxtimes y= p$. Moreover, any point $p$ in $\S^3$ equals $\mathcal L(p,p)$.

Consider the simplest example of a slim circle, i.e. an $\R$-circle $R$. By definition, it is the intersection of a copy of a real projective plane $\RP^2\subset \CP^2$ with the sphere $\S^3$. The part outside of the ball $\HdC$ is the projective plane minus an open disc: it is a Möbius band. Moreover, we have a natural parametrization of this closed Möbius band by $(R\times R)/(x,y)\sim (y,x)$, given by the map $(x,y) \to \mathcal L(x,y)$ from $R\times R$ to $\CP^2$.

The goal of this section is to extend this construction to
any slim circle $E$. We will define an
extension $\mathrm{RP}(E)$ to $\CP^2$ similar to the case of $\R$-circles. Our construction of the extension outside
the ball $\HdC$ is canonical, whereas inside we make some arbitrary choices. We will mainly focus on what happens
outside later on, so this will not be a problem. Recall that a subset of the sphere is hyperconvex if no three points are contained in a $\C$-circle, see \cref{definition:hyperconvex}.

\begin{proposition}\label{prop:extension_outside}
Let $E$ be a horizontal and hyperconvex  circle. Then the map $\mathcal L$ from $E\times E$ to $\CP^2$ defines an embedding of the Möbius band $E\times E/(x,y)\sim (y,x)$ into $\HuuC$ whose intersection with the sphere $\S^3$ is $E$.
\end{proposition}

\begin{proof}
As $E$ is slim, it is horizontal. This means that $\mathcal L$ is continuous. Moreover, for any $x\neq y$ in $E$, we have $\mathcal L(x,y) = x\boxtimes y$, which is a point outside the closed ball, whereas $\mathcal L(x,x)=x \in E$. So $\mathcal L$ descends into a continuous immersion of $E\times E/(x,y)\sim (y,x)$ into $\CP^2 \setminus \HdC$ whose intersection with the sphere $\S^3$ is $E$.

As $E\times E/(x,y)\sim (y,x)$ is compact, the last point to prove is the injectivity of our map. We have to check that $\mathcal L(x,y) = \mathcal L(x',y')$ if and only if 
$(x,y)=(x',y')$ or $(x,y) = (y',x')$. So let $p=\mathcal L(x,y)= \mathcal L(x',y')$. First, if $p\in E$, then we have $p=x=y=x'=y'$ which proves what we want. Assume now that  $p\not\in E$. This means $x\neq y$, $x'\neq y'$ and $p=x\boxtimes y =x'\boxtimes y'$. This translates into the fact that $x,y,x'$ and $y'$ all lie in the $\C$-circle $p^\perp \cap \S^3$. As $E$ is hyperconvex, it intersects at most twice this $\C$-circle. This means, as wanted, that the sets $\{x,y\}$ and $\{x',y'\}$ are equal.
\end{proof}

We also want to extend $E$ inside. We do not have a nice construction as above, so will arbitrarily choose a good enough extension. From now on, we choose an arbitrary origin $o$ in $\HdC$. We denote by $D(E)$ the union of all (real) geodesics from $o$ to a point $x \in E$. As two distinct geodesics from $o$ cannot meet again in $\HdC$, the set $D(E)$ is a disk inside the ball $\HdC$, whose closure meet the sphere $\S^3$ exactly at $E$.

\begin{definition}
For a slim circle $E$, we denote by $\mathrm{RP}(E)$ the union of $D(E)$ and $\mathcal L(E,E)$.
\end{definition}

The Möbius band $\mathcal L(E,E)$ is invariant by any $\PU(2,1)$-transformation leaving $E$ invariant, by construction. Moreover we will see that it varies continuously with $E$ varying among slim circles. But the disc $D(E)$ does not enjoy the first property. This raises the question:
Is it possible to construct a natural disc $D(E)$ bounded by $E$, i.e. such that it is invariant by any $\PU(2,1)$-transformation leaving $E$ invariant and it varies continuously with $E$?

From the previous discussion, we deduce that $\mathrm{RP}(E)$ is topologically a projective plane $\RP^2$:

\begin{proposition}
For any slim circle $E$, the set $\mathrm{RP}(E)$ is homeomorphic to $\RP^2$.
\end{proposition}

\begin{proof}
The disc $D(E)$ is a disc whose boundary is $E$. $\mathcal L(E,E)$ is a Möbius band whose boundary is also $E$. Their union is thus homeomorphic to the gluing of a disc and a Möbius band along their boundary: it is a real projective plane.
\end{proof}

Now, we want to understand how these surfaces $\mathrm{RP}(E)$ deform when deforming $E$ inside the set of slim circles.

\subsection{Deformations of slim subsets and surjectivity}\label{sec:surjectivity}

We investigate in this section $\mathcal C^0$-deformations of horizontal sets. However, it is easily seen that 
$\mathcal C^0$-deformations do not preserve horizontality. For example, fix a loxodromic one-parameter subgroup $A$, and take a continuous family $p_t$ of points in the sphere such that only $p_0$ belongs the surface singled out by \cref{pro:paraboloids-1-parameter}. Then the familiy of sets $(\overline{A\cdot p_t})_t$ is $\mathcal C^0$-continuous and only $\overline{A\cdot p_0}$ is horizontal.

But the additional quantitative information given by slimness guarantees that deformations remain horizontal and the projective lines given by the line map $\mathcal L$ vary continuously.
\begin{proposition}\label{slim-deformations}
Let $E \subset \S^3$ be an horizontal subset and $\epsilon>0$. For $t\in (-\epsilon,\epsilon)$, let $E_t = f_t(E)$ be a continuous deformation. We assume that there is $0\leq \alpha <\dfrac{\pi}{2}$ such that all the sets $E_t$ are $\alpha$-slim.

Then the sets $E_t$ are all horizontal and the map $(t,p,q) \mapsto \mathcal L(f_t(p),f_t(q))$ is continuous on $(-\epsilon,\epsilon)\times E^2$.
\end{proposition}

\begin{proof}
The horizontality of $E_t$ is granted by the assumption that they are slim. What we really want to control is the second point.

We argue by contradiction. Suppose there exists converging sequences $p_n\to p$, $q_n\to q$ in $E$ and $t_n\to t$ in 
$(-\epsilon, \epsilon)$ such that the sequence of lines 
$l_n = \mathcal L(f_{t_n}(p_n),f_{t_n}(q_n))$
 does not converge to $l =\mathcal L(f_t(p),f_t(q))$. 
 
 We first note that this cannot happen if $p\neq q$, by 
 the assumption that $E_t$ is a continuous deformation: we have $f_{t_n}(p_n) \to f_t(p)$ and $f_{t_n}(q_n) \to f_t(q)$. $f_t$ is still an homeomorphism, so $f_t(p)\neq f_t(q)$ and the lines $l_n$ converge to $l$ by continuity of $\mathcal L$ outside the diagonal.
 
 Assume now $p=q$. Fix a sequence $r_n$ of points in $E$ such that $f_{t_n}(r_n) \to r\neq p$. Suppose by contradiction that $l_n \to l_\infty \neq l=\mathcal L(p,p)$. Then, by \cref{lem:Cartan-coalesce}, $\A(f_{t_n}(p_n),f_{t_n}(q_n),f_{t_n}(r_n))$ goes to $\pm \pi/2$. It is impossible, as all the $f_s$ are $\alpha$-slim. This concludes the proof.
\end{proof}

A corollary is that the surfaces $\mathrm{RP}(E_t)$ vary continuously.
\begin{corollary}\label{cor:continous_RP}
Under the hypothesis of the previous proposition, the map $t,p \mapsto f_t(p)$ can be extended into a homotopy $t,p \mapsto F_t(p)$ between $\mathrm{RP}(E)$ and $\mathrm{RP}(E_s)$.
\end{corollary}

\begin{proof}
Fix $t\in (-\epsilon,\epsilon)$ and let us define the map $F_t: \mathrm{RP}(E) \to \mathrm{RP}(E_t)$. Let $p$ be a point in $\mathrm{RP}(E)$. We shall consider three cases:
\begin{enumerate}
	\item if $p$ is in $E$, then we set $F_t(p) = f_t(p)$.
	\item if $p$ is in $D(E)$, then it is on a geodesic $ox$ for a unique $x\in E$, at distance $d\geq 0$ of $o$. We set $F_t(p)$ to be the point at distance $d$ of $o$ in the geodesic $o f_t(x)$.
	\item if $p$ is outside the ball, there are two points $x,y \in E$ such that $p = \mathcal L(x,y)$. Note that the pair ${x,y}$ is unique: they are the only $2$ intersections between $E$ and the $\C$-circle polar to $p$. We define $F_t(p)$ to be the point $\mathcal L(f_t(x), f_t(y))$.
\end{enumerate}
From the previous section and proposition, we see that $t,p \to F_t(p)$ is continuous, that $F_0$ is the identity map and that for each $t$, $F_t$ realizes a homeomorphism from $E$ to $E_t$.
\end{proof}

This corollary is the crucial point to prove that for slim deformations $E$ of an $\R$-circle, the map $\F_E$ is still surjective or, equivalently, they still are maximal slim subset of $\S^3$. The following proposition rephrases the Surjectivity item of \cref{thm:foliation-deformation}:
\begin{proposition}\label{thm:surjectivity}
Fix $0<\alpha<\dfrac{\pi}{2}$. For $t\in [0,1]$, let $\phi_t: \S^1 \to \S^3$ be a continuous deformation such that,  for each $t$, the set $E_t = \phi_t(\S^1)$ is $\alpha$-slim and moreover $E_0$ is a $\R$-circle.

Then, for all $t$, the map $\F_{E_t} : M_{E_t}\to \Omega_{E_t}$ is surjective. Equivalently, the set $E_t$ is a maximal slim circle of $\S^3$: any slim set containing $E_t$ is $E_t$ itself.
\end{proposition}

The proof uses the construction of the surfaces $\mathrm{RP}(E_t)$. More precisely, we will use that such surfaces intersect any (complex) line in $\CP^2$, as shown in the following lemma.
\begin{lemma}
Under the assumption of the theorem,
for any line $l\subset \CP^2$, and any $0\leq t \leq 1$, the intersection between $l$ and $\mathrm{RP}(E_t)$ is non-empty.
\end{lemma}

\begin{proof}
Any complex line $l$ meets the usual $\RP^2$ at any point inside $l\cap \bar l$. The intersection is moreover transverse unless $l$ is a real line.

Now, we work in the homology group $H_2(\CP^2, \frac{\Z}{2\Z})$, with its intersection form denoted by $i$. The previous remark translates in this setting into the property $i([l],[\RP^2]) = 1\in\frac{\Z}{2\Z}$. Note that working with $\frac{\Z}{2\Z}$-coefficients avoids problems related to the non-orientability of $\RP^2$.

The previous \cref{cor:continous_RP} proves that, for any $t$, the surface $\mathrm{RP}(E_t)$ is a continuous deformation of $\mathrm{RP}(E_0)$. By assumption, $E_0$ is an $\R$-circle so that $\mathrm{RP}(E_0)$ is a copy of $\RP^2$. So we obtain $[\mathrm{RP}(E_t)] = [\RP^2]$. This in turn translates into the intersection property $i([l],[\mathrm{RP}(E_t)]) = 1$.

We conclude that any line $l$ intersects $\mathrm{RP}_t$. Moreover, if the intersection is transverse, it is an odd number of points.
\end{proof}

We can now conclude the proof of the theorem.
\begin{proof}
Fix $t\in [0,1]$, and $p\in \S^3 \setminus E_t$. We want to prove that $p$ belongs to an arc $\arc xy$, with $x\neq y \in E_t$. This implies that $p$ belongs to the image of $\F_{E_t}$.

Consider the line $p^\perp$ in $\CP^2$. By the previous lemma, it intersects $\mathrm{RP}(E_t)$. The intersection of $p^\perp$ with $\HdC$ is empty, whereas its intersection with the sphere is reduced to $\{p\}$. Note that, by assumption, $p$ is not in $E_t = \mathrm{RP}(E_t)\cap \S^3$.
So the intersection point is not in the disc $D(E_t)$. This means by construction that $p$ is a point $x\boxtimes y$  for some $x\neq y \in E_t$. This implies that one of the arcs $\arc xy$ or $\arc yx$ contains $p$.

We have just proven that any point $p\not \in E_t$ belongs to an arc. This proves the surjectivity of $\F_{E_t}$. 
\end{proof}

\begin{remark}
The proof of \cref{thm:surjectivity} is valid as soon as $E_0$ verifies that $i([RP(E_0)],[l])=1$.
\end{remark}

Before moving on and proving the last point of \cref{thm:foliation-deformation}, we exhibit in the next subsection a simple example where the foliation does indeed deform as a new foliation.

\subsection{A one-parameter deformation of the foliation by arcs of $\C$-circles}
Let us come back to the example we have studied in \cref{section:bent}. For any angle $0<\theta<2\pi$, we consider the subset of $\S^3$ defined in Heisenberg coordinates by
$$E_{\theta} = \left\{[x,0],\,x\in \R_+\right\}\cup \left\{[y e^{i\theta},0],\,y\in \R_+\right\}.$$ 
Note that when $\theta=\pi$, the curve $E_\theta$ is the boundary of $\HdR$ in $\S^3$.

\begin{theorem}\label{prop:deform-standard-foliation}
For any $\theta\in [\pi/2,3\pi/2]$, the set of arcs of $\C$-circles with endpoints in $E_\theta$ defines a foliation of $\S^3\setminus E_{\theta}$.
\end{theorem}

\begin{proof}
To prove Proposition \ref{prop:deform-standard-foliation}, we need to
\begin{enumerate}
\item prove that any point $p\in\S^3$ outside $E_\theta$ belongs to some arc of $\C$-circle hitting $E_\theta$ twice;
\item prove that any two arcs of $\C$-circle both hitting $E_\theta$ twice never meet unless they share at least one endpoint in $E_\theta$.
\end{enumerate}

The first point follows directly from \cref{thm:surjectivity} and \cref{prop:bent-Rcircles-slim} : $E_\theta$ is $|\pi-\theta|/2$-slim, and it is obtained from an $\R$-circle by a homotopy which is given by the bending.

To prove the second point, we use a numerical criterion to determine when two $\C$-circles are disjoint. From \cref{lem:Ccircles-meet}, we know that, given four points $a,b,c,d$ in $\S^3$ such that $a\neq b$ and $c\neq d$, the $\C$-circles $C_{ab}$ and $C_{cd}$ spanned by $(a,b)$ and  $(c,d)$ respectively are disjoint if and only if 
\begin{equation}\label{eq:condi-C-cercles}\la (\ba\boxtimes \bb)\boxtimes (\bc\boxtimes \bd), (\ba\boxtimes \bb)\boxtimes (\bc\boxtimes \bd)\ra\neq 0\end{equation}

So what we need to do is to take $a,b,c,d$ as above in $E_{\theta}$ and prove the the left-hand side of \eqref{eq:condi-C-cercles} doesn't vanish when $|\theta-\pi|\leqslant \pi/2$ unless one of $a=c$, $a=d$, $b=c$, $b=d$ happens. We denote by 
$\Delta_1 = \left\{[x,0],\,x\in \R_+\right\}$ and $\Delta_2=\left\{[y e^{i\theta},0],\,y\in \R_+\right\}$ the two half-lines whose union is $E_\theta$.

Considering the possible relative positions of the two arcs of $\C$-circles, we are left with the following four cases.
\begin{multicols}{2}
\begin{enumerate}
\item The four endpoints all belong to one of the $\Delta_i$'s,
\item Three of the endpoints lie in $\Delta_1$ and one in $\Delta_2$,
\item One of the two arcs has its endpoints in $\Delta_1$ and the other in $\Delta_2$,
\item Both arcs have one endpoint in $\Delta_1$ and one in $\Delta_2$.
\end{enumerate}
\end{multicols}

The first case follows direcly from \cref{coro:foliation}.  We will not give details for each of the other three cases, but let us consider the fourth one, which is the most intricate.
Assume that $a$ and $c$ are in $\Delta_1$ and that $b$ and $d$ are in $\Delta_2$. We may then chose lifts as in \eqref{eq:lift-Siegel} so that there exists four non-negative real numbers $x$, $y$, $z$, $t$, such that ($x\neq 0$ or $z\neq 0$) and ($y\neq 0$ or $t\neq 0$), for which we have:

\begin{equation}\label{lifts-Delta1-Delta2}
\ba = \begin{bmatrix} -x^2\\x\\1\end{bmatrix},\,
\bb = \begin{bmatrix} -y^2\\ye^{i\theta}\\1\end{bmatrix},\,
\bc = \begin{bmatrix} -z^2\\z\\1\end{bmatrix},\,
\bd = \begin{bmatrix} -t^2\\te^{i\theta}\\1\end{bmatrix},\,
\end{equation}

Plugging these values into the left-hand side of \eqref{eq:condi-C-cercles} and reorganising, the condition becomes

\begin{equation}\label{eq:condi-2}
0\neq  (x-z)(t-y)\Bigl(-\alpha \cos^2\theta +\beta \cos\theta -\gamma\Bigr)
\end{equation}
where
\begin{eqnarray}
\alpha & = & 16xyzt(x+z)(y+t) \nonumber\\
\beta & = & 4 ((xt+yz)^2+(ty+xz)^2+2(tx + yz)(xy + tz))(ty + xz)\nonumber\\
\gamma & = &  4\Bigl(t^2 + 2ty + z^2)ty^2z+(t^2 + 2xz + z^2)tx^2z \nonumber\\
& & \quad\quad+ (x^2 + 2ty + y^2)t^2xy+ (x^2 + y^2 + 2xz)xyz^2\Bigr)\nonumber
\end{eqnarray}

The conditions on $x$, $y$, $z$ and $t$ impose that $\alpha\geqslant 0$, $\beta\geqslant 0$ and $\gamma>0$. Since $\cos\theta \leqslant 0$  the right-hand side of \eqref{eq:condi-2} can only vanish if $x=z$ or $t=y$, which is the expected result in this case.

The remaining two cases are treated in the same way, only simpler since the analogue of  \eqref{eq:condi-2} has degree 1 in $\cos\theta$.

\end{proof}

\subsection{Obstruction to the foliation property and rigidity}\label{sec:non-injectivity}

We prove in this section the second part of \cref{thm:foliation-deformation}, namely the non-injectivity. It shows that the deformation of the previous subsection is very specific and we can not hope to deform the foliation in general.
First of all, we introduce a bit of vocabulary for loxodromic transformations. Let $\gamma \in \PU(2,1)$ be loxodromic. Let $\lambda = re^{i\theta}$ be its eigenvalue of greatest modulus, with $\theta \in (-\pi,\pi]$ and $r>1$. $\lambda$ is well-defined only up to multiplication by $e^{\frac{2i\pi}{3}}$, so we normalize by choosing $-\pi/3<\theta\leq \pi/3$. Up to conjugation, we choose the representative of $\gamma$ in $\SU(2,1)$ to be the diagonal matrix with diagonal entries $(\lambda, \bar \lambda/\lambda, 1/\bar \lambda)$. The trace of this matrix is real if and only if $\lambda$ is real. Equivalently, the trace of $\gamma^3$ is well-defined for $\gamma \in \PU(2,1)$ and it is real iff the trace of the choosen lift is real.
\begin{definition}
A loxodromic element $\gamma$ is \emph{non-real} if its trace is non-real.

Its \emph{rotation factor} is the angle $3\theta \in (-\pi,\pi]$ where $\theta$ is the normalized argument of its eigenvalue of greatest modulus.
\end{definition}
Note that $3\theta=\pi$ corresponds to a real loxodromic element even if its rotation factor is not $0$.
With this definition, we can state the following slightly more general version of the non-injectivity property. Note in particular that the slimness assumption is not fully needed and the hyperconvexity property is enough:
\begin{theorem}\label{thm:non-injectivity}
Let $E$ be a hyperconvex circle which is invariant under the action of a non-real loxodromic map
$\gamma\in \PU(2,1)$ with fixed repulsive and attractive fixed points $p^-,p^+\in E$.
Then there exists an infinite family of $\C$-circles $C_n$ such that

\begin{enumerate}
\item For all $n$, $C_n$ meets $E$ twice.
\item For all $n$, $C_n$ meets the $\C$-circle through $p^-$ and $p^+$ outside $\{p^-,p^+\}$.
\end{enumerate}
\end{theorem}

Let us make a series of reductions before actually proving the theorem. First, note that $E\setminus\{p^-,p^+\}$ has two connected components. So, up to passing to the action of $\gamma^2$, we may assume that $\gamma$ preserves each of this component. Note that if $\gamma$ is non-real, $\gamma^2$ has a rotation factor different from $0$. Moreover, up to conjugation, one can suppose that $\gamma$ is diagonal with fixed points $0$ and $\infty$ in Heisenberg model. The action of $\gamma$ then forces $E$ to spiral around $0$. This can be stated, in Heisenberg coordinates:
\begin{lemma}\label{lem:arg-unbounded}
Let $\gamma$ be a loxodromic map fixing $0$ and $\infty$ with rotation factor $\beta = 3\theta\neq 0$. Let $c: [0,+\infty]$ be a path  $c(s)= [z(s),t(s)]$ that is an homeomorphism on its image. Assume that this image that is $\gamma$-invariant and hyperconvex, and moreover $c(0)=0$ and $c(+\infty)=\infty$. 

Then any continuous lift $s\mapsto \tilde{\arg}(z(s))$ of the argument of $z$ over $0<s<+\infty$ is onto $\R$ and proper. 
\end{lemma}

\begin{proof}
Since $c$ is hyperconvex, the vertical projection $s\mapsto z(s)=\frac 12\overline {\Pi_\infty(c(s))}$ is injective, see \cref{rmk:projection-tangent-line}. As $c(s)$ is never $0$ for $0<s<+\infty$, the quantity $z(s)$ never vanishes, and the lift of argument is well-defined once choosen the lift of the argument of $c(1)$.

Applying the loxodromic element in coordinates, we compute that $\gamma\cdot c(s) = [re^{-3i\theta}z(s),r^2 t(s)] = c(s')$. The lift $\tilde{\arg}(z(s'))$ of the argument of $z(s')$ is of the form $\tilde{\arg}(z(s))-3\theta+2k\pi$ for some $k\in \Z$. Note that, by our assumption on the rotation factor, we have $-3\theta + 2k\pi\neq 0$. So the image of $\tilde{\arg}(z(s))$ is invariant by translation by $-3\theta + 2k\pi\neq 0$. It proves it is onto $\R$ and proper.
\end{proof}

\begin{proof}[Proof of Theorem \ref{thm:non-injectivity}]
Following the discussions above, we can assume that $p^-=0$, $p^+=\infty$ with their usual lifts, and $\gamma$ preserves both connected components of $E\setminus\{p^-,p^+\}$. The proof below shows the the Theorem is true under the weaker assumption of horizontality.

Let $a$ be close to $p^-$ and $b$ to $p^+$. Using the local parametrization given by \cref{lem:local}, we can write for lifts:
\[ \ba = \bp^- + x \bp^-\boxtimes \bp^+ + o(x) \textrm{ and } \bb = \bp^+  +y \bp^-\boxtimes \bp^+ + o(y)\]
These coordinates are directly linked to the Heisenberg coordinates: we have $\ba = [z(a),t(a)]$ and $\bb =[z(b), t(b)]$ where: 
\[z(a)=-\frac{x}{2}+o(x) \quad \mathrm{ and }\quad\frac{2}{\overline z(b)}=y+o(y).\]
In particular, \cref{lem:arg-unbounded} implies that the arguments mod. $2\pi$ of $x$ and $y$ oscillates infinitely as $a$ goes to $p^-$ and $b$ to $p^+$.

By \cref{lem:Ccircles-meet}, the fact that the $\C$-circle through $a,b$ intersects the one through $p^-,p^+$ is equivalent to the equality $(\ba\boxtimes \bb)\boxtimes(\bp^-\boxtimes \bp^+)=0$. We can compute directly, once noticed that $\bp^-\boxtimes(\bp^-\boxtimes \bp^+) = \frac 12 \bp^-$ and $(\bp^-\boxtimes \bp^+)\boxtimes \bp^+ = \frac 12 \bp^+$. Indeed, we have 
\[\ba\boxtimes \bb = \bp^-\boxtimes \bp^+ + \frac {x\bp^+ + y\bp^-}{2} +o(\sqrt{|x|^2+|y|^2}).\]
Computing the box-product with $\bp^-\boxtimes \bp^+$ leads to:
\[(\ba\boxtimes \bb)\boxtimes(\bp^-\boxtimes \bp^+) = \frac{y\bp^- - x\bp^+}{2} + +o(\sqrt{|x|^2+|y|^2}).\]
One can now compute the square of the Hermitian norm of this last vector, getting:
\[
\la (\ba\boxtimes \bb)\boxtimes(\bp^-\boxtimes \bp^+),(\ba\boxtimes \bb)\boxtimes(\bp^-\boxtimes \bp^+)\ra = -\Re(x\bar y) +o(x^2+y^2)
\]

Since the arguments of $x$ and $y$ oscillates, this last value has an infinite number of change of signs as $a$ goes to $p^-$ and $b$ to $p^+$. Hence, it vanishes infinitely many times: this gives an infinite number of $\C$-circles hitting the one through $p^-,p^+$. 
\end{proof}

This concludes the proof of \cref{thm:foliation-deformation}.  It is the main ingredient in the following rigidity theorem. Let us recall that a representation $\rho$ is a convex-cocompact and slim deformation of a $\R$-fuchsian representation $\rho_0$ if there is a path of convex-cocompact and slim representations linking $\rho$ to $\rho_0$.

\begin{theorem}\label{thm:crown-uniformisations}
Let $\Sigma$ be a closed hyperbolic surface. 
Denote by $\Gamma$ the fundamental group of $\Sigma$.
Consider $\rho_0 : \Gamma \to \PO(2,1)\subset\PU(2,1)$ an $\R$-fuchsian representation.

Then, for any convex-cocompact and slim deformation $\rho$ of $\rho_0$, we have: 
\begin{enumerate}
	\item The limit set $\Lambda_\rho$ is a maximal slim circle.
	\item The arcs $\arc{x}{y}$, for $x\neq y \in \Lambda_\rho$, are pairwise disjoint if and only if $\rho$ is $\R$-fuchsian.
\end{enumerate}
\end{theorem}

\begin{proof}
The first item follows from \cref{thm:surjectivity}: the map $\mathcal F$ is surjective on $\Omega_\rho$, which means that any point in $\Omega_\rho:=\S^3\setminus \Lambda_\rho$ belongs to an arc with endpoints in $\Lambda_\rho$. In particular, any superset of $\Lambda_\rho$ has $3$ points on a $\C$-circle, so is not slim. 

The second point is a corollary of \cref{thm:non-injectivity}. Indeed, from \cite{Acosta-real}  the fact that all loxodromic elements in the $\rho(\Gamma)$ have real trace implies that  $\rho(\Gamma)$ preserves a totally geodesic real plane and therefore $\rho$ is $\R$-fuchsian. So, let $\rho$ be a non-$\R$-fuchsian deformation of $\rho_0$. Then some element $\rho(\gamma)$ is a non-real loxodromic transformation. The limit set $\Lambda_\rho$ is invariant under this element. Then \cref{thm:non-injectivity} implies that some arcs $\arc xy$ intersect. The other implication is direct.
\end{proof}

\section{Crown-type spherical CR uniformisation of 3-manifolds}\label{sec:Crowns}

We now look at the geometric meaning of slimness and the deformed foliation in the equivariant case, i.e. assuming that the slim circle is the limit set of a convex-cocompact group. We will see that it gives CR-spherical uniformizations on unit tangent bundles and drilled unit tangent bundles.

So let $\Gamma$ be a lattice in $\PO(2,1)$ and denote by $\rho_0$ the $\R$-fuchsian representation given by $\Gamma\subset \PO(2,1) \subset \PU(2,1)$. Let $\Lambda_0$ be the $\R$-circles which is the limit set of $\rho_0(\Gamma)$ and $\Omega_0$ its complement. We have seen in \cref{sec:deformation-R-fuchsian} that we have a natural identification $\rho_0(\Gamma)\backslash \Omega_0\simeq \unit{(\Gamma\backslash\bHdR)}$. This is a CR-uniformization of the unit tangent bundle to the surface.

Now, let us deform $\rho_0$ into a convex cocompact representation $\rho$. Denote by $\Lambda_\rho$ the limit set of $\rho(\Gamma)$ and by $\Omega_\rho$ its complement.  In fact, $\rho(\Gamma)\backslash \Omega_\rho$ is still homeomorphic to $\unit{(\Gamma\backslash\bHdR)}$ (see \cref{proposition:deformation-quotient}). We want to use arcs of $\C$-circles to construct natural CR-uniformization on the unit tangent bundle drilled along a geodesic. For that, we need to define  supersets of the limit set, that we call \emph{crowns}.

\subsection{Crowns}

Let $\gamma\in\PO(2,1)$ be a loxodromic element. The axis of $\gamma$ in $\HdR$ is naturally oriented. We call again axis of $\gamma$ and denote by ${\rm axis}(\gamma)$ the oriented lift of the $\HdR$-axis of $\gamma$ to the unit tangent bundle $\unit{\HdR}$. The goal of this section is to explore the analogy between this notion of axis in $\unit{\HdR}$ and arcs of $\C$-circles in $\S^3$. This analogy has already been noticed in the discussion following \cref{coro:foliation}, and in \cref{proposition:homeoUT}. 

For any loxodromic element $\delta\in\PU(2,1)$, we denote by $\alpha_\delta$ the arc of $\C$-circle $\arc{a^-}{a^+}$, where
$a^+$ and $a^-$ are the attractive and repulsive fixed points of $\delta$. Note that $\alpha_\delta$ is naturally oriented toward $a^+$.  We call $\alpha_\delta$ the {\it axis at infinity} of $\delta$.

\begin{definition}
Let $\Delta$ be a convex-cocompact subgroup of $\PU(2,1)$ whose limit set $\Lambda_\Delta$ is a topological circle, and $\delta\in\Delta$ be a loxodromic element.  We call the {\it crown associated to $\delta$} the subset of $\S^3$ defined as
\begin{equation} {\rm Crown}_{\Delta,\delta} = \Lambda_\Delta \cup\Bigl(\bigcup_{g\in\Delta} g\cdot\alpha_\delta\Bigr).
\end{equation}
We denote by $\Omega_{\Delta,\delta}\subset \Omega_\Delta$ the complement of ${\rm Crown}_{\Delta,\delta}$ in $\S^3$.
\end{definition}

Note that by construction ${\rm Crown}_{\Delta,\delta}$ is closed and $\Delta$-invariant,
and $\Omega_{\Delta,\delta}$ is open and $\Delta$-invariant. These two objects only depend on the  
$\Delta$-conjugacy class of $\delta$. Moreover, the action of $\Delta$ on $\Omega_{\Delta,\delta}$ is properly 
discontinuous. Eventually, the stabilizer in $\Delta$ of $\alpha_\delta$ is the cyclic group generated by $\delta$, so that the union may be rewritten:  
\[\bigcup_{g\in\Delta} g\cdot\alpha_\delta = \bigcup_{[g]\in \Delta/<\delta>} [g]\cdot\alpha_\delta.\]

\begin{definition}
We say that ${\rm Crown}_{\Delta,\delta}$ is {\it embedded} whenever the arcs of $\C$-circles $g\cdot\alpha_\delta$ are pairwise disjoint.
\end{definition}

For a cocompact $\R$-fuchsian group $\Gamma\subset \PO(2,1)\subset\PU(2,1)$, the situation is clear:
\begin{proposition}\label{proposition:R-Fuchsian-crown-embedded}
Let $\Gamma \subset \PO(2,1)\subset\PU(2,1)$ be a cocompact $\R$-Fuchsian group with limit set $\Lambda=\bHdR\subset\S^3$. 
Denote by $\Sigma$ the surface $\Gamma\backslash\HdR$.
Then, for any   loxodromic element $\gamma\in\Gamma$, we have
\begin{enumerate}
\item  ${\rm Crown}_{\Gamma,\gamma}$ is embedded.
\item The quotient $\Gamma\backslash \Omega_{\Gamma,\gamma}$ is homeomorphic to the 3-manifold obtained by drilling out from the unit tangent bundle $\unit{\Sigma}$ the orbit of the geodesic flow corresponding to $\alpha$.
\end{enumerate}
\end{proposition}

\begin{proof}
The first item follows directly from \cref{coro:foliation}. The second item follows from \cref{proposition:homeoUT}. Using the notation therein, the map $\mathcal{F}_\Lambda\circ\Phi$ restricts as a $\Gamma$-equivariant homeomorphism from the complement in $\unit{\HdR}$ of the union of all axes of elements conjugate to $\gamma$ in $\Gamma$ to $\Omega_{\Gamma,\gamma}$. The result is obtained by taking quotient.
\end{proof}

\begin{figure}
\begin{center}
 \begin{tabular}{cc}
 \scalebox{0.55}{\includegraphics{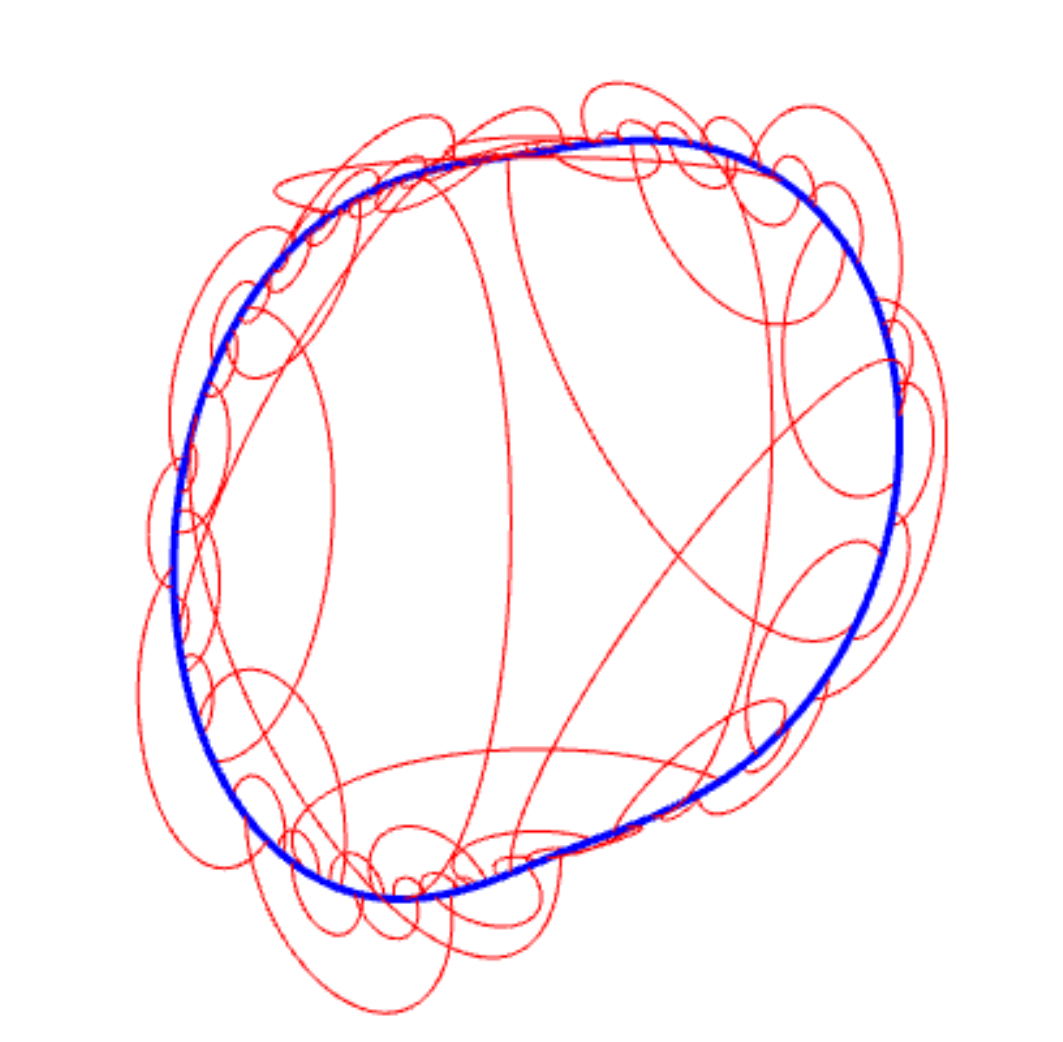}}
   & 
\scalebox{0.55}{\includegraphics{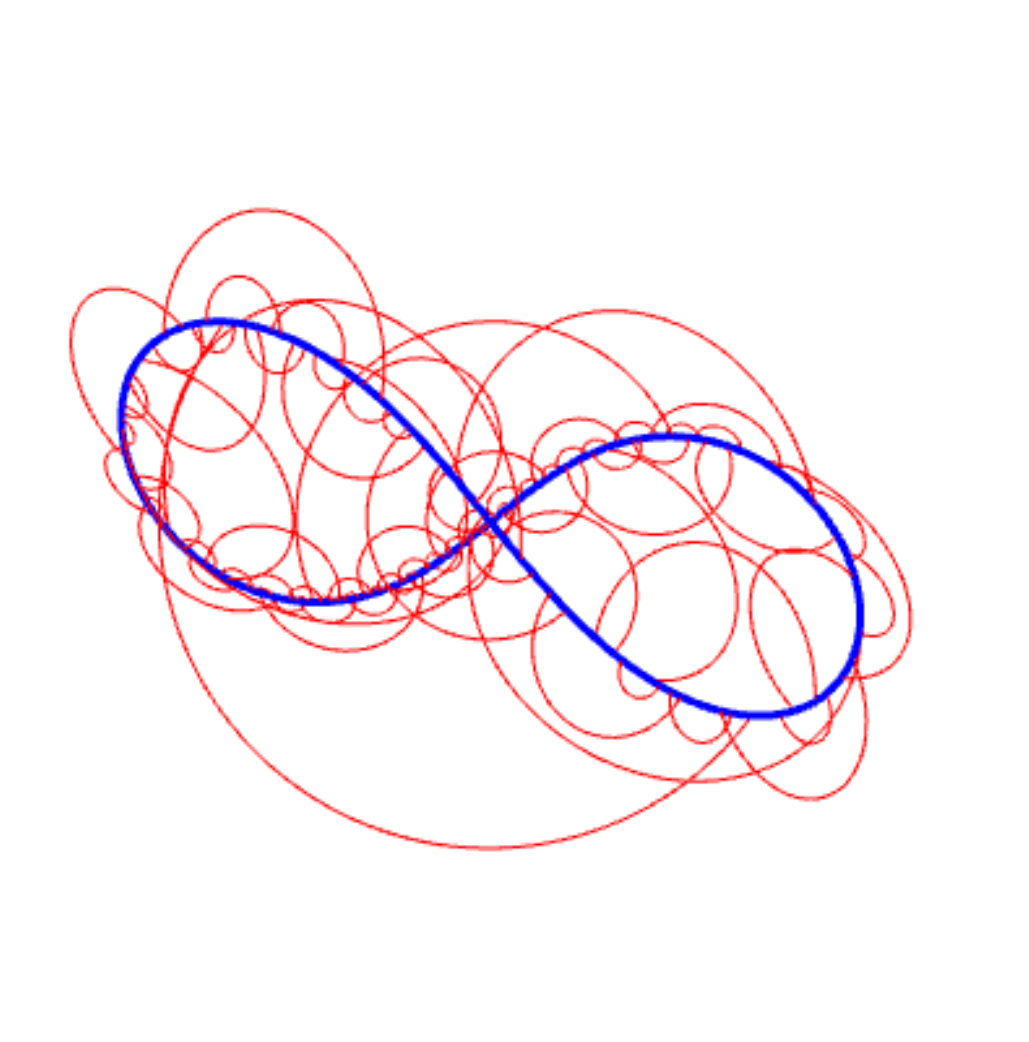}}\\
   \end{tabular}
\caption{Two views of an approximation of the crown $C_{\Gamma,\gamma}$, where $\Gamma$ is the $\R$-Fuchsian $(3,3,4)$-triangle group, and $\gamma$ is the word $\iota_3\iota_2\iota_1\iota_2$ (see \Cref{ex:triangle-groups}). Here the blue curve is the $\R$-circle which is the limit set of the group $\Gamma$. The red curves form the orbit of the axis at infinity of $\gamma$.\label{fig:Crowns}}
\end{center}   
\end{figure}

Let $\Sigma$ be a closed hyperbolic surface and $\lambda$ be a closed oriented geodesic of $\Sigma$. We say that $\lambda$ is {\it filling} whenever the complement of $\lambda$ in $\Sigma$ is a union of topological discs. We denote by $\unit{\Sigma(\lambda)}$ the unit tangent bundle drilled out along the natural lift of the oriented geodesic $\lambda$. A direct corollary of \cref{proposition:R-Fuchsian-crown-embedded} reads:

\begin{corollary}\label{coro:drilled}
For any hyperbolic surface $\Sigma$ and any closed oriented geodesic $\lambda$, the $3$-manifold $\unit{\Sigma(\lambda)}$ admits a CR-spherical uniformisation with a real fuchsian holonomy.
\end{corollary}

\begin{proof}
Let $\Gamma \subset \mathrm{PO}(2,1) \subset\PU(2,1)$ be the fundamental group of $\Sigma$ and $\gamma\in \Gamma$ be a primitive element whose oriented axis is $\lambda$. Then, \cref{proposition:R-Fuchsian-crown-embedded} states that $\unit{\Sigma(\lambda)}$ is homeomorphic to $\Gamma \backslash \Omega_{\Gamma,\gamma}$.
\end{proof}

Note that one can construct a lot of cusped hyperbolic $3$-manifolds in that way: Theorem 1.12 in \cite{FoulonHasselblatt} states that $\unit{\Sigma(\lambda)}$ is  hyperbolic as soon as $\lambda$ is filling (see also Calegari's blog \cite{Blog-Calegari}).

\begin{example}\label{ex:triangle-groups}
Let $2\leq p\leq q \leq r$ be three integers, with $\frac 1p +\frac 1q + \frac 1r < 1$. 
Consider the group
$$\Gamma = \la \iota_1,\iota_2,\iota_3 \vert \iota_k^2 = (\iota_1\iota_2)^p=(\iota_2\iota_3)^q=(\iota_3\iota_1)^r=1\ra.$$ 

It is the $(p,q,r)$-triangle group, which is hyperbolic. It can be seen uniquely - up to conjugacy - as a subgroup of  $\PO(2,1)\subset \PU(2,1)$. Each of the $\iota_k$'s is a complex reflection of order two that fixes pointwise a complex line of $\HdC$ which intersects $\HdR$ along a geodesic $\gamma_k$. Let $\Gamma_2$ be the even subgroup of $\Gamma$, and let $\gamma$ be the geodesic in $\HdR$ which is the axis of a (hyperbolic) element $w\in\Gamma_2$.  By \cref{proposition:R-Fuchsian-crown-embedded}, the quotient of $\S^3\setminus {\rm Crown}_{\Gamma_2,\gamma}$ is homeomorphic to the complement of the axis of $\gamma$ in the unit tangent bundle of the orbisurface $\Gamma_2\backslash\HdR$.

In the special case where $(p,q,r)=(3,3,4)$ and $w$ is the word $\iota_3\iota_2\iota_1\iota_2$, then the resulting $3$-manifold is the figure eight knot complement. This fact is proved in \cite{Dehornoy}.
\end{example}

\subsection{Deformations}

We now prove that, after a small deformation of $\rho_0$, the crown deforms and gives new CR-spherical uniformization of the drilled unit tangent bundle, with non real holonomy. By the analysis of the previous \cref{section:deform-foliation}, the arcs of $\C$-circles in the crown could intersect. We prove that it is not the case, at least locally.

\begin{theorem}\label{pro:deformed_crowns}
Let $\Sigma$ be a hyperbolic surface 
and $\lambda$ be an oriented closed geodesic.
 Denote by $\Gamma$ the fundamental group of $\Sigma$
 and by $\gamma$ a primitive element whose axis lifts $\lambda$. Consider $\rho_0 : \Gamma \to \PO(2,1)\subset\PU(2,1)$ an $\R$-fuchsian representation.

Then there exits a neighborhood $U$ of $\rho_0$ (of convex-cocompact and slim deformations $\rho$ of $\rho_0$) such that for any $\rho$ in $U$, we have:
\begin{enumerate}
	\item $\mathrm{Crown}_{\rho(\Gamma),\rho(\gamma)}$ is embedded and homotopic in $\S^3$ to $\mathrm{Crown}_{\rho_0(\Gamma),\rho_0(\gamma)}$.
	\item The quotient $\rho(\Gamma) \backslash \Omega_{\rho(\Gamma),\rho(\gamma)}$ is homeomorphic to $\unit{\Sigma(\lambda)}$.
\end{enumerate}
\end{theorem}

In order to prove this theorem, we need two different arguments: first that the crowns $\mathrm{Crown}_{\rho(\Gamma),\rho(\gamma)}$ remain embedded along a small deformation of $\rho_0$ and second that the whole quotient $\rho(\Gamma)\backslash \Omega_\rho$ is always homeomorphic to the unit tangent bundle $\unit{\Sigma}$. 

For the first argument, we prove in the following lemma that we have indeed only a finite number of arcs to watch to insure that no intersection happen.
\begin{lemma}\label{lem:finite-C-circles-K}
Let $\rho: \Gamma \to \PU(2,1)$ be a convex cocompact representation with a slim limit set $\Lambda_\rho$. Fix a compact subset $K$ of $\Omega_\rho$.

Then, for all $\gamma\in \Gamma$ with $\gamma$ loxodromic, the following set is finite:
\[
\left\{ [\delta]\in \Gamma/<\gamma> \textrm{ such that } \rho(\delta)\cdot \alpha_{\rho(\gamma)}\cap K\neq \emptyset\right\}.
\]
\end{lemma}

For the proof of this lemma, we use the polarity, that identifies any $\C$-circle with a point in $\HuuC$. Recall from \cref{sec:extension} that the Möbius band $\mathcal L(\Lambda_\rho,\Lambda_\rho)$ in $\mathrm{RP}(\Lambda_\rho)$ is exactly the set of points in $\HuuC$ polar to $\C$-circles that hit $\Lambda_\rho$ twice.

\begin{proof}
Let $H\subset \HuuC$ be the set of polar to $\C$-circles meeting $K$:
\[H:= \left\{p\in \HuuC | (\partial L_p) \cap K \neq \emptyset\right\}.\] The intersection of the closure of $H$ with $\S^3$ is exactly $K$: indeed, it consists of points $p$ in $\S^3$ whose polar line $p^\perp$ meets $K$. But the only point in $\S^3\cap p^\perp$ is $p$ itself, so $p\in K$.

The chosen compact $K$ avoids $\Lambda_\rho$: $\Lambda_\rho \cap K = \emptyset$. Moreover, $\Lambda_\rho$ is the intersection of $\S^3\cap \mathrm{RP}(\Lambda_\rho)$ (see \cref{sec:extension}). We deduce that the intersection beween $H$ and the Möbius band $\mathcal L(\Lambda_\rho,\Lambda_\rho)$ is compact.

Now, the orbit of the geodesic $\mathrm{axis}(\gamma)$ in $\HdR$ is discrete in the space of geodesic of $\HdR$. Equivalently, using polarity in the real case, the orbit $\mathcal O_0$ of polars to these geodesic axis in $\RP^2$ is discrete in the Möbius band $\HuuR= \partial_\infty \Gamma^2/(x,y)\sim(y,x)$. Denote by $p_\gamma$ the polar point to the axis at infinity $\alpha_{\rho(\gamma)}$ of $\rho(\gamma)$. By construction, we have $p_\gamma = \rho(\gamma)_- \boxtimes \rho(\gamma)_+$. For any $\delta\in \Gamma$, the polar to  $\rho(\delta) \cdot \alpha_{\rho(\gamma)}$ is the point $\rho(\delta)\cdot p_\gamma = (\rho(\delta)\cdot\rho(\gamma)_-) \boxtimes (\rho(\delta)\cdot\rho(\gamma)_+)$. One can express this in other, more adapted, terms. Recall from \cref{sec:deformation-R-fuchsian} that we have a boundary map $B_\rho: \partial_\infty \HdR \to \Lambda_\rho$. This boundary map induces an embedding 
$(x,y) \to \mathcal L(B_\rho(x),B_\rho(y))$ of $\HuuR$ into $\HuuC$ whose image is the Möbius band $\mathcal L(\Lambda_\rho,\Lambda_\rho)$, see \cref{prop:extension_outside}.

The above expression of $\rho(\delta)\cdot p_\gamma$ means that the orbit $\mathcal O_\rho$ of $p_\gamma$ in $\mathrm{RP}(\Lambda_\rho)$ is exactly the image of $\mathcal O_0$ by this embedding. It implies that this orbit $\mathcal O_\rho$ is discrete in the Möbius band $\mathcal L(\Lambda_\rho,\Lambda_\rho)$. As $H\cap \mathcal L(\Lambda_\rho,\Lambda_\rho)$ is compact, the intersection between the orbit $\mathcal O_\rho$ and $H$ is a finite number of points.

By polarity, we have proven that the set of axis at infinity in the orbit of $\alpha_{\rho(\gamma)}$ that intersect $K$ is finite.
\end{proof}

The second argument is classical in the world of geometric structure and follows from Ehresmann-Thurston principle. The next proposition is a consequence of a theorem by Guichard and Wienhard \cite[Thm 9.12]{GuichardWienhard}. The language in which \cite{GuichardWienhard} states and proves this theorem is quite different from ours, so we provide a short proof.

\begin{proposition}\label{proposition:deformation-quotient}
 Let $\rho$ be a convex cocompact deformation of $\rho_0$. Then $\rho(\Gamma)\backslash \Omega_\rho$ is homeomorphic to $\unit{\Sigma}$. 
\end{proposition}

\begin{proof}
For any convex-cocompact representation $\rho$, let $X_\rho$ denote the compact quotient $\rho(\Gamma)\backslash \Omega_\rho$. Note that $\Lambda_\rho$ is a circle in the sphere, so the quotient $X_\rho$ is connected. We prove, following Guichard-Wienhard, that the diffeomorphism class of $X_\rho$ is constant under small deformation.

Indeed, we know already that such $\rho$ is the holonomy representation of a spherical CR uniformisation of $X_{\rho}$. Let $\hat X$ be the $\Gamma$-covering of $X_{\rho}$ and fix a compact fundamental domain $D\subset \hat X$. The developing map $s_\rho : \hat X \to \Omega_\rho \subset \S^3$ sends $D$ far away from the compact $\Lambda_\rho$. By the Ehresmann-Thurston principle \cite{Goldman-geometricstructures}, any small (convex-cocompact) deformation $\rho'$ is also a holonomy representation of a spherical CR structure on $X_\rho$. The developing map $s_{\rho'}$ is close to $s_\rho$. As $\Lambda_{\rho'}$ varies continuously (\cref{pro:limit-set-continuous}), $s_{\rho'}(D)$ avoids $\Lambda_{\rho'}$ for small enough deformations. Hence, $s_{\rho'}$ is a local diffeomorphism from $\hat X$ to $\Omega_{\rho'}$. By $\Gamma$-equivariance, it descends to a local diffeomorphism $\phi_{\rho'}$ from $X_\rho$ to $X_{\rho'} = \Omega_{\rho'}/\rho'(\Gamma)$. Both $X_\rho$ and $X_{\rho'}$ are compact and connected, so $\phi_{\rho'}$ is a covering. As $\phi_{\rho'}$ deforms to $\phi_\rho = \mathrm{Id}_{X_{\rho}}$ when $\rho'$ deforms to $\rho$, $\phi_{\rho'}$ is an actual diffeomorphism isotopic to the identiy.

By connexity arguments, for any $\rho$ in the whole connected component of $\rho_0$ in the space of convex cocompact representations, we have that $X_{\rho} =  \rho(\Gamma)\backslash \Omega_\rho$ is diffeomorphic to $\unit{\Sigma} = X_{\rho_0}$.
\end{proof}

With these two preliminary results at hand, we may proceed with the proof of \cref{pro:deformed_crowns}.

\begin{proof}[Proof of \cref{pro:deformed_crowns}]
In order to prove the first point, we have to prove that the arcs of $\C$-circles $\rho(\delta)\cdot \alpha_{\rho(\gamma)}$ are pairwise disjoint for $\rho$ close to $\rho_0$. 

We can choose a compact $K \subset \S^3$ such that for all small enough deformations $\rho$ of $\rho_0$, $K$ avoids $\Lambda_\rho$ and contains a fundamental domain for $\rho(\Gamma)$ acting on $\Omega_\rho$. 
So, if some intersection happens between two arcs $\rho(\delta)\cdot \alpha_{\rho(\gamma)}$ and $\rho(\delta')\cdot \alpha_{\rho(\gamma)}$, one such intersection also happens inside $K$. So we just have to control the behavior of arcs that meet $K$.

If the deformations are small enough, $\Lambda_\rho$ is always slim (\cref{proposition:deformation-R-fuchsian}). It implies that the set of arcs $\rho(\delta)\cdot \alpha_{\rho(\gamma)}$ intersecting $K$ is finite, by the previous \cref{lem:finite-C-circles-K}. Moreover this set is locally constant. So there is an open neighborhood $U$ of $\rho_0$, for which the following set is finite:
\[
\left\{ \delta \in \Gamma, \exists \rho \in U, \rho(\delta)\cdot \alpha_{\rho(\gamma)} \, \mathrm{ intersects }\, K\right\}.
\]

So we have to control a finite set of arcs of $\C$-circles. At $\rho_0$, from  \cref{proposition:R-Fuchsian-crown-embedded}, we know that these finite number of arcs do not meet. As they vary continuously with $\rho$, it remains true in a small neighborhood.

The second point follows: in the quotient $\rho(\Gamma) \backslash \Omega_{\rho(\Gamma)}\simeq \unit{\Sigma}$, the projection of the set of arcs $\delta\cdot \alpha_\gamma$ is a curve, which varies continuously with $\rho$ from the first point. For the $\R$-fuchsian representation $\rho_0$, the axis at infinity $\alpha_\gamma$ identifies with the geodesic $\mathrm{axis}(\gamma)=\lambda$ in $\unit{\Sigma}$, so its projection remains homotopic to $\lambda$ throughout the deformation. As a consequence, the quotient $\rho(\Gamma) \backslash \Omega_{\rho(\Gamma),\rho(\gamma)}$ is homeomorphic to $\unit{\Sigma(\lambda)}$.
\end{proof}

\bibliographystyle{alpha}
\bibliography{biblio}

\begin{flushleft}
  \textsc{E. Falbel\\
  Institut de Math\'ematiques de Jussieu-Paris Rive Gauche \\
CNRS UMR 7586 and INRIA EPI-OURAGAN \\
 Sorbonne Universit\'e, Facult\'e des Sciences \\
4, place Jussieu 75252 Paris Cedex 05, France \\}
 \verb|elisha.falbel@imj-prg.fr|
 \end{flushleft}
 
 \begin{flushleft}
  \textsc{A. Guilloux\\
  Institut de Math\'ematiques de Jussieu-Paris Rive Gauche \\
CNRS UMR 7586 and INRIA EPI-OURAGAN \\
 Sorbonne Universit\'e, Facult\'e des Sciences \\
4, place Jussieu 75252 Paris Cedex 05, France \\
Institut Fourier, UMR 5582, Laboratoire de math\'ematiques\\ 
  Universite Grenoble Alpes, CS 40700, 38058 Grenoble cedex 9, France\\}
 \verb|antonin.guilloux@imj-prg.fr|
 \end{flushleft}
 
\begin{flushleft}
 \textsc{ P. Will\\
  Institut Fourier, UMR 5582, Laboratoire de math\'ematiques\\ 
  Universite Grenoble Alpes, CS 40700, 38058 Grenoble cedex 9, France\\}
  \verb|pierre.will@univ-grenoble-alpes.fr|
\end{flushleft}

\end{document}